\documentclass[10pt,a4paper]{article}

\RequirePackage{my_packages}
\RequirePackage{my_theorems}
\RequirePackage{my_macros}
\RequirePackage{my_tikz}

\usepackage[backend = biber,        
            language = english ,
            style = numeric-comp,   
            giveninits = true,
            isbn = false,
            url = false,
            doi = false,
            sorting = nyt,
            maxnames = 6,
            backref=false
            ]{biblatex}           

\renewbibmacro{in:}{%
  \ifentrytype{article}{}{%
  \printtext{\bibstring{in}\intitlepunct}}}
\usepackage{tikz}
\tikzset{every loop/.style={}}
\usetikzlibrary{arrows,shapes,decorations.markings,automata,backgrounds,petri,bending,calc}

\usepackage{tikz-cd} 
\tikzset{
    labl/.style={anchor=south, rotate=90, inner sep=.5mm}
}

\usepackage[capitalize]{cleveref}
\usepackage{paralist}

\AtEveryBibitem{%
  \clearfield{note}%
  \clearfield{edition}%
  \clearlist{language}
}

\DeclareNameAlias{labelname}{given-family}

\def\bfa{\mathbf{a}}
\def\bfb{\mathbf{b}}
\def\bfc{\mathbf{c}}
\def\bfd{\mathbf{d}}
\def\bft{\mathbf{t}}
\def\bfx{\mathbf{x}}
\def\bfz{\mathbf{z}}

\title{
Pullbacks and intersections\\in categories of graphs of groups
}

\author{
    Jordi Delgado, \small{\url{jorge.delgado@upc.edu}}\\
    \small{Universitat Politècnica de Catalunya - BarcelonaTech (UPC), Barcelona, Spain}\\
    \small{ORCID: \url{https://orcid.org/0000-0002-8365-8929}}
    \and
    Marco Linton, \small{\url{marco.linton@icmat.es}}\\
    \small{Instituto de Ciencias Matemáticas (ICMAT), Madrid, Spain}\\
    \small{ORCID: \url{https://orcid.org/0000-0002-1081-5268}}
    \and
    Jone Lopez de Gamiz Zearra, \small{\url{jone.lopezdegamiz@ehu.eus}}\\
    \small{University of the Basque Country - EHU, Bilbao, Spain}\\
    \small{ORCID: \url{https://orcid.org/0000-0002-3725-9740}}
    \and
    Mallika Roy, \small{\url{mallikaroy75@gmail.com}}\\
    \small{Indian Institute of Science (IISc), Bengaluru, India}\\
    \small{ORCID: \url{https://orcid.org/0000-0002-9730-5980}}
    \and
    Pascal Weil, \small{\url{pascal.weil@cnrs.fr}}\\
    \small{CNRS, Univ. Sorbonne Paris Nord, LIPN, UMR 7030, F-93430 Villetaneuse, France}\\
    \small{CNRS, ReLaX, IRL 2000, Siruseri, India}\\
    \small{ORCID: \url{https://orcid.org/0000-0003-2039-5460}}
}
\date{\today}

\begin{document}
\maketitle

\begin{abstract}
We develop a categorical framework for studying graphs of groups and their morphisms, with emphasis on pullbacks. More precisely, building on classical work by Serre and Bass, we give an explicit construction of the so-called \emph{$\AA$-product} of two morphisms into a graph of groups $\AA$ --- a graph of groups which, within the appropriate categorical setting, captures the intersection of subgroups of the fundamental group of $\AA$. We show that, in the category of pointed graphs of groups, pullbacks always exist and correspond precisely to pointed $\AA$-products. In contrast, pullbacks do not always exist in the category of unpointed graphs of groups. However, when they do exist, and we show that it is the case, in particular, under certain acylindricity conditions, they are again closely related to $\AA$-products. We trace, all along, the parallels with Stallings' classical theory of graph immersions and coverings, in relation to the study of the subgroups of free groups. Our results are useful for studying intersections of subgroups of groups that arise as fundamental groups of graphs of groups. As an example, we carry out an explicit computation of a pullback which exhibits two finitely generated subgroups of a Baumslag--Solitar group intersecting in a non-finitely generated subgroup.
\end{abstract}

\vspace{15pt}


\newpage

\tableofcontents

\section*{Introduction}
\addcontentsline{toc}{section}{Introduction}

Building on the idea that the study of actions on trees can reveal the internal structure of groups, Bass--Serre theory --- initially developed by Serre and Bass \cite{ser80,
bas93} --- extends the covering space intuition for free groups to a much more general framework. Important special cases of the theory are the older constructions of amalgamated free products and HNN-extensions which arise from actions on trees with a single orbit of edges. Just as free groups can be understood as fundamental groups of quotient graphs of the trees on which they act, Bass--Serre theory identifies groups acting on trees as fundamental groups of certain combinatorial objects called \emph{graphs of groups}, for which a generalised covering theory can be developed to describe and analyse subgroup properties \cite{bas93}.

\medskip

In the classical setting, the main insight is that the free group $F_n$ (of rank $n$) can be realised as the fundamental group of a certain graph (a bouquet of $n$ circles). This immediately provides a natural bijection between the subgroups of the free group and the covering spaces of the associated graph, or equivalently, with a well-defined family of pointed, labeled graphs --- called \emph{Stallings' graphs} --- in which the elements of the subgroup correspond to the (reduced) labels of the closed paths at a distinguished vertex.
This correspondence not only transparently reveals the structure of the subgroups (\emph{Nielsen--Schreier theorem}) but also offers a visual and combinatorial method for investigating the lattice of subgroups of $F_n$.

Moreover, under this bijection, immersions of finite graphs correspond exactly to finitely generated subgroups and can be easily computed from any finite generating set 
by repeatedly identifying edges with the same image --- a transformation now known as a \emph{Stallings' folding}. As a consequence, the theory is particularly well-suited for algorithmic questions, and has been immensely fruitful in both producing new results and providing more intuitive proofs for classical ones.
See \parencite{sta83}
for the original
article by Stallings, and
\parencite{
km02,
dv24} for later expositions in a more combinatorial, graph-theoretic language.

\medskip

One well-known application of Stallings' graphs is the study of subgroup intersections in free groups, a topic that originated in the mid-20th century with Howson’s proof of the \emph{finitely generated intersection property} (\fgip): the intersection of two finitely generated subgroups of a free group is itself finitely generated.
Stallings' interpretation of subgroups as graphs provides a very natural description of the intersection of two subgroups $H,K \leqslant F_n$: since the elements of each subgroup are given by labels of closed paths in the corresponding Stallings' graphs $\Gamma_H$ and~$\Gamma_K$, the elements of the intersection $H \cap K$ are given by labels of closed paths readable in both $\Gamma_H$ and $\Gamma_K$; that is, readable in the pullback $\Gamma_H \times \Gamma_K$. 
So, pullbacks essentially encode subgroup intersection; concretely, the Stallings' graph of the intersection $H \cap K$ is precisely the pointed core of the pullback $\Gamma_H \times \Gamma_K$.
Since the pullback of two finite graphs is again finite, Howson's result follows immediately.

\medskip

The idea of realizing groups as fundamental groups --- essentially sets of closed paths with concatenation --- of suitably chosen objects has been extremely useful and it has been extended far beyond the realm of free groups. Families of groups in which Stallings' theory has been successfully extended include
free products~\parencite{iva99},
amalgams of finite groups~\parencite{mar07},
groups acting freely on $\ZZ^n$-trees~\parencite{ns12}, 
virtually free groups~\parencite{ssv16}, quasi-convex subgroups of automatic groups \parencite{kmw17},
relatively quasi-convex
subgroups of finitely presented relatively hyperbolic groups
\parencite{kw20}, free-abelian by free groups~\parencite{del17}, right-angled Coxeter groups and groups acting on 
CAT(0) cube complexes~\parencite{bl18,dl21,bkl22}
and, quite extensively,
non-free actions of groups on graphs
\parencite{sta91,
bf91,
dun97,
dun98,
gui98,
dd99,
bow01,
kwm05,
Ar06%
}. 

\medskip

While much of this work focuses on algorithmic aspects --- typically through the development of suitable versions of Stallings’ folding process --- the interpretation of subgroups as covering spaces (or as fundamental groups of appropriate subobjects) provides deep insight into subgroup structure and behaviour, independent of algorithmic considerations. 
Intersections, for example, admit a natural description along the same lines as in free groups: once the elements of subgroups are seen as closed paths in certain objects, the elements of their intersection correspond precisely to those paths readable as closed paths in both, that is, in their pullback. 
Thus, pullbacks --- defined in a suitable category --- naturally arise as candidates for representing intersections of subgroups.

\medskip

Despite the broad scope of applications of Stallings graphs variations, results regarding intersections remain limited beyond free groups. A significant contribution in this direction was made by Ivanov, who, in a series of papers
\parencite{iva99,
iva01,
iva18},
developed a generalisation of Stallings' graphs (and the corresponding idea of pullbacks) to describe intersections within free products and used it to derive a `geometric proof'
of B.\ Baumslag's generalisation of Howson's theorem for free products \cite{bau66}, 
among other results regarding intersections. 
In a different direction, another variation of Stallings' graphs allowed for a complete description of intersections within free-abelian times free groups \parencite{dv13, dv22}, a well-known family of groups that does not have the \fgip\
Again, the idea is to enrich the original Stallings' graphs to represent subgroups in this family and derive the natural version of pullback in this context. The combination of these two descriptions enabled the study of intersections in larger families, such as the family of Droms right-angled Artin groups \parencite{dvz18}, which can be iteratively described as a combination of the two.

\bigskip

The aim of this work is to extend to graphs of groups the pullback constructions which allowed to describe subgroup intersections in the previous contexts. \emph{Graphs of groups} 
are combinatorial structures that assign a group to each vertex and edge of an underlying graph, along with monomorphisms from each edge group to the groups at its endpoints, capturing how vertex groups are `glued together' along the graph.
If a base vertex of the underlying graph is designated, we talk of \emph{pointed graphs of groups}. Bass \cite{bas93} established a natural bijective correspondence between these objects (and morphisms between them) and groups acting on trees; see Appendix \ref{appendix}.

Natural notions of $\AA$-paths and their (homotopy-type) congruences lead to the definition of the fundamental group of a graph of groups $\AA$ at a vertex $u$, denoted by~$\pi_1(\AA, u)$.
Graphs of groups constitute a far-reaching generalisation of two ubiquitous constructions in group theory, namely HNN extensions ${A\ast_\psi = \langle A, t\mid t^{-1} c t = \psi(c),\ c\in C \rangle}$ and amalgamated free products $A \ast_{C =\psi(C)} B = \langle A, B \mid c = \psi(c),\ c\in C \rangle$, which can be recovered as the fundamental groups of the two kinds of 1-edge  graphs of groups in Figure~\ref{fig: HNN and amalgameted free product}. In the diagram on the left (resp. right), $C$ is a subgroup of $A$ and $\psi$ is a monomorphism from $C$ to $A$ (resp. from $C$ to $B$).
\begin{figure}[H]
\centering
\begin{tikzpicture}[shorten >=3pt, node distance=2cm and 1.75cm, on grid,auto,-latex]
\begin{scope}
\node[state] (vA) {$A$};
    \node[] (Aab) [above = 0.25 of vA]{};
    \node[] (eA) [right = of vA]{$C$};
    \node[] (Aac) [above = 0.25 of eA]{};
    \path[->] (vA) edge[loop right,min distance=19mm,in=335, out=25]
    node[pos=0.85,below] {$\psi$}
    node[pos=0.15,above] {$\id_{C}$}
            (vA);
\end{scope}

\begin{scope}[xshift = 4cm]
    \node[state] (vA) {$A$};
    \node[state] (vB) [right = 3 of vA]{$B$};
    \node (eC) [above right = 0.2 and 1.5 of vA]{$C$};

    \path[->] (vA) edge[]
    node[pos=0.8,above] {$\psi$}
    node[pos=0.2,above] {$\id_{C}$} (vB);
\end{scope}
\end{tikzpicture}
\vspace{-15pt}
\caption{HNN extensions (left) and amalgamated free products (right), as graphs of groups}
\label{fig: HNN and amalgameted free product}
\end{figure}
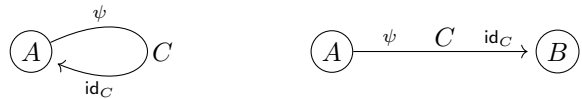

In order to describe pullbacks of graphs of groups, we first give a precise definition of suitable categories whose objects are graphs of groups and whose arrows are classes of morphisms under certain equivalence relations. We study both the category $\GrGp$ of graphs of groups and the category $\GrGp^*$ of connected pointed graphs of groups. In this paper,  we focus on the existence, construction, and structural properties of pullbacks in these categories, leaving to the separate paper \cite{dllrw_fgip} most of the applications to the study of intersections and the important finitely generated intersection property.

\medskip

Our main results, to be presented in more detail below, are the following.
\begin{itemize}
\item Pullbacks exist in the category $\GrGp^*$ of pointed graphs of groups, and we give an explicit construction for them (Theorem \ref{thm A}).

This generalises what can be naturally deduced from the Bass--Serre theory, for so-called core graphs of groups. We do not require coreness or, even, the immersion hypothesis. In addition, the explicit description of pullbacks we give is key to applications developed in~\cite{dllrw_fgip}.

\item Pullbacks do not always exist in $\GrGp$, the category of unpointed graphs of groups. However, they do exist under certain acylindricity conditions (Theorems~\ref{thm B} and \ref{thm E}).

\item Pullbacks are closely related to a construction called the \emph{$\AA$-product} of  graphs of groups: The $\AA$-product of two graphs of groups $\BB$ and $\CC$ (or, more precisely, two morphisms $\mu^B\colon \BB \to \AA$ and $\mu^C \colon \CC \to \AA$) is a graph of groups $\BB\wtimes_\AA\CC$, together with projection morphisms $\rho^B$ and $\rho^C$ to $\BB$ and $\CC$, respectively, satisfying certain desirable lifting properties in the spirit of the definition of pullbacks.

\item If the morphisms $\mu^B$ and $\mu^C$ are between connected, pointed graphs of groups, we identify a vertex $x$ in $\BB\wtimes_\AA\CC$ such that the connected component of the $\AA$-product containing $x$ is  the pullback of our two morphisms. This is similar to the classical construction of pullbacks of Stallings' graphs.

\item In further similarity, we relate the different connected components of $\BB\wtimes_\AA\CC$ to intersections of conjugates of the images of the fundamental groups of $\BB$ and $\CC$ under $\mu^B$ and $\mu^C$, and to certain double cosets of these subgroups (Theorems~\ref{thm C} and \ref{thm D}).
\end{itemize}

\medskip

The paper is organised as follows.
In Section~\ref{sec: graphs of groups}, after briefly reviewing relevant
notions from graph and group theory, we present the basic theory of graphs of groups and their morphisms, largely following the approach of 
\cite{ser80} 
and \cite{bas93}, fixing some terminology and notational preferences along the way.

In Section~\ref{ssec: morphisms}
we focus on morphisms of graphs of groups, for which we slightly adapt the formalisms in \cite{bas93} and \cite{kwm05}.
If $\mu\colon \BB \to \AA$ is a morphism of graphs of groups, and $v$ is a vertex of $\BB$, we say that $\mu\colon (\BB,v) \to (\AA,\mu(v))$ is a morphism of pointed graphs of groups.
Then $\mu$ induces a morphism $\mu_*\colon \pi_1(\BB,v) \to \pi_1(\AA,\mu(v))$ between the corresponding fundamental groups. Special attention is given to the notion of an immersion of graphs of groups, which generalises the topological notion of the same name and, together with coreness, yields a natural description of subgroups of fundamental groups --- extending those given by Stallings' graphs and subsequent constructions.

\medskip 

In order to define the categories $\GrGp$ and $\GrGp^*$,
morphisms of graphs of groups need to be considered up to equivalence. These equivalence relations are introduced in Section~\ref{sec: GrGp}: one equivalence relation between unpointed morphisms, denoted by~$\approx$, and one between pointed morphisms, denoted by~$\sim$. The definitions of both kinds of equivalence involve the existence of certain families of parameters, whose behaviour --- crucial in the description of pullbacks --- is studied in Section~\ref{ssec: parameters}. The category $\GrGp$ is then defined as the category with objects graphs of groups and with arrows $\approx$-equivalence classes of morphisms. Similarly, the category $\GrGp^*$ is defined as the category with objects pointed graphs of groups and with arrows $\sim$-equivalence classes of morphisms.

\medskip 

In Section~\ref{sec: pullbacks}, we establish our first main theorems, on pullbacks in $\GrGp^*$ and $\GrGp$. More precisely, in Section~\ref{ssec: A-product},  we develop the main construction on which our description of pullbacks is based, the \emph{$\AA$-product} of two 
morphisms of graphs of groups $\mu^B \colon \BB \to \AA$ and $\mu^C \colon \CC \to \AA$ (or of the graphs of groups $\BB$ and $\CC$, if the morphisms $\mu^B$ and $\mu^C$ are understood),
which we denote by $\BB \wtimes_\AA \CC$.
Although the full definition is quite technical, we outline its structure below. The $\AA$-product $\BB \wtimes_\AA \CC$ is a graph of groups built on top of the pullback $\gr{B} \times_{\gr{A}} \gr{C}$ of the underlying graphs for $\BB$ and~$\CC$. Each vertex and edge of $\gr{B} \times_{\gr{A}} \gr{C}$ gives rise to a set of double cosets --- of the images of the vertex and edge groups of $\BB$ and $\CC$ in those of $\AA$, under the morphisms
$\mu^B$ and $\mu^C$. These double cosets are taken to be the vertices and edges of the product $\BB \wtimes_\AA \CC$. The corresponding vertex and edge groups are given by certain twisted pullbacks of the images of the vertex and edge groups of $\BB$ and $\CC$ under $\mu^B$ and $\mu^C$. The missing ingredients of the definition of $\BB\wtimes_\AA\CC$, namely the adjacency maps and the projection morphisms $\rho^B$ and $\rho^C$, from the  $\AA$-product to $\BB$ and $\CC$, respectively, are defined naturally.
A pointed version of the $\AA$-product, based on morphisms of pointed graphs of groups, arises seamlessly by selecting the natural basepoint in $\BB \wtimes_\AA \CC$. 

Although the definition of $\AA$-products depends on certain choices of parameters, we prove that the resulting objects are independent of the defining parameters (up to isomorphism in $\GrGp^*$ in the pointed case, and in $\GrGp$ in the unpointed case). Technical lifting lemmas implicit in the definition of pullbacks are established in Section~\ref{sec: lifts}. More precisely, we show that, if $\sigma^B\colon \DD \to \BB$ and $\sigma^C\colon \DD \to \CC$ are morphisms such that $\mu^B \circ \sigma^B$ and $\mu^C \circ \sigma^C$ are equivalent, there exists at least one morphism (a lift) $\sigma$ from $\DD$ to $\BB\wtimes_\AA\CC$ such that $\sigma^B$ (resp. $\sigma^C$) is equivalent to $\rho^B \circ \sigma$ (resp. $\rho^C \circ \sigma$),
and we describe all possible such lifts.

Since the existence of pullbacks is determined by the existence and uniqueness of such lifts, this leads to the following theorem.

\begin{maintheorem}[Theorem~\ref{thm: pullbacks in GrGp*}]
\label{thm A}
Pullbacks exist in the category $\GrGp^*$. More precisely,
the pullback of two morphisms between connected pointed graphs of groups is their pointed $\AA$-product.
\end{maintheorem}
A variant of this result restricted to immersions easily follows and provides a relevant description of subgroup intersection in terms of the $\AA$-product (Corollary~\ref{cor: intersection of subgroups}).
In Section~\ref{sec: worked out example} we carry out an explicit computation of the pullback of two graphs of groups immersing into the standard graph of groups of a Baumslag--Solitar group. We use this to give an explicit example of two finitely generated subgroups of $\bs(m, n)$ (with $|m|, |n|>1$) that intersect in a non-finitely generated subgroup. The existence of such subgroups was already known by work of Paramantzoglou \cite{pa12}.
In \cite{dllrw_fgip}, this approach is exploited further to derive the \fgip\ and related properties from the obtained pullback description in broader classes of groups.

\medskip

On the other side, we show that, somewhat surprisingly, pullbacks do not always exist in the category $\GrGp$ of unpointed graphs of groups (see Example~\ref{ex: pullbacks may not exist}).

\begin{maintheorem}[Example~\ref{ex: pullbacks may not exist} and Theorem~\ref{thm: pullbacks_exist_sometimes}]
\label{thm B}
Pullbacks do not always exist in the category~$\GrGp$. However, when they do, they are subgraphs of groups of the $\AA$-product of the morphisms involved.
\end{maintheorem}

In Section \ref{sec: subcategory of core gog}, we restrict our attention to morphisms of graphs of groups that uniquely describe subgroups
--- much in the spirit of how Stallings' graphs describe subgroups of the free group ---
and apply our previous results on pullbacks to derive information about intersections.
Suitable adaptations of the notions of immersion and covering to graphs of groups, due to Bass \cite{bas93}, play an analogous role in describing subgroups within this broader setting. As in the classical case, the bottom line is that coverings entail a locally bijective behaviour whereas immersions entail only a locally injective behaviour. 

In particular, we recall the graphical notions of pointed and unpointed core, and their extension to the realm of graphs of groups (denoted by $\core(\AA,u)$ and $\core(\AA)$, respectively), see Section~\ref{sec: core gog and pullbacks}. It turns out that, as in the graphical case, $\core(\AA,u)$
is a minimal subgraph of groups of $\AA$ 
which encapsulates 
the fundamental group of~$\AA$ at $u$, that is, $\pi_1(\AA,u) = \pi_1(\core(\AA,u),u)$.

After reformulating Bass' definition of coverings of a graph of groups in our language, in Section~\ref{sec: covers} we establish the main results one may expect from the topological setting: lifting results for coverings of graphs of groups in $\GrGp^*$ and $\GrGp$ (Theorems \ref{thm: lifting_unpointed_morphisms} and \ref{thm: lifting_pointed_morphisms}), and bijective correspondences between the isomorphism classes in $\GrGp^*$ of pointed coverings of $(\AA,u)$ and subgroups of $\pi_1(\AA,u)$, and  between isomorphism classes in $\GrGp$ of coverings of $\AA$ and conjugacy classes of subgroups of $\pi_1(\AA,u)$.

In Section~\ref{sec: double cosets} we investigate the connected components of the $\AA$-product $\BB\wtimes_\AA\CC$ of immersions in terms of certain double cosets of the form $B\,a\,C$ in the fundamental group~$A = \pi_1(\AA, u)$, where $B$ and $C$ are the images in $A$ of the fundamental groups of $\BB$ and $\CC$ at vertices that map onto $u$. This correspondence is sharper when one of the two morphisms is a covering. More precisely, we prove the following theorems, using the above notation. 

\begin{maintheorem}[Theorem~\ref{thm: localy_elliptic_double_cosets}]
\label{thm C}
    Let $\mu^B\colon \BB \to \AA$ and $\mu^C\colon \CC\to \AA$ be immersions of graphs of groups and let $\DD = \BB\wtimes_{\AA}\CC$ be the corresponding $\AA$-product. Then, there exists a bijection between the set of connected components in $\core(\DD)$ and the set of double cosets $B\,g\,C$ such that $B^g\cap C$ is not locally elliptic. 
\end{maintheorem}

\begin{maintheorem}[Theorem~\ref{thm: double_cosets}]
\label{thm D}
    Let $\mu^B\colon \BB\to \AA$ be a covering and let $\mu^C\colon \CC\to \AA$ be an immersion of graphs of groups. Then, there exists a bijection between the set of components of the $\AA$-product $\BB\wtimes_{\AA}\CC$ with non-trivial fundamental group and the set of double cosets $B\,g\,C$ such that $B^g\cap C$ is not trivial.
\end{maintheorem}

Note that, in the case of Stallings' graphs for subgroups of free groups, local ellipticity is equivalent to triviality, and hence there is no distinction between Theorems~C and~D.

Finally, in Section \ref{sec: acylindrical gog}, we leverage some of our previous results to obtain specific conclusions regarding pullbacks and intersections within the family of acylindrical graphs of groups.

\begin{maintheorem}[Corollary~\ref{cor: acylindrical pullback}]
\label{thm E}
    If $\AA$ is an acylindrical graph of groups, then, within the category of core graphs of groups with immersions, pullbacks of morphisms to $\AA$ exist.
\end{maintheorem}

\section{Graphs of groups and their morphisms}\label{sec: graphs of groups}

The interplay between graphs and groups is at the center of this paper. Background facts about graphs and groups are discussed in Section~\ref{sec: graphs and groups}. 
We then introduce standard material about graphs of groups, in the spirit of \cite{ser80} 
and \cite{bas93}, establishing along the way the notation we will use throughout. The formal definition of graphs of groups is given in Section~\ref{sec: define gog}, where we also walk the reader through the central notions of $\AA$-paths and the fundamental group of a graph of groups.
In Section~\ref{ssec: morphisms}, we give a careful discussion of morphisms between graphs of groups, largely following \cite{bas93} and \cite{kwm05}.

\subsection{Graphs and groups}\label{sec: graphs and groups}

Most of the general notation and terminology for groups used in this paper is standard. We outline here some specific conventions adopted.
If $G$ is a group and $S$ a subset of $G$, $\langle S\rangle$ denotes the subgroup of $G$ generated by $S$. If $G = \langle S\rangle$ for some finite set $S$, we say that $G$ is \emph{finitely generated}.

If $H$ is a subgroup of $G$, we write $H \leqslant G$.
If $H \leqslant G$, a \emph{left coset} (\resp \emph{right coset}) of $H$ is a set of the form $g H = \{gh \st h \in H \}$ (\resp $H g = \{hg \st h \in H \}$), where $g\in G$. We denote by $G/H$ (\resp $H\backslash G$) the set of left (right) cosets of $H$ in $G$.

Similarly, if $H,K \leqslant G$, an \emph{$(H,K)$-double coset} is a set of the form $HgK
=
\{ hgk \st h \in H \text{ and } k \in K\}$, where $g\in G$; and we denote by $\dblcoset{H}{G}{K}$ the set of $(H,K)$-double cosets in $G$. 

If $x,g\in G$, we write $x^g$ for the conjugate $g\inv xg$, and we denote by $\gamma_g$ the inner automorphism $x \mapsto x^g$ of $G$.

In this paper, by \emph{graph} we mean a \emph{directed graph with a fixpoint-free involution on edges}, sometimes known as a \emph{Serre graph}.
Formally, a \emph{graph} $\gr{A}$ consists of a set of vertices $V(\gr{A})$, a set of edges $E(\gr{A})$, origin and target maps $o, t\colon E(\gr{A})\to V(\gr{A})$ and a fixpoint-free involution $e \mapsto e\inv$  on $E(\gr{A})$, such that $t(e\inv) = o(e)$ (and $o(e\inv) = t(e)$) for every edge $e$.

If $k\ge 1$, a \emph{path of length $k$} in a graph $\gr{A}$ is a sequence $(e_1,\dots, e_k)$ of consecutive edges, that is, such that $t(e_i) = o(e_{i+1})$ for each $i \in [1,k-1]$. We say that $o(e_1)$ and $t(e_k)$ are the \emph{initial} and \emph{terminal} vertices of the path, or that it is a path \emph{from $o(e_1)$ to $t(e_k)$}. By convention, we also consider that there exists a \emph{trivial path} (of length 0) at every vertex of $\gr{A}$. If $(e_1,\dots, e_k)$ and $(e_{k+1},\dots, e_\ell)$ are consecutive paths (that is, if $t(e_k) = o(e_{k+1})$), then the \emph{concatenation} of these two paths is the path $(e_1,\dots, e_k, e_{k+1},\dots, e_\ell)$, of length $k+\ell$. Trivial paths act as identities with respect to concatenation in the natural way. The \defin{inverse} of a path $(e_1,\ldots,e_k)$ is the path $(e_k^{-1},\ldots,e_1^{-1})$.

A (sub)path of the form $(e,e^{-1})$ is called a \emph{backtracking} (at $t(e)$). A path without backtracking is called \emph{reduced}.
Two paths in a graph $\gr{A}$ are \emph{equivalent} (or rather, \emph{path homotopic}) if one can be obtained from the other through a finite sequence of backtracking insertions and removals. Such paths are necessarily \emph{coterminal}, that is, they have the same initial and terminal vertices. 
Each equivalence class of paths contains a unique reduced path. Equivalence is clearly compatible with concatenation, endowing the set of reduced paths in $\gr{A}$ with the structure of a groupoid, called the \emph{fundamental groupoid} of the graph $\gr{A}$, denoted by $\pi_1(\gr{A})$.

A path $p = (e_1,\dots, e_k)$ is said to be \emph{closed} (or a \emph{circuit}) at $u \in V (\gr{A})$, if its initial and terminal vertices are equal to $u$, that is, if $o(e_1) = t(e_k) = u$.  For every vertex $u \in V(\gr{A})$, the restriction of $\pi_1(\gr{A})$ to reduced circuits at $u$ is a group called the \defin{fundamental group} of $\gr{A}$ at $u$, denoted by $\pi_1(\gr{A},u)$.

A graph $\gr{A}$ is said to be \emph{core with respect to a vertex} $u \in V(\gr{A})$ if every vertex in $\gr{A}$ appears in some reduced circuit at $u$.
A non-trivial circuit $(e_1,\dots, e_k)$ is said to be \emph{cyclically reduced} if it is reduced and $e_k \neq e_1^{-1}$.
The graph $\gr A$ is said to be \emph{core} if every one of its vertices appears in some cyclically reduced circuit. Note that a graph that is core with respect to a vertex is always connected, while a core graph is not necessarily connected. The \emph{core of a graph $\gr{A}$ with respect to a vertex $u$}, denoted by $\core(\gr{A}, u)$, is the union of all reduced circuits based at $u$. The \emph{core of $\gr{A}$}, denoted by $\core(\gr{A})$, is the union of all cyclically reduced circuits.

All the maps between graphs considered in this paper are morphisms of graphs, that is, they map vertices to vertices and edges to edges, and they preserve the origin and target maps.

\paragraph{Notational convention}
In the sequel, we use \texttt{sans-serif} typeface ($\gr{A},\gr{B},\gr{C}, \ldots$) to denote graphs.

\subsection{Graphs of groups and their fundamental group}\label{sec: define gog}\label{sec: fundamental group}

A \emph{graph of groups} $\AA = (\gr{A}, \{A_u\}, \{A_e\}, \{\alpha_e, \omega_e\})$ consists of
\begin{itemize}
\item an underlying graph $\gr{A}$;
\item a collection of groups $\{A_u\}$ indexed by $V(\gr{A})$, called the \emph{vertex groups};
\item a collection of groups $\{A_e\}$ indexed by $E(\gr{A})$, such that $A_e = A_{e^{-1}}$, called the \emph{edge groups};
\item monomorphisms $\alpha_e\colon A_e\to A_{o(e)}$, $\omega_e\colon A_e\to A_{t(e)}$ called the \emph{edge maps}, satisfying $\alpha_{e} = \omega_{e^{-1}}$ and $\omega_e = \alpha_{e^{-1}}$. 
\end{itemize}
A \emph{pointed graph of groups} is a pair $(\AA,u)$ where $\AA$ is a graph of groups and $u$ is a distinguished vertex of the underlying graph, usually called the \emph{basepoint}.

\paragraph{Notational convention}
We are going to deal with a number of graphs of groups in this paper. To lighten notation, we implicitly keep the following convention: the underlying graphs of $\AA, \BB, \CC, \DD$ are $\gr A, \gr B, \gr C, \gr D$, respectively. The names of vertices of $\AA$ involve the letter $u$, e.g. $u, u', u_0,\dots$, and those of $\BB$, $\CC$ and $\DD$ involve the letters $v$, $w$ and $x$, respectively. Similarly the names of edges of $\AA$, $\BB$, $\CC$, $\DD$ involve the letters $e$, $f$, $g$ and $h$, respectively.

\medskip

Let $\AA = (\gr{A}, \{A_u\}, \{A_e\}, \{\alpha_e, \omega_e\})$ be a graph of groups. An \emph{$\AA$-path} is a sequence of the form
\[p \,=\, (a_0, e_1, a_1, \dots, e_k, a_k),\]
where $(e_1,\dots, e_k)$ is a path in $\gr{A}$ (called the \emph{underlying path of $p$}), $a_{i-1} \in A_{o(e_i)}$ for all $i \in [1,k]$, and $a_k \in A_{t(e_k)}$. $\AA$-paths inherit terminology and notation from their underlying paths in the natural way. For example, the \emph{length} of an $\AA$-path, and its initial and terminal vertices, are those of the underlying path. An $\AA$-path is \emph{closed} (an \emph{$\AA$-circuit}) if its underlying path is closed. If $u \in V(\gr{A})$, we say that the $\AA$-path $(1_{A_u})$, of length 0, is the \emph{trivial $\AA$-path at $u$}. Note that there are $\AA$-paths of length 0 which are not trivial: if $a\in A_u \setminus \{1_{A_u}\}$, then $(a)$ is a non-trivial $\AA$-path whose underlying path is the trivial path at $u$.

Concatenation of paths has a natural counterpart in $\AA$-paths: if $p = (a_0, e_1, a_1, \dots, e_k, a_k)$ is an $\AA$-path  from $u$ to $u'$ and $q = (a'_0, e'_1, a'_1, \dots, e'_{k'}, a'_{k'})$ is an $\AA$-path from $u'$ to $u''$, then the \textit{concatenation} of $p$ and $q$ is the $\AA$-path $p\,q = (a_0, e_1, a_1, \dots, e_k, a_ka'_0, e'_1, a'_1, \dots, e'_{k'}, a'_{k'})$ from $u$ to $u''$. Note in particular that if $e\in E(\gr{A})$ is an edge from $u$ to $u'$, $a\in A_u$ and $a'\in A_{u'}$, then $(a,e,a') = (a)\ (1,e,1)\ (a')$. More generally, every $\AA$-path is a product of $\AA$-paths of length 0 (that is, of elements of vertex groups), and $\AA$-paths of length 1, of the form $(1,e,1)$.

The concatenation operation endows the set of all $\AA$-paths with the structure of a category, where objects are the vertices of $\gr{A}$ and arrows are $\AA$-paths. We introduce a first set of important definitions.

\begin{defn}\label{def: simAA and =AA, reduced, inverse}
Let $\AA$ be a graph of groups and let $p = (a_0,e_1,a_1,\dots, e_k,a_k)$ be an $\AA$-path.
\begin{enumerate}[(1)]
\item The congruence on $\AA$-paths generated by the pairs $\left((\alpha_e(x),e,\omega_e(x)\inv),\ (1,e,1)\right)$ ($e\in E(\gr A)$, $x\in A_e$) is denoted by $\sim_\AA$.

\item The congruence on $\AA$-paths generated by the pairs $\left((1,e,\omega_e(x),e\inv,1),\ (\alpha_e(x))\right)$ ($e\in E(\gr A)$, $x\in A_e$) is denoted by $=_\AA$.

\item The $\AA$-path $p$ is said to be \emph{reduced} if, for every $i \in [1,k-1]$ such that $e_{i+1} = e_i\inv$, the group element $a_i$ is not in $\omega_{e_i}(A_{e_i})$.

\item The \defin{inverse ($\AA$-path)} of $p$ is
$p\inv = (a_k\inv,e_k\inv,\dots, a_1\inv,e_1\inv,a_0\inv)$. Note that $pp^{-1} =_{\AA} (1)$ (at vertex $o(p)$) and $p^{-1}p =_{\AA} (1)$ (at vertex $t(p)$), justifying the notation $p^{-1}$.
\end{enumerate}
\end{defn}

We record the following useful facts, several of which are elementary. Statement~\eqref{eq: for reduced, sim is =} is  \cite[Corollary 1.10]{bas93}.

\begin{prop}\label{prop: sim vs equiv}
Let $\AA$ be a graph of groups, and let $p, q$ be $\AA$-paths.
\begin{enumerate}[(1)]
\item If $p\sim_\AA q$ or $p =_\AA q$, then $p$ and $q$ are coterminal. If $p\sim_\AA q$, then $p$ and $q$ have the same length. If $p =_\AA q$, then the lengths of $p$ and $q$ are congruent mod 2.\label{eq: length preserved}

\item There is a reduced $\AA$-path $p'$ such that $p =_{\AA} p'$. \label{eq: reduce path}

\item If $p\sim_\AA q$, then $p =_\AA q$. \label{eq: 1e1}

\item If $p$ and $q$ are reduced $\AA$-paths, we have $p =_\AA q$ if and only if $p\sim_\AA q$. \label{eq: for reduced, sim is =}

\item An $\AA$-path is reduced if and only if it is a shortest representative of its $=_\AA$-class. \label{eq: shortest is reduced}
\end{enumerate}
\end{prop}

\begin{proof}
Statements \eqref{eq: length preserved} and \eqref{eq: reduce path} are immediate from the definitions.
To establish Statement~\eqref{eq: 1e1}, we note that 
\begin{align*}
(1,e,1) &\enspace=\enspace (\alpha_e(x))\ (\alpha_e(x\inv))\ (1,e,1) \\
&\enspace=_\AA\enspace (\alpha_e(x))\ (1,e,\omega_e(x\inv),e^{-1},1)\ (1,e,1) \\
&\enspace=_\AA\enspace (\alpha_e(x),e,\omega_e(x)\inv)\ (1,e\inv,1,e,1) \\
&\enspace=_\AA\enspace (\alpha_e(x),e,\omega_e(x)\inv)\ (1),
\end{align*}
that is,  $(1,e,1) =_\AA (\alpha_e(x),e,\omega_e(x)\inv)$. Statement~\eqref{eq: for reduced, sim is =} is proved in \cite[Corollary 1.10]{bas93}. Statement~\eqref{eq: shortest is reduced} follows from \eqref{eq: for reduced, sim is =} (using the fact that $\sim_\AA$-equivalent $\AA$-paths have the same length).
\end{proof}

The following fact will be used extensively throughout the article. It follows from the definition of the $\sim_{\AA}$-congruence.

\begin{lem}
\label{lem: equivalent A-paths} 
If $\AA$ is a graph of groups and if $p$ and $q$ are co-terminal $\AA$-paths, 
$p = (a_0,e_1,a_1,\dots, e_k,a_k)$ and $q = (b_0,f_1,b_1,\dots, f_\ell,b_\ell)$, then $p \sim_\AA q$ if and only if $k = \ell$, $e_i = f_i$ for each $i\in [1,k]$, and there exist elements $x_i\in A_{e_i}$($i\in [1,k]$) such that
\begin{align*}
a_0\,\alpha_{e_1}(x_1) & = b_0, \\
\omega_{e_i}(x_i)^{-1} a_i\,\alpha_{e_{i+1}}(x_{i+1}) & = b_i \quad \text{for } i \in [1,k-1], \text{and}\\
\omega_{e_k}(x_k)^{-1}\,a_k & = b_k.
\end{align*}

\end{lem}

We conclude this section by introducing the essential notions of fundamental groupoid and fundamental group of a graph of groups.

\begin{defn}
The quotient $\pi_1(\AA)$ of the category of $\AA$-paths by $=_\AA$ is a groupoid, called the \emph{fundamental groupoid} of $\AA$. If $u\in V(\gr{A})$, the \emph{fundamental group of $\AA$ at $u$}, written $\pi_1(\AA,u)$, is the local group of $\pi_1(\AA)$ at $u$, that is, the set of arrows (i.e., $=_\AA$-classes of $\AA$-paths) of $\pi_1(\AA)$ from $u$ to $u$, with the induced product and group structure.
\end{defn}

\subsection{Morphisms between graphs of groups}\label{ssec: morphisms}

In this section, we reformulate facts and definitions from \cite{bas93} and \cite{kwm05}. Formal differences are discussed in Remark~\ref{rem: differences with KMW}. 

\begin{defn}\label{def: morphism}
Let $\AA = (\gr{A}, \{A_u\}, \{A_e\}, \{\alpha_e, \omega_e\})$ and $\BB = (\gr{B}, \{B_v\}, \{B_f\}, \{\alpha_f, \omega_f\})$ be graphs of groups. A \emph{morphism of graphs of groups} $\mu\colon\BB\to \AA$ consists of
\begin{itemize}
\item a morphism of underlying graphs $[.]\colon \gr{B}\to \gr{A}$ (sometimes also written $\mu$);
\item for every $v\in V(\gr{B})$, a monomorphism $\mu_v\colon B_v\to A_{[v]}$;
\item for every $f\in E(\gr{B})$, a monomorphism $\mu_f\colon B_f\to A_{[f]}$; and
\item for every edge $f\in E(\gr{B})$ from $v$ to $v'$, elements $f_{\alpha}\in A_{[v]}$, $f_{\omega}\in A_{[v']}$, called the \emph{twisting elements for $f$}, such that $(f\inv)_{\alpha} = f_{\omega}$ and $(f\inv)_\omega = f_\alpha$,
\end{itemize}
satisfying the following property: for each edge $e\in E(\gr{A})$, each edge $f\in E(\gr{B})$ from $v$ to $v'$ such that $[f] = e$, we have
\begin{align} 
\label{eq: twisted commutation alpha}
\alpha_e\circ\mu_{f} &= \inn{f_{\alpha}}\! \circ \mu_{v} \circ \alpha_{f}\\
\label{eq: twisted commutation omega}
\omega_e\circ\mu_f &= \inn{f_{w}} \! \circ \mu_{v'} \circ \omega_{f}.
\end{align}
See Figure~\ref{fig: hom GGs} for a visual summary.

If $(\BB, v)$ and $(\AA, u)$ are pointed graphs of groups and $\mu\colon\BB\to\AA$ is a morphism of graphs of groups, we say that $\mu\colon(\BB,v)\to(\AA,u)$ is a \emph{morphism of pointed graphs of groups} (or, sometimes, a \emph{pointed morphism of graphs of groups}) if $[v] = u$. 
\end{defn}
\begin{figure}[H]
\centering
\begin{tikzpicture}[shorten >=3pt, node distance=3cm and 2cm, on grid,auto,-latex]

    \node[] (B) {$\gr{B}$};
    \node[] (v) [right = 2 of B]{$v$};
    \node[] (v') [right = 4 of v]{$v'$};
    \node[] (f) [above right = 0.25 and 2 of v]{$f$};
    
    \node[] (Bv) [below = 0.75 of v]{$B_v$};
    \node[] (Bf) [below right = 0.75 and 2 of v]{$B_f$};
    \node[] (Bv') [below = 0.75 of v']{$B_{v'}$};

    \node[] (Au) [below = 1.5 of Bv]{$A_u$};
    \node[] (Au2) [below = 1.5 of Au]{$A_u$};
    \node[] (Ae) [right = 2 of Au2]{$A_e$};
    \node[] (Au') [below = 1.5 of Bv']{$A_{u'}$};
    \node[] (Au'2) [below = 1.5 of Au']{$A_{u'}$};
    
    \node[] (u) [below = 0.75 of Au2]{$u$};
    \node[] (uu) [left = 0.75 of u]{$\mu(v) = $};
    \node[] (u') [below = 0.75 of Au'2]{$u'$};
    \node[] (uu') [right = 0.75 of u']{$ = \mu(v')$};
    \node[] (e) [below right = 0.25 and 2 of u]{$e = \mu(f)$};
    \node[] (A) [left = 2 of u] {$\gr{A}$};

    \path[->] (B) edge[]
    node[pos=0.5,left]{$\normalsize{\mu}$}
    (A);

    \path[-{Latex}] (v) edge[] 
    node[pos=0.1]{${f_\alpha}$}
    node[pos=0.85]{${f_\omega}$}
    (v');
    \path[-{Latex}] (u) edge[] 
    node[pos=0.1]{${f_\alpha \atop \longrightarrow}$}
    node[pos=0.85]{${f_\omega \atop \longleftarrow}$}
    (u');

    \path[>->] (Bf) edge[] 
    node[pos=0.5,above]{$\alpha_{f}$}
    (Bv);
    \path[>->] (Bf) edge[] 
    node[pos=0.5]{$\omega_{f}$}
    (Bv');
    \path[>->] (Ae) edge[] 
    node[pos=0.5]{$\alpha_{e}$}
    (Au2);
    \path[>->] (Ae) edge[] 
    node[pos=0.5,below]{$\omega_{e}$}
    (Au'2);

    \path[>->] (Bv) edge[] 
    node[pos=0.5,left]{$\mu_{v}$}
    (Au);
    \path[>->>] (Au) edge[] 
    node[pos=0.5,left]{$\gamma_{f_\alpha}$}
    (Au2);
    \path[>->] (Bf) edge[] 
    node[pos=0.5,left]{$\mu_{f}$}
    (Ae);
    \path[>->] (Bv') edge[] 
    node[pos=0.5]{$\mu_{v'}$}
    (Au');
    \path[>->>] (Au') edge[] 
    node[pos=0.5,right]{$\gamma_{f_\omega}$}
    (Au'2);
        
\end{tikzpicture}
\caption{Scheme of a morphism  $\mu\colon \BB \to \AA$ of graphs of groups}
\label{fig: hom GGs}
\end{figure}
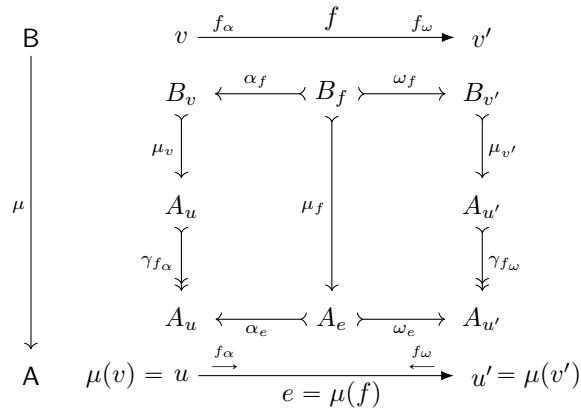

\begin{rem}\label{rem: Hom GGs}
The twisting elements can be seen as labels ``dropped'' by edge $f$ onto edge $e$, activated when $f$ is traversed, with $f_\alpha$ near the origin of $e$, pointing towards the middle of the edge, and $f_\omega$ near the end of $e$, pointing towards the middle of the edge as well, see Figure~\ref{fig: hom GGs}.
\end{rem}

\begin{rem}
If $\mu \colon \BB \to \AA$ is a morphism of graphs of groups, then, for all $f \in E(\gr{B})$, $\mu_f =\mu_{f^{-1}}$.
This follows easily from Equations~\eqref{eq: twisted commutation alpha} and~\eqref{eq: twisted commutation omega} taking into account that
for every $e \in E(\gr{A})$,
    $\alpha_{e^{-1}} = \omega_{e}$; and
    for every $f \in E(\gr{B})$,
    $\alpha_{f^{-1}} = \omega_{f}$ and
    $(f^{-1})_{\alpha} = f_{\omega}$.
\end{rem}

\begin{defn} \label{rem: morphisms on A-paths}
A morphism of graphs of groups $\mu \colon \BB \to \AA$ extends to a map from $\BB$-paths to $\AA$-paths as follows: if
\[
p = (b_0, f_1, b_1, \dots, f_k, b_k)
\]
is a $\BB$-path, with $f_i$ an edge from $v_{i-1}$ to $v_i$ and $e_i = \mu(f_i) \in E(\gr{A})$, then:
\[
\mu(p) = \left(\mu_{v_0}(b_0)(f_1)_\alpha, e_1, (f_1)\inv_\omega\mu_{v_1}(b_1)(f_2)_\alpha, \dots, e_k, (f_k)_\omega\inv\mu_{v_k}(b_k)\right).
\]
It is worth noting that, if $f\in E(\gr{B})$, then the $\mu$-image of the length 1 $\BB$-path $(1,f,1)$ is $(f_\alpha, \mu(f), f_\omega\inv)$. 
\end{defn}

\begin{prop}\label{prop: preservation of sim and equiv}
Let $\mu\colon \BB \rightarrow \AA$ be a morphism of graphs of groups as above, and let $p, q$ be $\BB$-paths.
\begin{enumerate}[(1)]
\item If $p$ and $q$ are consecutive $\BB$-paths, then $\mu(p\,q) = \mu(p)\,\mu(q)$.\label{mu preserves products}

\item If $p =_\BB q$, then $\mu(p) =_\AA \mu(q)$. \label{mu preserves equivalence}
\end{enumerate}
\end{prop}

\begin{proof}
The verification of Statement~\eqref{mu preserves products} is immediate.

To establish Statement~\eqref{mu preserves equivalence}, we only need to show that, for each edge $f\in E(\gr{B})$ and each $x\in B_f$, we have $\mu((1,f,\omega_f(x),f^{-1},1)) =_\AA \mu((\alpha_f(x)))$. Indeed, letting
$v = o(f)$, $v'=t(f)$ and $e = \mu(f)$,
we have
\begin{align*}
\mu((1,f,\omega_f(x),f^{-1},1)) &= (f_\alpha,e,f_\omega\inv\,\mu_{v'}(\omega_f(x))\,(f\inv)_\alpha, e\inv, (f\inv)_\omega\inv) \\
&= (f_\alpha,e,f_\omega\inv\,\mu_{v'}(\omega_f(x))\,f_\omega, e\inv, f_\alpha\inv) \\
&= (f_\alpha,e,\omega_e(\mu_f(x)), e\inv, f_\alpha\inv)\quad\text{by Condition~\eqref{eq: twisted commutation alpha}} \\
&=_\AA (f_\alpha\,\alpha_e(\mu_f(x))\,f_\alpha\inv)\\
&=_\AA (\mu_v(\alpha_f(x))) = \mu((\alpha_f(x)))\quad\text{by Condition~\eqref{eq: twisted commutation alpha} again,}
\end{align*}
which concludes the proof.
\end{proof}

The following is an immediate corollary of Proposition~\ref{prop: preservation of sim and equiv}.

\begin{cor}\label{cor: mu* is a group morphism}
A morphism of graphs of groups $\mu\colon \BB \to \AA$ induces a morphism of groupoids $\mu_*\colon \pi_1(\BB) \to \pi_1(\AA)$.

If $v \in V(\gr{B})$ and $u = [v] \in V(\gr{A})$, the restriction of $\mu_*$ to $\pi_1(\BB,v)$ is a group morphism $\mu_*\colon \pi_1(\BB,v) \to \pi_1(\AA,u)$.
\end{cor}

It is easy to see that the image of a non-reduced
$\BB$-path under a morphism is never reduced, but the converse is not necessarily true. This justifies the introduction of the important notion of an \emph{immersion}; namely, those for which this converse also holds, see Proposition~\ref{folded}.

\begin{defn}\label{def: folded}
Let $\AA$ and $\BB$ be graphs of groups and let $\mu\colon \BB \to \AA$ be a morphism of graphs of groups. We say that $\mu$ is an \emph{immersion} if the following holds:
\begin{enumerate}[(1)]
\item if $f$ and $f'$ are edges of $\gr{B}$ with the same image $e = \mu(f) = \mu(f') \in E(\gr{A})$ and the same initial vertex $v = o(f) = o(f')$, then $f = f'$ if and only if $\mu_v(B_v)f_{\alpha}\alpha_e(A_e) = \mu_v(B_v)f_{\alpha}'\alpha_e(A_e)$;

\item if $f$ is an edge of $\gr B$, $e = \mu(f)$ and $v = o(f)$, then $(\alpha_e\circ\mu_f)(B_f) = \mu_v(B_v)^{f_{\alpha}}\cap \alpha_e(A_e)$.
\end{enumerate}
\end{defn}

\begin{rem}
\label{rem: folded}
    Bass \cite{bas93} defines an immersion of graphs of groups to be a morphism $\mu\colon \BB\to \AA$ such that for each vertex $v\in V(\gr{B})$, the following map
    \begin{align*}
        \bigsqcup_{f\in \Star(v)}B_v/\alpha_f(B_f) &\to \bigsqcup_{e\in\Star(\mu(v))}A_{\mu(v)}/\alpha_{e}(A_{e})\\
        b\,\alpha_f(B_f) &\mapsto \mu_v(b)\, f_{\alpha}\, \alpha_{\mu(f)}(A_{\mu(f)})
    \end{align*}
    is injective. Recall that the \emph{star}, $\Star(v)$, of a vertex $v$ is defined to be the collection of edges with origin $v$. Bass' definition can be seen to be equivalent to Definition~\ref{def: folded}.
\end{rem}

The announced characterisation of immersions follows, see \cite[Lemma 4.2]{kwm05} for a proof. Note that there the authors use the notion of \emph{folded $\AA$-graph} to refer to an immersion of graphs of groups, see Remark~\ref{rem: differences with KMW}.

\begin{prop}\label{folded}
A morphism $\mu\colon \BB\to \AA$ of graphs of groups is an immersion if and only if it sends reduced $\BB$-paths to reduced $\AA$-paths.
\end{prop}

\begin{cor}\label{cor: composition of folded}
Let $\mu \colon \AA\to \BB$ and $\nu\colon \BB\to \CC$ be morphisms of graphs of groups. If $\mu$ and $\nu$ are immersions, so is $\nu\circ\mu$. Also, if $\nu$ and $\nu\circ\mu$ are immersions, so is $\mu$.
\end{cor}

\begin{proof}
This is a direct consequence of the characterisation of immersions in Proposition~\ref{folded} and of the observation that any morphism of graphs of groups maps non-reduced paths to non-reduced paths.
\end{proof}

This leads to the following important property of immersions.

\begin{cor}\label{cor: folded morphisms}
Let $\mu\colon \BB \to \AA$ be an immersion of graphs of groups.
\begin{enumerate}[(1)]
\item The morphism $\mu_*\colon \pi_1(\BB) \to \pi_1(\AA)$ is injective in the following sense: if $p, q$ are coterminal $\BB$-paths and $\mu(p) =_\AA \mu(q)$, then $p =_\BB q$. \label{folded preserves cong}

\item If $v\in V(\gr{B})$ and $u = [v] \in V(\gr{A})$, then the morphism $\mu_*\colon \pi_1(\BB,v) \to \pi_1(\AA,u)$ is injective. \label{folded implies injective}
\end{enumerate}
\end{cor}

\begin{proof}
Let $p,q$ be coterminal $\BB$-paths such that $\mu(p) =_\AA \mu(q)$. Then the concatenation $r = p\,q\inv$ is well-defined. Let $v = o(p)$ and $u = [v]$. Let also $r'$ be a shortest representative of the $=_\BB$-class of $r$; in particular, $r'$ is a reduced $\BB$-path by Proposition~\ref{prop: sim vs equiv}~\eqref{eq: shortest is reduced}.
Then $\mu(r') =_\AA \mu(r) =_\AA 1_u$ 
(the trivial $\AA$-path at vertex $u$). By Proposition~\ref{folded}, $\mu(r')$ is reduced and hence $\mu(r') = 1_u$. By definition of $\mu$, it follows that $r' = 1_v$ and hence, $p =_\BB rq =_\BB r'q =_\BB q$. This concludes the proof of Statement~\eqref{folded preserves cong}.
Statement~\eqref{folded implies injective} follows immediately.
\end{proof}

We also note the following technical result, which will be useful later.

\begin{lem}
\label{lem:immersion_inclusion}
    Let $\mu\colon (\BB, v_0)\to (\AA, u_0)$ be an immersion of connected graphs of groups. If $\mu_*$ is an isomorphism, then $\mu$ restricts to an inclusion on the underlying graphs and to isomorphisms on the vertex and edge groups.
\end{lem}

\begin{proof}
    Let us first remark that if $\mu_*$ is an isomorphism, then $\mu^v_*$ is also an isomorphism for all $v\in V(\gr{B})$, where here $\mu^v\colon (\BB, v) \to (\AA, \mu(v))$ denotes the morphism obtained from $\mu$ by changing basepoints.
    
    If $\mu$ did not restrict to an inclusion of the underlying graphs, then either a pair of distinct vertices $v_1, v_2$ are identified in $\gr{A}$, or no vertices are identified and a pair of distinct edges $f_1, f_2$ are identified in $\gr{A}$. In the first case, there would be a reduced path $b$ in $\gr{B}$ connecting $v_1$ with $v_2$ which maps to a reduced loop in $\gr{A}$ at $\mu(v_1)$ by Proposition~\ref{folded}. In particular, this reduced loop would be a non-trivial element in $\pi_1(\AA, \mu(v_1))$, which is not the image of any element in $\pi_1(\BB, v_1)$ under $\mu^{v_1}$. This contradicts the assumption that $\mu_*$ was an isomorphism. In the second case, the loop $f_1f_2^{-1}$ is reduced in $\gr{B}$ but maps to a non-reduced loop in $\gr{A}$, contradicting Proposition~\ref{folded}. Hence, we may conclude that $\mu$ restricts to an inclusion on the underlying graph.
    
    Now let $v\in V(\gr{B})$ and let $a\in A_{\mu(v)}$ be an arbitrary element. Since $\mu^v_*$ is an isomorphism, there is a reduced $\BB$-circuit $b$ at $v$ so that $\mu(b) =_{\AA} a$. By Proposition~\ref{folded}, $\mu(b)$ is a reduced $\AA$-circuit, and we must have $\mu(b)\sim_{\AA} a$. It follows that $b\in B_v$. Thus, $\mu$ is an isomorphism on each vertex group. It is an isomorphism on each edge group by definition of an immersion.
\end{proof}

\begin{rem}\label{rem: differences with KMW}
There are formal differences between our presentation and the definitions in the paper \cite{kwm05} by Kapovich, Weidmann and Myasnikov. Let $\AA$ be a graph of groups. These authors define the notions of an $\AA$-graph $\calB$, of the graph of groups $\BB$ defined by $\calB$, and finally of a folded $\AA$-graph. In the section above, we start from a graph of groups $\BB$.
Our definition of an immersion of graphs of groups $\mu\colon \BB \to \AA$ exactly coincides with their notion of a folded $\AA$-graph.
We also made the following change in the notation of the twisting elements of an edge $f$: what is denoted by $f_\omega$ in \cite{kwm05} is what we denote by $(f_\omega)\inv$.
\end{rem}

\paragraph{Notational Convention}
Since we shall often deal with multiple morphisms of graphs of groups, it will be useful to set out some notational simplifications for ease of reading. Our general setup will be to consider two morphisms from pointed graphs of groups $(\BB,v)$ and $(\CC,w)$ to a common pointed graph of groups $(\AA,u)$, which we will denote by $\mu$ with a superscript to denote their different provenance, e.g., $\mu^B\colon (\BB, v)\to (\AA, u)$ and $\mu^C\colon (\CC, w)\to (\AA, u)$.

\section{The categories of graphs of groups, pointed and unpointed} \label{sec: GrGp}

In this section we define the categories $\GrGp$ and $\GrGp^*$, of graphs of groups and pointed graphs of groups.

\subsection{Equivalence of morphisms and the categories $\GrGp$ and $\GrGp^*$}\label{sec: equivalence and categories}

The definition of these categories depends on equivalence relations between morphisms of pointed and unpointed graphs of groups, which we now define. We opted for ``internal'' definitions, strictly in terms of the parameters of morphisms. A useful characterisation is given in Proposition~\ref{prop: change of data}.

\begin{defn}\label{def: mor_sim}
Let $\mu^1,\mu^2$ be morphisms of graphs of groups from $\BB$ to $\AA$. We say that $\mu^1$ and $\mu^2$ are \emph{$\approx$-equivalent}, written $\mu^1\approx\mu^2$, if $\mu^1$ and $\mu^2$ agree on the underlying graph $\gr B$ and there exist elements $a_f\in A_{[f]}$ for each edge $f\in E(\gr{B})$ and $a_v\in A_{[v]}$ for each vertex $v\in V(\gr{B})$ such that the following holds:
\begin{enumerate}[(1)]
        \item\label{item: inverse of a f} for every edge $f\in E(\gr{B})$,  $a_{f\inv} = a_f$;

        \item\label{item: locally at v} for every vertex $v\in V(\gr{B})$, we have $\mu^2_v = \gamma_{a_v}\circ\mu_v^1$;
        \item\label{item: a v versus a f} for every vertex $v\in V(\gr{B})$ and every edge $f$ with $o(f) = v$, we have $a_v = f_\alpha^1\ \alpha_{[f]}(a_f)\,(f_\alpha^2)\inv$, where $f_\alpha^1$ and $f_\alpha^2$ are the twisting elements associated to $f$ in $\mu^1$ and $\mu^2$, respectively.
\end{enumerate}
We call the pairs of tuples $((a_v)_v,(a_f)_f)$ the \emph{(vertex and edge) parameters associated to the equivalence} $\mu^1\approx \mu^2$ (or $\mu^1\sim\mu^2$), and we denote the set of these parameters by $P(\mu^1,\mu^2)$. We will also use the notation $P_V(\mu^1,\mu^2)$ (resp. $P_E(\mu^1,\mu^2)$) to designate the set of vertex (resp. edge) parameters.

If instead $\mu^1, \mu^2\colon (\BB, v_0)\to (\AA, u_0)$ are morphisms of connected pointed graphs of groups, then they are \emph{$\sim$-equivalent}, written $\mu^1\sim \mu^2$, if the above holds with the additional condition that $a_{v_0} = 1$. We slightly abuse this definition by considering that if $\mu^1 \sim \mu^2$, then also $\mu^1 \approx \mu^2$ (whereas it is the unpointed versions of $\mu^1$ and $\mu^2$ which are $\approx$-equivalent). We then define the set of parameters for $\mu^1 \sim \mu^2$ to be the subset $P^*(\mu^1, \mu^2)$ of $P(\mu^1, \mu^2)$ consisting of the pairs of tuples $((a_v)_v,(a_f)_f) \in P(\mu^1, \mu^2)$ such that $a_{v_0} = 1$. The notation $P^*_V(\mu^1,\mu^2)$ (resp. $P^*_E(\mu^1,\mu^2)$) is defined in the natural way.
\end{defn}

\begin{rem}\label{rem: basics on equivalence}
    Item~\eqref{item: a v versus a f} in Definition~\ref{def: mor_sim} shows that there is a natural bijection between $P_V(\mu^1,\mu^2)$ and $P_E(\mu^1,\mu^2)$. Moreover, it is directly verified that $P(\mu^2,\mu^1) = P(\mu^1,\mu^2)\inv$. We also observe that, if $((1)_v, (1)_f) \in P(\mu^1,\mu^2)$, then $\mu^1 = \mu^2$.
\end{rem}

\paragraph{Notational convention}
We will encounter in the sequel many pairs of tuples of the same form as the elements of $P(\mu^1,\mu^2)$. Unless explicitly mentioned otherwise, we will write such tuples with a single symbol (say, $\bfa$, $\bfb$, etc) with the understanding that they stand for $\bfa = (\bfa_V,\bfa_E) = ((a_v)_v, (a_f)_f)$, $\bfb = (\bfb_V,\bfb_E) = ((b_v)_v, (b_f)_f)$, etc. Where appropriate, we will take products of such tuples, say, $\bfa\,\bfb = ((a_v\,b_v)_v, (a_f\,b_f)_f)$. If $\tau$ is a morphism defined on the appropriate graph of groups, we will also write $\tau(\bfa)$ for $(\tau(\bfa_V),\tau(\bfa_E)) = ((\tau(a_v))_v, (\tau(a_f))_f)$.

\medskip

\begin{rem}
In the language of \cite{kwm05}, two morphisms of graphs of groups (resp. pointed graphs of groups) are $\approx$-equivalent (resp. $\sim$-equivalent) if and only if one can be obtained from the other by performing (possibly infinitely many) $A0$- and $A1$-moves. In the case of morphisms of pointed graphs of groups, $A0$-moves at the basepoint are forbidden.
\end{rem}

A statement analogous to Condition~\eqref{item: locally at v} in Definition~\ref{def: mor_sim} holds for the edge parameters.

\begin{lem}
\label{rem:edge_equivalence}
If $\mu^1\approx \mu^2$ and $\bfa = ((a_v)_v, (a_f)_f) \in P(\mu^1, \mu^2)$, then, for each edge $f\in E(\gr{B})$, we have $\mu^2_f = \gamma_{a_f}\circ\mu^1_f$. Note that this includes the case where $\mu^1\sim \mu^2$.
\end{lem}
\begin{proof}
Let $f \in E(\gr B)$ and $v = o(f)$. We have
\begin{align*}
\alpha_e\circ\mu^2_f &= \gamma_{f_{\alpha}^2}\circ\mu^2_v\circ\alpha_f \\
					&= \gamma_{f_{\alpha}^2}\circ\gamma_{a_v}\circ\mu^1_v\circ\alpha_f \\
					&= \gamma_{f_{\alpha}^2}\circ\gamma_{a_v}\circ\gamma_{(f_{\alpha}^1)^{-1}}\circ\alpha_e\circ\mu^1_f\\
					&= \gamma_{(f_{\alpha}^1)^{-1}a_vf_{\alpha}^2}\circ \alpha_e\circ\mu^1_f\\
					&= \gamma_{\alpha_e(a_f)}\circ \alpha_e\circ\mu^1_f\\
					&= \alpha_e\circ \gamma_{a_f}\circ\mu^1_f.
\end{align*}
The result follows since $\alpha_e$ is injective.
\end{proof}

The following characterisations of the $\sim$- and $\approx$-equivalences will be useful in what follows.

\begin{prop}
\label{prop: change of data}
    Let $\mu^1, \mu^2\colon \BB\to\AA$ be morphisms of graphs of groups. Then $\mu^1\approx \mu^2$ if and only if for each vertex $v\in V(\gr{B})$, there is an element $a_v\in A_{\mu^2(v)}$ such that
    \[
    \mu^2(p) \sim_\AA a_{o(p)}^{-1}\,\mu^1(p)\, a_{t(p)}
    \]
    for all $\BB$-paths $p$. Furthermore, $(a_v)_v \in P_V(\mu_1, \mu_2)$.
    
    Let $\mu^1, \mu^2\colon (\BB, v_0)\to(\AA, u_0)$ be two morphisms of connected pointed graphs of groups. Then $\mu^1\sim \mu^2$ if and only if for each vertex $v\in V(\gr{B})$ there is an element $a_v\in A_{\mu^2(v)}$, with $a_{v_0} = 1$, such that
    \[
    \mu^2(p) \sim_\AA \mu^1(p)\, a_{t(p)}
    \]
    for all $\BB$-paths $p$ originating at $v_0$. Here too, $(a_v)_v \in P^*_V(\mu_1, \mu_2)$.
\end{prop}

\begin{proof}
Suppose that $\mu^1\approx \mu^2$ and let $a \in P(\mu^1, \mu^2)$. By Definition~\ref{def: mor_sim},
if $f \in E(\gr B)$ is an edge from $v_1$ to $v_2$ with image $e$ in $E(\gr{A})$, then its twisting elements satisfy
\begin{align*}
f_{\alpha}^2 &= a_{v_1}\inv\ f_{\alpha}^1\ \alpha_e(a_f) \text { and} \\
f_{\omega}^2 = (f\inv)_{\alpha}^2 &= a_{v_2}\inv\ (f\inv)_{\alpha}^1\ \alpha_{e\inv}(a_{f\inv}) = a_{v_2}\inv\ f_{\omega}^1\ \omega_{e}(a_f).
\end{align*}
It follows that
\begin{align*}
\mu^2(1,\,f,\,1) &= (f^2_\alpha,\,e,\,(f^2_\omega)\inv) \\
&= (a_{v_1}\inv\,f_{\alpha}^1\,\alpha_e(a_f),\ e,\ \omega_{e}(a_f)\inv\,(f^1_\omega)\inv\, a_{v_2}) \\
&\sim_\AA a_{v_1}\inv\,\mu^1(1,\,f,\,1)\, a_{v_2}.
\end{align*}
 Condition~\eqref{item: locally at v} in Definition~\ref{def: mor_sim} also shows that if $v \in V(\gr{B})$ and $b \in B_v$, then $\mu^2_v(b) = a_v\inv \, \mu^1_v(b) \, a_v$. Note that every $\BB$-path is a product of length 0 $\BB$-paths (i.e. elements of vertex groups) and of length 1 $\BB$-paths of the form $(1,\, f,\, 1)$. As a result, if $p$ is a $\BB$-path from vertex $v$ to vertex $w$, then $\mu^2 (p) \sim_{\AA} a_v\inv \, \mu^1(p) \, a_w$. Thus for both the unpointed and pointed cases, one implication is proved.
We now turn to the converse implications, handling first the unpointed case.

Suppose that elements $a_v\in A_{\mu^2(v)}$ exist for each vertex $v\in V(\gr{B})$, satisfying the hypothesis in the statement.
Let $p$ be a $\BB$-path from $v$ to $w$. The fact that $\mu^2(p) \sim_\AA a_{v}^{-1}\,\mu^1(p)\, a_w$ implies that $\mu^1(p)$ and $\mu^2(p)$ have the same underlying path: in particular, $\mu^1$ and $\mu^2$ coincide on each edge of $p$. Since every edge of $\gr{B}$ sits in a $\BB$-path, $\mu^1$ and $\mu^2$ coincide on $E(\gr{B})$ and, consequently, on $V(\gr{B})$.

Let now $f\in E(\gr{B})$ be an edge from $v$ to $w$ and let $b\in B_v$. Since $a_v^{-1}\,\mu^1(b, f, 1)\,a_w \sim_\AA \mu^2(b, f, 1)$, by Lemma \ref{lem: equivalent A-paths} there exists an element $a_f\in A_{[f]}$ such that 
\[
a_v^{-1}\,\mu^1_v(b)\,f_{\alpha}^1\,\alpha_{[f]}(a_f)  = \mu^2_v(b)\,f_{\alpha}^2.
\]
Rearranging, we have
\[
\mu^2_v(b) = a_v^{-1}\,\mu_v^1(b)\, f_{\alpha}^1\,\alpha_{[f]}(a_f)\,(f_{\alpha}^2)^{-1}.
\]
Setting $b = 1$ we see that $a_v = f^1_{\alpha}\,\alpha_{[f]}(a_f)\,(f^2_{\alpha})\inv$ and so
\[
\mu^2_v = \gamma_{a_v}\circ\mu^1_v.
\]
Since $v$ and $f$ were chosen arbitrarily, we have established that $\mu^1\approx\mu^2$ and that $\bfa = ((a_v)_v, (a_f)_f)$ is in $P(\mu^1, \mu^2)$.

The proof of the pointed version is almost identical. As above, $\mu^1$ and $\mu^2$ coincide on each edge of any $\BB$-path originating at $v_0$. Since $\BB$ is connected, every edge of $\gr{B}$ sits in such a $\BB$-path and hence $\mu^1$ and $\mu^2$ coincide on $E(\gr{B})$ and $V(\gr{B})$.

Let $f\in E(\gr{B})$ be an edge from $v$ to $w$, let $b\in B_v$ and let $p = (b, f, 1)\,p_1$ be a $\BB$-path from $v$ to $v_0$. Since $a_v^{-1}\,\mu^1((b, f, 1)\,p_1) \sim_\AA \mu^2((b, f, 1)\,p_1)$, there exists an element $a_f\in A_{[f]}$ such that 
\[
a_v^{-1}\,\mu^1_v(b)\,f_{\alpha}^1\,\alpha_{[f]}(a_f)  = \mu^2_v(b)\,f_{\alpha}^2.
\]
Rearranging, we have
\[
\mu^2_v(b) = a_v^{-1}\,\mu_v^1(b)\, f_{\alpha}^1\, \alpha_{[f]}(a_f)\,(f_{\alpha}^2)^{-1}.
\]
Setting $b = 1$ shows that $a_v = f^1_{\alpha}\, \alpha_{[f]}(a_f)\, (f^2_{\alpha})\inv$ and so
\[
\mu^2_v = \gamma_{a_v}\circ\mu^1_v.
\]
Since $\BB$ is connected and $v$ and $f$ were chosen arbitrarily, this shows that $\mu^1\sim \mu^2$, with $\bfa = ((a_v)_v, (a_f)_f) \in P^*(\mu^1, \mu^2)$.
\end{proof}

For further reference, we record the following observation.

\begin{rem}
\label{rem: B-paths equivalence}
    Within the proof of Proposition \ref{prop: change of data}, we showed that if $\mu^1\approx \mu^2$ and $\bfa \in P(\mu^1, \mu^2)$, then the following holds: for each $\BB$-path $p = (b_0, f_1, b_1, \ldots, f_n, b_n)$ visiting vertices $v_0$, $v_1$, \ldots, $v_n$, we have
    \begin{align*}
        (a_{v_0})^{-1}\mu^1_{v_0}(b_0)(f_1)^1_{\alpha}\alpha_{[f_1]}(a_{f_1}) &= \mu^2_{v_0}(b_0)(f_1)^2_{\alpha}\\
        \omega_{[f_i]}(a_{f_i})^{-1}((f_i)^1_{\omega})^{-1}\mu^1_{v_i}(b_i)(f_{i+1})^1_{\alpha}\alpha_{[f_{i+1}]}(a_{f_{i+1}}) &= ((f_i)^2_{\omega})^{-1}\mu^2_{v_i}(b_i)(f_{i+1})^2_{\alpha}\\
        \omega_{[f_n]}(a_{f_n})^{-1}((f_n)^1_{\omega})^{-1}\mu^1_{v_n}(b_n)(a_{v_n}) &= ((f_n)^2_{\omega})^{-1}\mu^2_{v_n}(b_n).
    \end{align*}
\end{rem}

Proposition~\ref{prop: change of data} leads to the proof that the $\sim$ and $\approx$ relations are preserved under the composition of morphisms.

\begin{cor}
\label{cor: composition}
    Let $\mu^1, \mu^2\colon \CC\to \BB$ and $\nu^1, \nu^2\colon \BB\to \AA$ be morphisms of graphs of groups. If $\mu^1\approx\mu^2$ and $\nu^1\approx\nu^2$, then $\nu^1\circ\mu^1\approx\nu^2\circ\mu^2$.

    The same holds for pointed morphisms, replacing $\approx$ with $\sim$.
\end{cor}

\begin{proof}
    By Proposition~\ref{prop: change of data}, there exist tuples $(b_w)_{w\in V(\gr C)} \in P_V(\mu^1, \mu^2)$ and $(a_v)_{v \in V(\gr B)} \in P_V(\nu^1, \nu^2)$, such that the following holds: for every $\CC$-path $q$ from $w_1$ to $w_2$, we have $\mu^1(q) \sim_\BB  b_{w_1}^{-1}\ \mu^2(q)\ b_{w_2}$; and for every $\BB$-path $p$ from $v_1$ to $v_2$, we have $\nu^1(p) \sim_\AA  a_{v_1}^{-1}\ \nu^2(q)\ a_{v_2}$.
    As a result,
    $$\nu^1(\mu^1(q)) \sim_\AA a_{\mu^1(w_1)}\inv\ \nu^2(b_{w_1}\inv\ \mu^2(q)\ b_{w_2})\ a_{\mu^1(w_2)} = (\nu^2(b_{w_1})\,a_{\mu^1(w_1)})\inv\ \nu^2(\mu^2(q))\ (\nu^2(b_{w_2})\,a_{\mu^1(w_2)}).$$
    Proposition~\ref{prop: change of data} again shows that $\nu^1\circ\mu^1 \approx \nu^2\circ\mu^2$.
    
    If $\mu^1$, $\mu^2$, $\nu^1$ and $\nu^2$ are morphisms of pointed graphs of groups between $(\BB, v_0)$, $(\CC, w_0)$ and $(\AA,u_0)$, and if $\mu^1\sim \mu^2$ and $\nu^1 \sim \nu^2$, then the tuples $(a_v)_v$ and $(b_w)_w$ can be chosen such that $a_{v_0}$ and $b_{w_0}$ are trivial, and hence $\nu^2_{v_0}(b_{w_0})\, a_{v_0} = 1$. It follows that $\nu^2\circ\mu^2 \sim \nu^1\circ\mu^1$.
\end{proof}

Corollary~\ref{cor: composition}, in turn, allows us to define our two categories.

\begin{defn}
    The \emph{category of graphs of groups}, denoted by $\GrGp$, is the category with:
    \begin{itemize}
        \item \textbf{Objects:} graphs of groups.
        \item \textbf{Arrows:} morphisms of graphs of groups, considered up to $\approx$-equivalence.
    \end{itemize}
\end{defn}

\begin{defn}
    The \emph{category of pointed graphs of groups}, denoted by $\GrGp^*$, is the category with:
    \begin{itemize}
        \item \textbf{Objects:} pointed graphs of groups with connected underlying graph.
        \item \textbf{Arrows:} pointed morphisms of graphs of groups, considered up to $\sim$-equivalence.
    \end{itemize}
\end{defn}

The categories $\GrGp$ and $\GrGp^*$ are well-defined since the composition of $\approx$-equivalent (resp.~$\sim$-equivalent) morphisms yields $\approx$-equivalent (resp. $\sim$-equivalent) morphisms by Corollary \ref{cor: composition}. Let us point out that, in $\GrGp$, graphs of groups $\AA$ and $\BB$ are \emph{isomorphic} if there exist morphisms $\mu\colon \AA \to \BB$ and $\nu\colon \BB\to \AA$ such that $\mu\circ\nu$ is $\approx$-equivalent to the identity morphism $1_{\BB}\colon\BB\to\BB$ and $\nu\circ\mu$ is $\approx$-equivalent to the identity morphism $1_{\AA}\colon \AA\to\AA$. The analogous statement, replacing $\approx$ by $\sim$, holds for $\GrGp^*$.

\subsection{Parameters for the $\approx$- and $\sim$-equivalences}\label{ssec: parameters}

In this section we record some facts about the parameters for $\approx$- and $\sim$-equivalences of morphisms of graphs of groups. These statements will be indispensable for the technical Section \ref{sec: lifts}.
We start with investigating the behaviour of tuples of parameters under certain operations: the transitivity of the equivalence relations $\sim$ and $\approx$ and the composition of morphisms.

\begin{lem}
\label{lem: double equivalence elements}
    Let $\mu, \tau, \nu\colon \BB\to \AA$ be morphisms of graphs of groups with $\mu\approx \tau$ and $\tau\approx \nu$. If $\bfb \in P(\mu, \tau)$ and $\bfb' \in P(\tau, \nu)$, then $\bfb\, \bfb' \in P(\mu, \nu)$.

    Similarly, if $\mu, \tau, \nu$ are morphisms of pointed graphs of groups from $(\BB,v_0)$ to $(\AA,u_0)$, with $\mu\sim \tau$ and $\tau\sim \nu$, and if $\bfb \in P^*(\mu, \tau)$ and $\bfb' \in P^*(\tau, \nu)$, then $\bfb\, \bfb' \in P^*(\mu, \nu)$.
\end{lem}

\begin{proof}
    Let $\bfb\in P(\mu, \tau)$ and $\bfb' \in P(\tau, \nu)$. By Proposition \ref{prop: change of data}, for every $\BB$-path $p$ from $v$ to $v'$, we have $\tau(p) \sim_{\AA} b_v\inv\, \mu(p)\, b_{v'}$ and $\nu(p) \sim_{\AA} {b'}_v\inv\, \tau(p)\ b'_{v'}$. In particular, $\nu(p) \sim_{\AA} (b_v\,b'_v)\inv\, \tau(p)\ b_{v'}\,b'_{v'}$. It follows (again by Proposition \ref{prop: change of data}) that $(\bfb\,\bfb')_V = (b_v)_v\, (b'_v)_v \in P_V(\mu, \nu)$. Moreover, by Definition~\ref{def: mor_sim}, for each edge $f\in E(\gr{B})$ with origin $v$, we have $b_v = f_{\alpha}^{\mu}\,\alpha_e(b_f)\,(f_{\alpha}^{\tau})^{-1}$ and $b_v' = f_{\alpha}^{\tau}\,\alpha_e(b_f')\,(f_{\alpha}^{\nu})^{-1}$. It follows that $b_v\,b_v' = f_{\alpha}^{\mu}\,\alpha_e(b_f\,b_f')\,(f_{\alpha}^{\nu})^{-1}$. Thus $\bfb\,\bfb' \in P(\mu, \nu)$.
    
    Suppose now that $\mu, \nu, \tau$ are morphisms of pointed graphs of groups $\mu\sim \tau$ and $\tau\sim \nu$, and that $\bfb \in P^*(\mu, \tau)$ and $\bfb' \in P^*(\tau, \nu)$. Then $\bfb \in P(\mu, \tau)$ and $\bfb' \in P(\tau, \nu)$, and the $\approx$-statement shows that $\bfb\, \bfb'\in P(\mu, \tau)$. Since $b_{v_0} = b'_{v_0} = 1$, we also have $(b\,b')_{v_0} = 1$, so $\bfb\, \bfb'\in P^*(\mu, \tau)$.
\end{proof}

\begin{lem}
\label{cor: composition associated elements}
    Let $\mu, \nu\colon \BB\to \AA$ and $\tau\colon \CC\to \BB$ be morphisms. If $\mu\approx\nu$, then $\mu\circ\tau \approx \nu\circ\tau$. Moreover, if $\bfb \in P(\mu, \nu)$, then $\bfc \in P(\mu\circ\tau, \nu\circ\tau)$, where
    \[
    c_w = b_{\tau(w)} \quad\textrm{and}\quad c_{g} = b_{\tau(g)}
    \]
    for all $w\in V(\gr C)$ and $g\in E(\gr C)$.
\end{lem}

\begin{proof}
    If $p$ is a $\BB$-path from $v_1$ to $v_2$, we have $\nu(p) \sim_{\AA} b_{v_1}^{-1}\, \mu(p)\, b_{v_2}$ by Proposition \ref{prop: change of data}. Thus, if $p$ is a $\CC$-path from $w_1$ to $w_2$, we have $\nu(\tau(p))\sim_{\AA} b_{\tau(w_1)}^{-1}\, \mu(\tau(p))\, b_{\tau(w_2)}$. By Proposition \ref{prop: change of data} again, this implies that $\mu\circ\tau \approx \nu\circ\tau$ and $(c_w)_w = (b_{\tau(w)})_w \in P_V(\mu\circ\tau, \nu\circ\tau)$.

    By Definition \ref{def: mor_sim}, the corresponding tuple $(c_g)_{g\in E(\gr C)} \in P_E(\mu\circ\tau, \nu\circ\tau)$ satisfies the following: if $g\in E(\gr C)$, $f = \tau(g)$, $e = \mu(f)$, $w = o(g)$ and $v = \mu(w)$, then
    \[
    \alpha_{e}(c_{g}) = (\mu(g_{\alpha})\,f_{\alpha}^{\mu})^{-1}\, c_w\, (\nu(g_{\alpha})\,f_{\alpha}^{\nu})  = (\mu(g_{\alpha})\,f_{\alpha}^{\mu})^{-1}\, b_v\, (\nu(g_{\alpha})\,f_{\alpha}^{\nu}).
    \]
    By Remark \ref{rem: B-paths equivalence}, we have $b_v = \mu(b)\, f_{\alpha}^{\mu}\, \alpha_e(b_f)\, (f_{\alpha}^{\nu})^{-1}\, \nu(b)^{-1}$ for all $b\in B_v$. Thus, $\mu(g_{\alpha})\,b_v = b_v\, \nu(g_{\alpha})$, it follows that 
    \[
    \alpha_{e}(c_{g}) = (f_{\alpha}^{\mu})\inv\, b_v\, f_{\alpha}^{\mu} = \alpha_e(b_f),
    \]
    and hence $c_g = b_f = b_{\tau(g)}$.
\end{proof}

\begin{lem}\label{cor: composition associated elements 2}
       Let $\mu, \nu\colon \CC\to \BB$ and $\tau\colon \BB\to \AA$ be morphisms. If $\mu\approx\nu$, then $\tau\circ\mu\approx\tau\circ\nu$. Moreover, if $\bfc \in P(\mu, \nu)$, then $\tau(\bfc) \in P(\tau\circ\mu, \tau\circ\nu)$.

       Conversely, if $\mu, \nu$ agree on the underlying graph, $\tau\circ\mu\approx\tau\circ\nu$ and $\tau(\bfc)\in P(\tau\circ\mu, \tau\circ\nu)$ for some $\bfc\in \prod_wB_w\times\prod_gB_g$, then $\mu\approx \nu$ and $\bfc\in P(\mu, \nu)$.
\end{lem}

\begin{proof}
    We need to check that $\tau(\bfc)$ satisfies the three properties of Definition \ref{def: mor_sim}. Items (\ref{item: inverse of a f}) and (\ref{item: locally at v}) from Definition (\ref{def: mor_sim}) are immediate. For Item (\ref{item: a v versus a f}), using the fact that $c_w = g_{\alpha}^{\mu}\alpha_f(c_g)(g_{\alpha}^{\nu})^{-1}$ for each vertex $w\in V(\gr{C})$ and edge $g\in E(\gr{C})$ with $o(g) = w$ and $f = \mu(g) = \nu(g)$, we see that
    \begin{align*}
    \tau(c_w) &= \tau(g_{\alpha}^{\mu})\,(\tau(\alpha_f(c_g)))\, (\tau(g_{\alpha}^{\nu}))^{-1}\\
                &= \tau(g_{\alpha}^{\mu})\, f_{\alpha}^{\tau}\, \alpha_e(\tau(c_g))\, (\tau(g_{\alpha}^{\nu})\, f_{\alpha}^{\tau})^{-1}
    \end{align*}
    where here $e = \tau(f)$. Since $\tau(g_{\alpha}^{\mu})f_{\alpha}^{\tau} = g_{\alpha}^{\tau\circ\mu}$ and $\tau(g_{\alpha}^{\nu})f_{\alpha}^{\tau} = g_{\alpha}^{\tau\circ\nu}$, this implies that Item (\ref{item: a v versus a f}) holds and so $\tau(\bfc)\in P(\tau\circ\mu, \tau\circ\nu)$ as claimed.

    For the converse, Items (\ref{item: inverse of a f}) and (\ref{item: locally at v}) from Definition \ref{def: mor_sim} are immediate. For Item (\ref{item: a v versus a f}) we see that $\tau(c_w) = \tau(g_{\alpha}^{\mu})\, f_{\alpha}^{\tau}\, \alpha_e(\tau(c_g))\, (\tau(g_{\alpha}^{\nu})\, f_{\alpha}^{\tau})^{-1}$ for each $w\in V(\gr{C})$, $g\in E(\gr{C})$ with $o(g) = w$ and where $f$ and $e$ are the $\mu$ and $\tau\circ\mu$ images of $g$. Then we have $\tau(c_w) = \tau(g_{\alpha}^{\mu})\,(\tau(\alpha_f(c_g)))\, (\tau(g_{\alpha}^{\nu}))^{-1}$ and so $c_w = g_{\alpha}^{\mu}\, \alpha_f(c_g)\, (g_{\alpha}^{\nu})^{-1}$ as required.
\end{proof}

The parameters of the trivial equivalence $\mu\approx \mu$ play a special role in the sequel.

\begin{defn}\label{def: centralizer new}
    Let $\mu\colon \BB\to \AA$ be a morphism of graphs of groups with $\gr{B}$ connected. The \emph{centraliser of $\mu$} is defined to be $C(\mu) = P(\mu, \mu)$. We will also write $C_V(\mu)$ for $P_V(\mu, \mu)$ and $C_E(\mu)$ for $P_E(\mu, \mu)$.
\end{defn}

Recall that, if $G$ is a group and $H$ is a subgroup, $C_{G}(H)$ denotes the usual centraliser of $H$ in $G$.

\begin{lem}\label{lem: centralizer characterization}
    Let $\mu\colon \BB\to \AA$ be a morphism of graphs of groups with $\gr{B}$ connected and let $\bfz \in C(\mu)$. The following holds.
    \begin{enumerate}[(1)]
        \item\label{item: centralizing property of C(mu)} For each vertex $v\in V(\gr B)$ and edge $f\in E(\gr B)$, we have $z_v\in C_{A_{\mu(v)}}(\mu_v(B_v))$ and $z_f\in C_{A_{\mu(f)}}(\mu_f(B_f))$.
        
        \item\label{item: path property of C(mu)} For each $\BB$-path $p$ from $v_0$ to $v_1$, we have $z_{v_0}\, \mu(p) \sim_{\AA} \, \mu(p) \, z_{v_1}$.
    \end{enumerate}
\end{lem}

\begin{proof}
    The vertex parameter property of Item~\eqref{item: centralizing property of C(mu)} and Item~\eqref{item: path property of C(mu)} follow directly from Definition~\ref{def: mor_sim}. Let us now consider an edge $f$ of $\gr B$, with origin $v$, and let $e = \mu(f)$. By Definition~\ref{def: mor_sim}~\eqref{item: a v versus a f}, for each $b\in B_v$, we have 
    $$a_v = f_\alpha\, \alpha_e(z_f)\, f_\alpha\inv \enspace\textrm{and hence}\enspace f_\alpha\, \alpha_e(z_f)\, f_\alpha\inv\, \mu_v(b) = \mu_v(b)\, f_\alpha\, \alpha_e(z_f)\, f_\alpha\inv.$$
    Now let $b\in B_f$. Recall that $\alpha_e(\mu_f(b)) = f_\alpha\inv\, \mu_v(\alpha_f(b))\, f_\alpha$ (see Definition~\ref{def: morphism}). It follows that
    \begin{align*}
    (f_\alpha\, \alpha_e(z_f)\, f_\alpha\inv)\, \mu_v(\alpha_f(b)) &= \mu_v(\alpha_f(b))\, (f_\alpha\, \alpha_e(z_f)\, f_\alpha\inv) \\
    (f_\alpha\, \alpha_e(z_f)\, f_\alpha\inv)\, (f_\alpha\, \alpha_e(\mu_f(b))\, f_\alpha\inv) &=  (f_\alpha\, \alpha_e(\mu_f(b))\, f_\alpha\inv)\, (f_\alpha\, \alpha_e(z_f)\, f_\alpha\inv) \\
    f_\alpha\, \alpha_e(z_f)\, \alpha_e(\mu_f(b))\, f_\alpha\inv &= f_\alpha\, \alpha_e(\mu_f(b))\, \alpha_e(z_f)\, f_\alpha\inv\textrm{ which yields}\\
    \alpha_e(z_f\, \mu_f(b)) &= \alpha_e(\mu_f(b)\, z_f).
    \end{align*}
Since $\alpha_e$ is injective, it follows that $z_f \in C_{A_e}(\mu_f(B_f))$.
\end{proof}

We record the following technical lemma.

\begin{lem}\label{lem: trivial element of centraliser}
    Let $\mu\colon \BB \to \AA$ be a morphism of graphs of groups such that $\gr B$ is connected. Let $\bfz \in C(\mu)$ and $v_0\in V(\gr B)$. If $z_{v_0} = 1$, then $z_v = 1$ for every $v\in V(\gr B)$ and $z_f = 1$ for every $f\in E(\gr{B})$.
\end{lem}

\begin{proof}
    Let $f\in E(\gr B)$ be an edge such that $o(f) = v_0$ and let $e = \mu(f)$ and $v = t(f)$. By Definition~\ref{def: mor_sim}, we have $z_{v_0} = f_\alpha\, \alpha_e(z_f)\, f_\alpha\inv$, and hence, $z_f = 1$. It follows that $z_{f\inv} = 1$ as well, and so $z_v = f_\omega\, \omega_e(z_{f\inv})\, f_\omega\inv = 1$. Thus, if $v$ is at distance 1 from $v_0$ in $\gr B$, then $z_v = 1$. We conclude using the connectedness of $\gr B$.
\end{proof}

We can now characterise tuples of parameters for the same $\approx$-equivalence.

\begin{lem}
\label{lem: almost uniqueness of associated elements}
    Let $\mu, \nu\colon \BB\to \AA$ be morphisms of graphs of groups with $\gr{B}$ connected such that $\mu\approx\nu$. If $\bfb, \bfb' \in P(\mu,\nu)$, there exists $\bfz\in C(\mu)$ such that $\bfb' = \bfz\, \bfb$. In other words, $P(\mu,\nu)$ is a right coset of the subgroup $C(\mu)$ in $\prod_u A_u \times \prod_e A_e$.
\end{lem}
\begin{proof}
    The equivalence $\mu \approx \mu$ is witnessed by the sequence $\mu \approx \nu \approx \mu$. We have $((b'_v)_v, (b'_f)_f \in P_(\mu, \nu)$ and $((b\inv_v)_v, (b\inv_f)_f \in P(\nu, \mu)$. We conclude by Lemma~\ref{lem: double equivalence elements}.
\end{proof}

This yields a uniqueness statement for $\sim$-equivalences.

\begin{cor}
    Let $\mu, \nu\colon (\BB, v_0)\to (\AA, u_0)$ be morphisms of pointed graphs of groups, where $\BB$ is connected and such that $\mu\sim\nu$. Then $P^*(\mu,\nu)$ is a singleton.
\end{cor}

\begin{proof}
    Suppose that $\bfb, \bfb' \in P^*(\mu,\nu)$. By Lemma~\ref{lem: almost uniqueness of associated elements}, $\bfz = \bfb'\,\bfb\inv \in C(\mu)$. Since $z_{v_0} = b'_{v_0}\,b_{v_0}\inv = 1$, Lemma~\ref{lem: trivial element of centraliser} shows that $\bfz = ((1)_v, (1)_f))$, and hence $\bfb = \bfb'$.
\end{proof}

We also need a characterisation of the morphisms $\nu$ such that $\tau\circ\mu \approx \tau\circ\nu$, in terms of the morphisms $\mu\colon \CC\to \BB$ and $\tau\colon \BB\to \AA$.

\begin{defn}\label{defn: twist set}
Let $\mu\colon \CC\to \BB$ and $\tau\colon \BB\to \AA$ be two morphisms of graphs of groups. The \emph{twist set} $T(\mu, \tau)$ of $\mu$ with respect to $\tau$ is the collection of pairs of tuples $\bft = ((t_w)_{w\in V(\gr{C})}, (t_g)_{g\in E(\gr{C})})$ satisfying the following. If $w \in V(\gr C)$, $v = \mu(w)$ and $u = \tau(v)$, we have :
\begin{enumerate}[(1)]
    \item\label{item: vertex in twist} $t_w \in A_u$ satisfies $\tau(\mu(C_w))^{t_w}\leqslant \tau(B_{v})$;
    \item\label{item: edge in twist} if $g\in E(\gr C)$ with origin $w$, then $t_g\in A_e$ satisfies $\tau(\mu(C_g))^{t_g}\leqslant \tau(B_f)$ and $t_g = t_{g^{-1}}$,
    \item\label{item: vertex and edge in twist} and $t_w^{-1}\, \tau(g_\alpha)\, (f_{\alpha}) \, \alpha_e(t_g) \, (f_{\alpha})^{-1} \in \tau(B_v)$, where $f = \mu(g)$.
\end{enumerate}

When $\mu\colon (\CC, w_0)\to (\BB, v_0)$ and $\tau\colon (\BB, v_0)\to (\AA, u_0)$ are pointed morphisms, the \emph{twist set at $w_0$} is the subset $T^{w_0}(\mu, \tau)$ of $T(\mu, \tau)$ of tuples $\bft$ with $t_{w_0} = 1$.

We will write $T_V(\mu, \tau)$ and $T_E(\mu, \tau)$ for the vertex and edge tuples, respectively.
\end{defn}

\begin{defn}
\label{defn: twist_morphism}
Let $\mu\colon \CC\to \BB$ and $\tau\colon \BB\to \AA$ be two morphisms of graphs of groups and let $\bft = ((t_w)_w, (t_g)_{g})\in T(\mu, \tau)$. Define the morphism $\mu^\bft\colon \CC\to \BB$ associated to $t$ as follows:
\begin{enumerate}[(1)]
\item $\mu^{\bft}$ agrees with $\mu$ on the underlying graph $\gr{C}$.
\item $\mu^\bft_w = \varphi_w\circ\mu_w$ for each $w\in V(\gr{C})$, where $\varphi_w\colon \mu(C_w)\to B_{\mu(w)}$ is the homomorphism induced by conjugation by $t_w$. More precisely, $\varphi_w = (\tau^{-1}_{\mu(w)}\circ\gamma_{t_w}\circ\tau_{\mu(w)})\mid_{\mu(C_w)}$ which is well-defined by the definition of $T(\mu, \tau)$. 
\item $\mu^\bft_g = \varphi_g\circ\mu_g$ for each $g\in E(\gr{C})$ where $\varphi_g\colon \mu(C_g)\to B_{\mu(g)}$ is the homomorphism induced by conjugation by $t_g$. More precisely, $\varphi_g = (\tau^{-1}_{\mu(g)}\circ\gamma_{t_g}\circ\tau_{\mu(g)})\mid_{\mu(C_g)}$ which is well-defined by the definition of $T(\mu, \tau)$.
\item For each $g\in E(\gr{C})$, let $f = \mu(g)$, $e = \tau(f)$, $w = o(g)$, $v = \mu(w)$ and $u = \tau(v)$. The $\alpha$-twisting element of $g$ for $\mu^\bft$ is $g^\bft_{\alpha} = \tau_v^{-1}(t_w^{-1}\, \tau(g_\alpha)\, (f_{\alpha}) \, \alpha_e(t_g) \, (f_{\alpha})^{-1})$, which is well-defined by assumption.
\end{enumerate}
\end{defn}

We begin by verifying that the morphisms $\mu^\bft$ from Definition \ref{defn: twist_morphism} satisfy $\tau\circ\mu\approx\tau\circ\mu^\bft$.

\begin{lem}
\label{lem: twist_approx}
    If $\mu\colon \CC\to \BB$ and $\tau\colon \BB\to \AA$ are morphisms of graphs of groups, then, for all $\bft\in T(\mu, \tau)$, we have $\tau\circ\mu\approx\tau\circ\mu^\bft$ and $\bft\in P(\tau\circ\mu, \tau\circ\mu^\bft)$.

    If $\mu\colon (\CC, w_0)\to (\BB, v_0)$ and $\tau\colon (\BB, v_0)\to (\AA, u_0)$ are morphisms of connected graphs of groups, then, for all $\bft\in T^{w_0}(\mu, \tau)$, we have $\tau\circ\mu\sim\tau\circ\mu^\bft$ and $\bft\in P^{*}(\tau\circ\mu, \tau\circ\mu^\bft)$.
\end{lem}

\begin{proof}
    Let $\bft = ((t_w)_w, (t_g)_g)$.
    We begin with the unpointed case. We have $\tau_{\mu(w)}\circ\mu^\bft_w = \gamma_{t_w}\circ\tau_{\mu(w)}\circ\mu_w$ for each $w\in V(\gr{C})$ by definition. Similarly, $\tau_{\mu(g)}\circ\mu^\bft_g = \gamma_{t_g}\circ\tau_{\mu(g)}\circ\mu_g$ for each $g\in E(\gr{C})$. And we also have $t_g = t_{g^{-1}}$ for each edge $g\in E(\gr C)$. We need to verify Item~\eqref{item: a v versus a f} in Definition~\ref{def: mor_sim}. Let $g$ be an edge in $\gr C$, $f$ and $e$ its $\mu$- and $\tau\circ\mu$-images, $w$ its origin, and $v$ and $u$ the $\mu$- and $\tau\circ\mu$-images of $w$. Recall that the $\alpha$-twisting elements for $g$ relative to $\tau\circ\mu$, $\mu^\bft$ and $\tau\circ\mu^\bft$ are
    \begin{itemize}
        \item $\tau(g_{\alpha})\,f_{\alpha}$ for $\tau\circ\mu$;
        \item $g^\bft_\alpha = \tau_v^{-1}(t_w^{-1}\, \tau(g_\alpha)\, (f_{\alpha}) \, \alpha_e(t_g) \, (f_{\alpha})^{-1})$ for $\mu^\bft$ and
        \item $\tau(g^\bft_\alpha)f_\alpha =  t_w^{-1}\, \tau(g_\alpha)\, (f_{\alpha}) \, \alpha_e(t_g)$ for $\tau\circ\mu^\bft$.
    \end{itemize}
    The verification of Item~\eqref{item: a v versus a f} is now straightforward. It follows that $\tau\circ\mu\approx\tau\circ\mu^\bft$ with the claimed parameters.

    For the pointed case, we note that we just proved that for all $\bft\in T^{w_0}(\mu, \tau)$ we have $\tau\circ\mu\approx\tau\circ\mu^\bft$ with parameters $\bft$. Since $t_{w_0} = 1$ by definition, we actually have $\tau\circ\mu\sim\tau\circ\mu^\bft$ and $\bft\in P^{*}(\tau\circ\mu, \tau\circ\mu^\bft)$.
\end{proof}

\begin{lem}
\label{lem: twist lemma}
Let $\mu\colon \CC\to \BB$ and $\tau\colon \BB\to \AA$ be morphisms of connected graphs of groups. 

Let $\nu\colon \CC\to \BB$ be a morphism that agrees with $\mu$ on the underlying graph and so that $\tau\circ\mu\approx\tau\circ\nu$. Then $P(\tau\circ\mu, \tau\circ\nu)$ is contained in $T(\mu, \tau)$ and $\nu = \mu^\bft$ for every $\bft \in P(\tau\circ\mu, \tau\circ\nu)$.

If $\bft_1, \bft_2\in T(\mu, \tau)$, then $\mu^{\bft_1}\approx\mu^{\bft_2}$ if and only if there are elements $\bfz\in C(\tau\circ\mu)$ and $\bfb\in \prod_w B_{\mu(w)} \times\prod_g B_{\mu(g)}$ so that $\bfz\,\bft_1\,\tau(\bfb) = \bft_2$. Moreover, any such $\bfb$ is a tuple of parameters for the equivalence $\mu^{\bft_1}\approx\mu^{\bft_2}$, that is, $\bfb \in P(\mu^{\bft_1}, \mu^{\bft_2})$.
\end{lem}

\begin{rem}\label{rem: a bijection with double cosets}
    Before proceeding with the proof, we remark that Lemmas \ref{lem: twist_approx} and \ref{lem: twist lemma} establish a bijection:
    \begin{align*}
    C(\tau\circ\mu)\Big\backslash\, T(\mu, \tau)\,\Big/\left(\prod_{w}\tau(B_{\mu(w)})\times\prod_g B_{\mu(g)}\right) &\enspace\longrightarrow\enspace \Big\{\nu\colon \CC\to \BB \mid \text{$\mu, \nu$ agree on $\gr{C}$ and}\,\, \tau\circ\mu\approx\tau\circ\nu\Big\}/\approx\\
    (C(\tau\circ\mu))\, \bft\,\left(\prod_{w}\tau(B_{\mu(w)})\times\prod_g B_{\mu(g)}\right) &\enspace\longmapsto\enspace \mu^\bft
    \end{align*}
    where $\mu^\bft$ is as in Definition \ref{defn: twist_morphism}.
\end{rem}

\begin{proof}[Proof of Lemma \ref{lem: twist lemma}]
Suppose that $\nu\colon \CC\to \BB$ is a morphism such that $\nu$ agrees with $\mu$ on $\gr C$ and $\tau\circ\mu\approx\tau\circ\nu$. We need to show that if $\bft  \in P(\tau\circ\mu, \tau\circ\nu)$, then $\bft\in T(\mu, \tau)$ and that $\nu = \mu^\bft$. By definition of $\approx$-equivalence, $\gamma_{t_w}\circ \tau_{\mu(w)} \circ \mu_w = \tau_{\nu(w)} \circ \nu_w$, so that $(\tau(\mu(C_w))^{t_w} = \tau(\nu(C_w)) \leq \tau(B_{\mu(w)})$, verifying Item~\eqref{item: vertex in twist} in Definition~\ref{defn: twist set}. Item~\eqref{item: edge in twist} is verified similarly. Since, $\mu$ and $\nu$ agree on the underlying graph $\gr{C}$, let $\mu(g) = \nu(g) =f$. To verify Item~\eqref{item: vertex and edge in twist}, we note that the $\alpha$-twisting element of an edge $g$ of $\gr C$ is $\tau(g_\alpha^\mu)\, f_\alpha$
relative to $\tau\circ\mu$, and  $\tau(g_\alpha^\nu)\, f_\alpha$ relative to $\tau\circ\nu$. By definition of parameters, if $w = o(g)$ and $e = \tau\circ\mu(g)$, we have
$$t_w = \tau(g_\alpha^\mu)\, f_\alpha\, \alpha_e(t_g)\, (\tau(g_\alpha^\nu)\,f_\alpha)\inv.$$
Using the fact that $\mu, \nu$ agree on the underlying graph, it follows that 
$$t_w\inv\, \tau(g_\alpha^\mu)\, f_\alpha\, \alpha_e(t_g)\, (f_\alpha)\inv = \tau(g_\alpha^\nu)\,f_\alpha\, \alpha_e(t_g)\inv\, (\tau(g_\alpha^\mu)\,f_\alpha)\inv\, \tau(g_\alpha^\mu)\, f_\alpha\, \alpha_e(t_g)\, (f_\alpha)\inv = \tau(g_\alpha^\nu).$$
Thus $\bft \in T(\mu, \tau)$.
Now we need to verify that $\nu = \mu^\bft$.

Note that $\bft$ sits in the parameter sets $P(\tau\circ\mu, \tau\circ\nu)$ and $P(\tau\circ\mu, \tau\circ \mu^\bft)$. As a result, $((1)_w, (1)_g) = \bft\,\bft\inv \in P(\tau\circ\mu^\bft, \tau\circ\nu)$ and it follows that $\tau\circ\mu^\bft = \tau\circ\nu$, see Remark~\ref{rem: basics on equivalence}. In particular, for each $w \in V(\gr C)$ with $\mu$-image $v$, the composite maps $\tau_v\circ\mu^\bft_w$ and $\tau_v\circ\nu_w$ coincide on $C_w$. Since $\tau_v$ is a monomorphism, it follows that $\mu^\bft_w = \nu_w$. By the same token, and using Lemma~\ref{rem:edge_equivalence}, we find that $\mu^\bft_g = \nu_g$ for any edge $g\in E(\gr C)$. Moreover the $\alpha$-twisting elements of $g$ for $\tau\circ\mu^\bft$ and $\tau\circ\nu$ are, respectively, $\tau(g^\bft_\alpha)\,f_\alpha$ and $\tau(g^\nu_\alpha)\,f_\alpha$ (where $f = \mu(g)$), it follows that $g^\nu_\alpha = g^\bft_\alpha$, and hence $\mu^\bft = \nu$.

Now let $\bft_1, \bft_2 \in T(\mu,\tau)$ be such that $\mu^{\bft_1} \approx \mu^{\bft_2}$. Let $\bfb \in P(\mu^{\bft_1}, \mu^{\bft_2})$. In particular $b_w \in B_{\mu(w)}$. Then $\tau(\bfb) \in P(\tau\circ\mu^{\bft_1}, \tau\circ\mu^{\bft_2})$ by Lemma~\ref{cor: composition associated elements 2}. By Lemma \ref{lem: twist_approx} we have that $\bft_i \in P(\tau\circ\mu, \tau\circ\mu^{\bft_i})$ for $i = 1, 2$. It then follows from Lemma~\ref{lem: double equivalence elements} that $\bft_1\,\tau(\bfb)\,\bft_2^{-1} \in P(\tau\circ\mu, \tau\circ\mu) = C(\tau\circ\mu)$, and hence $\bft_2 \in C(\tau\circ\mu)\,\bft_1\,\tau(\bfb)$, as claimed. Conversely, if $\bfb\in \prod_wB_{\mu(w)}\times\prod_gB_{\mu(g)}$ is so that $\bft_1\,\tau(\bfb)\,\bft_2^{-1} = \bfz\in C(\tau\circ\mu)$, then we see that $\tau(\bfb) \in P(\tau\circ\mu^{\bft_1}, \tau\circ\mu^{\bft_2})$ by Lemma \ref{lem: double equivalence elements}. Then by Lemma \ref{cor: composition associated elements 2} we see that $\mu^{\bft_1}\approx\mu^{\bft_2}$ with parameters $\bfb$.
\end{proof}

\begin{lem}
\label{lem: twist lemma pointed}
Let $\mu\colon (\CC, w_0)\to (\BB, v_0)$ and $\tau\colon (\BB, v_0)\to (\AA, u_0)$ be morphisms of connected graphs of groups.

Let $\nu\colon \CC\to \BB$ be a morphism that agrees with $\mu$ on the underlying graph and so that $\tau\circ\mu\sim\tau\circ\nu$. Then $P^*(\tau\circ\mu, \tau\circ\nu)$ is contained in $T^{w_0}(\mu, \tau)$ and, if $\bft\in T^{w_0}(\mu, \tau)$, then $\nu = \mu^\bft$.

If $\bft_1, \bft_2\in T^{w_0}(\mu, \tau)$, then $\mu^{\bft_1}\sim\mu^{\bft_2}$ if and only if $\bft_2^{-1}\,\bft_1 = \tau(\bfb)$ for some $\bfb\in \prod_wB_{\mu(w)}\times\prod_gB_{\mu(g)}$. Moreover, in that case, $\bfb \in P^{*}(\mu^{\bft_1}, \mu^{\bft_2})$.
\end{lem}

\begin{rem}
    As in Remark~\ref{rem: a bijection with double cosets}, we note that Lemmas \ref{lem: twist_approx} and \ref{lem: twist lemma pointed} establish a bijection:
    \begin{align*}
        T^{w_0}(\mu, \tau)\,\Big/\left(\prod_{w}\tau(B_{\mu(w)})\times\prod_g B_{\mu(g)}\right) &\to \Big\{\nu\colon (\CC, w_0)\to (\BB, v_0) \mid \text{$\mu, \nu$ agree on $\gr{C}$ and}\,\, \tau\circ\mu\sim\tau\circ\nu\Big\}/\sim\\
        \bft\,\left(\prod_{w}\tau(B_{\mu(w)})\times\prod_g B_{\mu(g)}\right) &\mapsto \mu^\bft
    \end{align*}
    where $\mu^\bft$ is as in Definition \ref{defn: twist_morphism}.
\end{rem}

\begin{proof}[Proof of Lemma \ref{lem: twist lemma pointed}]
    Consider a morphism $\nu$ which coincides with $\mu$ on $\gr C$ and such that $\tau\circ\mu \sim \tau\circ\nu$, and let $\bft \in P^*(\tau\circ\mu, \tau\circ\nu)$. Then by Lemma \ref{lem: twist lemma} we have that $\nu = \mu^\bft$. Since $t_{w_0} = 1$, we have that $\bft\in T^{w_0}(\mu, \tau)$ as required.
    
    Now let $\bft_1, \bft_2 \in T^{w_0}(\mu,\tau)$ be such that $\mu^{\bft_1} \sim \mu^{\bft_2}$, and let $\bfb \in P^*(\mu^{\bft_1}, \mu^{\bft_2})$. In particular $b_w \in B_{\mu(w)}$ with $b_{w_0} = 1$. Then $\tau(\bfb) \in P^*(\tau\circ\mu^{\bft_1}, \tau\circ\mu^{\bft_2})$ by Lemma~\ref{cor: composition associated elements 2}. By Lemma \ref{lem: twist_approx} we have that $\bft_i \in P^*(\tau\circ\mu, \tau\circ\mu^{\bft_i})$ for $i = 1, 2$. It then follows from Lemma~\ref{lem: double equivalence elements} that $\bft_1\,\tau(\bfb)\,\bft_2^{-1} \in P^*(\tau\circ\mu, \tau\circ\mu) = C(\tau\circ\mu)$. Lemma \ref{lem: trivial element of centraliser} then shows that $\bft_1\,\tau(\bfb)\,\bft_2^{-1} = ((1)_w, (1)_g)$. Thus, $\bft_2^{-1}\,\bft_1 = \tau(\bfb)$ as claimed. Conversely, if $\bfb\in \prod_wB_{\mu(w)}\times\prod_gB_{\mu(g)}$ is so that $\bft_2^{-1}\,\bft_1 = \tau(\bfb)$, then $b_{w_0} = 1$ and we see that $\tau(\bfb) \in P^*(\tau\circ\mu^{\bft_1}, \tau\circ\mu^{\bft_2})$ by Lemma \ref{lem: double equivalence elements}. Then by Lemma \ref{cor: composition associated elements 2} we see that $\mu^{\bft_1}\sim\mu^{\bft_2}$ with parameters~$\bfb$.
\end{proof}

We may also obtain improvements to Lemma \ref{lem: twist lemma} and Lemma \ref{lem: twist lemma pointed} under the additional assumption that $\tau$ is an immersion.

\begin{lem}
\label{lem: graph_determined}
    Let $\mu\colon \CC\to \BB$ and $\tau\colon \BB\to \AA$ be morphisms of connected graphs of groups, with $\tau$ an immersion.

    Let $\nu\colon \CC\to \BB$ be a morphism so that $\mu(w_0) = \nu(w_0)$ for some vertex $w_0\in V(\gr{C})$ and $\tau\circ\mu\approx\tau\circ\nu$. Let $\bft \in P(\tau\circ\mu, \tau\circ\nu)$ so that there exist $\bfz \in C(\tau\circ\mu)$ and $b_{w_0}\in B_{\mu(w_0)}$ satisfying $t_{w_0} = z_{w_0}\,\tau(b_{w_0})$. Then $\mu$ and $\nu$ agree on the underlying graph $\gr{C}$ and $\bft\in \bfz\,\left(\prod_w\tau(B_{\mu(w)})\times\prod_g\tau(B_{\mu(g)})\right)$.

    In particular, if $\mu\colon (\CC, w_0)\to (\BB, v_0)$ and $\tau\colon (\BB, v_0)\to (\AA, u_0)$ are pointed morphisms of connected graphs of groups with $\tau$ an immersion, and if $\nu\colon (\CC, w_0)\to (\BB, v_0)$ is a pointed morphism, then $\tau\circ\mu\sim\tau\circ\nu$ implies $\mu\sim\nu$.
\end{lem}

\begin{proof}
    We show that for all edges $g$ leading out from $w_0$, with terminal vertex $w$, we have $\mu(g) = \nu(g)$, $t_g\in z_g\,\tau(B_{\mu(g)})$ and $t_w\in d_w\,\tau(B_{\mu(w)})$. From here it will follow, by connectedness, that $\mu, \nu$ agree on the whole graph and that $\bft\in \bfz\,\left(\prod_w\tau(B_{\mu(w)})\times\prod_g\tau(B_{\mu(g)})\right)$.
    
    Let $v_0 = \mu(w_0) = \nu(w_0)$, and let $g\in E(\gr{C})$ be an edge from $w_0$ to $w$, $\mu(g) = f_1$ and $\nu(g) = f_2$. Since $\tau\circ\mu \approx \tau\circ\nu$, $\tau$ maps $f_1$ and $f_2$ to the same edge $e\in E(\gr A)$. By Lemma~\ref{lem: double equivalence elements}, we have $\bfz\inv\,\bft \in P(\tau\circ\mu, \tau\circ\nu)$ and hence
    \begin{equation}\label{eq: immersion case}
    \tau(b_{w_0}) = z\inv_{w_0}\,t_{w_0} = \tau(g_{\alpha}^{\mu})\, (f_1)^{\tau}_{\alpha}\, \alpha_{e}(z_g\inv\,t_g)\, \left(\tau(g_{\alpha}^{\nu})\, (f_2)^{\tau}_{\alpha}\right)^{-1}.\end{equation}
    Rearranging we see that $\tau(B_{v_0})\,(f_1)_{\alpha}^\tau\,\alpha_e(A_e) = \tau(B_{v_0})\, (f_2)_{\alpha}^{\tau}\, \alpha_e(A_e)$. Since $\tau$ is an immersion, this implies that $f_1 = f_2$, which we now write $f$. Equation~\eqref{eq: immersion case} now yields
    $$\tau(b_{w_0}) = z\inv_{w_0}\,t_{w_0} = \tau(g_{\alpha}^{\mu})\, f^{\tau}_{\alpha}\, \alpha_{e}(z_g\inv\,t_g)\, \left(\tau(g_{\alpha}^{\nu})\, f^{\tau}_{\alpha}\right)^{-1},$$
    which shows that $f^{\tau}_{\alpha}\, \alpha_{e}(z_g\inv\,t_g)\, (f^{\tau}_{\alpha})\inv \in \tau(B_{v_0})$.

    By the definition of immersions, this shows that there exists $b_g\in B_g$ such that $\alpha_e(z_g\inv\,t_g) = \tau(b_g)$. Now, considering the edge $g\inv$, from $w$ to $w_0$, we have 
    $$z_w\inv\,t_w = \tau((g\inv)_\alpha^\mu)\,(f\inv)_\alpha^\tau\,\alpha_e(z_g\inv\,t_g)\,((f\inv)_\alpha^\tau)\inv\, \tau((g\inv)_\alpha^\nu)\inv$$
    and hence $(f\inv)_\alpha^\tau\,\alpha_e(z_g\inv\,t_g)\,((f\inv)_\alpha^\tau)\inv \in \tau(B_w)$. Using, again, the definition of immersions, we see that implies that $z_w\inv\,t_w \in \tau(B_w)$. This completes the proof of the statement relative to $\approx$-equivalence.

    For the statement relative to $\sim$-equivalence, we note that $z_{w_0} = 1$ which trivially lies in $\tau(B_{v_0})$ and so $\mu\sim\nu$ by Lemma \ref{lem: twist lemma pointed}.
\end{proof}

\section{Products, pullbacks, and intersections}\label{sec: pullbacks}

In this section, we develop the essential construction
aimed at describing intersections in the fundamental group of a graph of groups and, fundamentally, describing pullbacks in the appropriate category.

We first recall the well-known concept of pullback in a general category: given two morphisms $\mu^B\colon B \to A$ and $\mu^C\colon C \to A$ with a common codomain, the \emph{pullback of $\mu^B$ and $\mu^C$},
denoted by $\mu^B \times_A \mu^C = (P,\rho^B,\rho^C)$,
consists of an object, usually denoted by $P = B \times_A C$, equipped with two morphisms $\rho^B \colon P \to B$ and $\rho^C \colon P \to C$ such that $\mu^B \circ \rho^B = \mu^C \circ \rho^C$, and such that for any other pair of morphisms $\sigma^B \colon P' \to B$ and $\sigma^C \colon P' \to C$ satisfying $\mu^B \circ \sigma^B = \mu^C \circ \sigma^C$, there exists a unique morphism $\sigma \colon P' \to P$ such that $\sigma^B = \rho^B \circ \sigma$ and $\sigma^C = \rho^C \circ \sigma$.
This situation is summarised in Figure~\ref{fig:enter-label}:

\begin{figure}[H]
    \centering
    \begin{tikzcd}
        P' \arrow[rd, "\exists !\,\sigma", dashed] \arrow[rdd, "\sigma^B"', bend right] \arrow[rrd, "\sigma^C", bend left] \\
        & P \arrow[d, "\rho^B"'] \arrow[r, "\rho^C"] & C \arrow[d, "\mu^C"]\\
        & B \arrow[r, "\mu^B"'] & A
    \end{tikzcd}
    \caption{The pullback $(P,\rho^B,\rho^C)$ of $\mu^B$ and $\mu^C$}
    \label{fig:enter-label}
\end{figure}
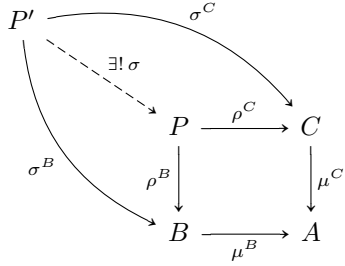

It is clear that the pullback, if it exists, is unique up to isomorphism in the considered category. We sometimes abuse terminology (tacitly assuming the morphisms $\mu^{B}$ and $\mu^{C}$) and refer to the pullback object $B \times_A C$ as the \emph{pullback of $B$ and $C$}.

Two important and classical examples of pullbacks are the following. In the category of sets (respectively, groups), if $\mu^B \colon B \to A$ and $\mu^C \colon C \to A$ are maps (respectively, morphisms), the pullback is the triple $(B \times_A C,\rho^B, \rho^C)$ where $B \times_A C$ is the set (respectively, subgroup) consisting of all pairs $(b,c)$ in the Cartesian product $B \times C$ such that $\mu^B(b) = \mu^C(c)$, and $\rho^B$ and $\rho^C$ are the natural projections.

In the category of (directed) graphs, the pullback of the morphisms
$\mu^{\gr{B}} \colon \gr{B} \to \gr{A}$ and $\mu^{\gr{C}} \colon \gr{C} \to \gr{A}$ is the triple $(\gr{B} \times_{\gr{A}} \gr{C},\rho^{\gr{B}},\rho^ {\gr{C}})$, where $\gr{B} \times_{\gr{A}} \gr{C}$ is the graph with
vertex and edge sets the set-categoric pullbacks $V(\gr B)\times_{V(\gr A)} V(\gr C)$ and $E(\gr B)\times_{E(\gr A)} E(\gr C)$, with the natural incidence functions, and where $\rho^B$ and $\rho^C$ are the natural projections.

In this section, we give an explicit construction for pullbacks in the category $\GrGp^*$ (Section~\ref{sec: final pointed pullbacks}). Pullbacks can be used to represent intersections of subgroups, see Corollary~\ref{cor: intersection of subgroups}. We give a detailed example of an application of these results in Section~\ref{sec: worked out example}: we show that Baumslag--Solitar groups not of the form $\bs(1,n)$ do not have the finitely generated intersection property\footnote{That is, they contain finitely generated subgroups whose intersection is not finitely generated.} by explicitly computing a pullback and showing that its fundamental group is not finitely generated. Since the groups of the form $\bs(1,n)$ are well-known to have the finitely generated intersection property (see \cite{mol68}), this completes the characterization of the Baumslag--Solitar groups with the \fgip, recovering a result of Paramantzoglou in~\cite{pa12}.

Pullbacks in $\GrGp$ are discussed in Section~\ref{sec: final unpointed pullback}. We show that they do not always exist in Example \ref{ex: pullbacks may not exist}, and in case they do exist, we derive information about their construction (see Theorem \ref{thm: pullbacks_exist_sometimes}).

These results depend on the construction of a particular product of $\BB$ and $\CC$ fibered over $\AA$, denoted by $\BB\wtimes_\AA\CC$, which depends on the morphisms $\mu^B$ and $\mu^C$, and which we introduce in Section~\ref{ssec: A-product}.

\subsection{A product construction}\label{ssec: A-product}

Let $\AA$, $\BB$, $\CC$ be
graphs of groups, with respective underlying graphs $\gr{A}$, $\gr{B}$, $\gr{C}$, and let ${\mu^B\colon \BB \to \AA}$ and $\mu^C\colon \CC\to \AA$ be morphisms of graphs of groups. In this section, we define a graph of groups called the \mbox{\emph{$\AA$-product}} of $\BB$ and $\CC$ (more precisely, of $\mu^B$ and $\mu^C$), denoted by 
\[
\DD = \BB \wtimes_\AA \CC
\]
and two morphisms $\rho^B\colon \DD\to\BB$ and $\rho^C\colon \DD\to\CC$.

We start with the description of its underlying graph.

\begin{defn}[Underlying graph]
\label{def: underlying graph A-product}
The underlying graph $\gr{D}$ of $\DD = \BB\wtimes_\AA\CC$ is defined as follows. For each $(v,w) \in V(\gr{B} \times_\gr{A} \gr{C})$ and $(f,g)\in E(\gr{B} \times_\gr{A} \gr{C})$, if $u = [v] = [w]$ and $e = [f] = [g]$, we let
\begin{equation}
\begin{aligned}  \label{eq: V(D) and E(D)}
    V_{v,w}(\gr{D}) &= \dblcoset{\mu_v^B(B_v)}{A_u}{\mu_w^C(C_w)} = \{ \mu_v^B(B_v)\,a\,\mu_w^C(C_w) \mid a\in A_u \},\\
    E_{f,g}(\gr{D}) &=  \dblcoset{\mu_f^B(B_f)}{A_e}{\mu_g^C(C_g)} = \{ \mu_f^B(B_f)\,a\,\mu_g^C(C_g) \mid a\in A_e \},
\end{aligned}
\end{equation}
be the \emph{$(v, w)$-vertices} and \emph{$(f, g)$-edges} of $\gr{D}$, respectively. We then define the vertex set and  edge set of $\gr{D}$ to be the disjoint unions:
\begin{align*}
V(\gr{D}) &= \bigsqcup_{(v,w)\,\in\, V(\gr{B} \times_\gr{A} \gr{C})} V_{v,w}(\gr{D}) \quad\textrm{ and}\\
E(\gr{D}) &= \bigsqcup_{(f,g)\,\in\, E(\gr{B} \times_\gr{A} \gr{C})} E_{f,g}(\gr{D}).
\end{align*}
The incidence maps are the following. If $e \in E(\gr{A})$ and $(f,g) \in E(\gr{B} \times_\gr{A} \gr{C})$ is an edge from $(v,w)$ to $(v',w')$ such that $[f] = [g] = e$ and $a\in A_e$, then the origin and terminal vertices of the $(f,g)$-edge $\mu_f^B(B_f)\,a\,\mu_g^C(C_g)$ are given by
\begin{equation} \label{eq: incidence in D}
\begin{aligned} 
o\left(\mu_f^B(B_f)\,a\,\mu_g^C(C_g)\right) &= \mu_v^B(B_v)\,(f_{\alpha}\,\alpha_e(a)\,g_{\alpha}^{-1})\,\mu_w^C(C_w),\\
t\left(\mu_f^B(B_f)\,a\,\mu_g^C(C_g)\right) &= \mu_{v'}^B(B_{v'})\,(f_{\omega}\,\omega_e(a)\,g_\omega^{-1})\,\mu_{w'}^C(C_{w'}).
\end{aligned}
\end{equation}
It is directly verified (using Equations~\eqref{eq: twisted commutation alpha} and~\eqref{eq: twisted commutation omega} in the definition of morphisms) that, if $h \in E(\gr D)$, the definition of $o(h)$ and $t(h)$ does not depend on the representative $a$ of the $(\mu_v^B(B_v),\mu_w^C(C_w))$-double coset which defines $h$.

The constructed graph $\gr{D}$ is naturally equipped with graph morphisms $\rho^B\colon \gr{D} \to \gr{B}$ and $\rho^C\colon \gr{D} \to \gr{C}$:
if $x \in V_{v,w}(\gr{D})$, then $\rho^B(x) = v$ and $\rho^C(x) = w$; and if $h\in E_{f,g}(\gr{D})$, then $\rho^B(h) = f$ and $\rho^C(h) = g$.

If $\mu^B$ and $\mu^C$ are in fact morphisms of pointed graphs of groups, say, $\mu^B\colon (\BB,v_0)\to (\AA, u_0)$ and $\mu^C\colon (\CC, w_0)\to (\AA, u_0)$, then $\gr{D}$ is made into a pointed graph by selecting the following $(v_0, w_0)$-vertex as the basepoint 
\begin{equation}\label{eq: basepoint in pointed}
x_0 = \mu_{v_0}^B(B_{v_0})\,\mu^C_{w_0}(C_{w_0}).
\end{equation}
The morphisms $\rho^B\colon (\gr{D}, x_0)\to (\gr{B}, v_0)$ and $\rho^C\colon (\gr{D}, x_0)\to (\gr{C}, w_0)$ are then morphisms of pointed graphs.
\end{defn}

\begin{rem}\label{rk: multiple double cosets}
The vertex and edge sets of $\gr D$ are defined as disjoint unions. Formally, this means that vertices and edges in $\gr{D}$ can be considered as objects of the form $(v,w,\mu_v^B(B_v)\,a\,\mu_w^C(C_w))$ with $a\in A_u$, and
$(f,g,\mu_f^B(B_f)\,a\,\mu_g^C(C_g))$ with $a\in A_e$, respectively; to lighten notation
we abuse language and assume that
the double cosets specified in \eqref{eq: V(D) and E(D)} carry (say, in the subscripts) the information about the vertices and edges in $\gr{B} \times_{\gr{A}} \gr{C}$ from which they originate. In particular, as sets of double cosets, $E_{f,g}(\gr{D})$ and $E_{f\inv,g\inv}(\gr{D})$ are always equal (since $\mu^B_{f^{-1}} = \mu^B_{f}$, $\mu^C_{g^{-1}} = \mu^C_{g}$, $B_{f\inv} = B_f$, and $C_{g\inv} = C_g$), but they describe different sets of edges in $\gr{D}$ (since $f\ne f\inv$ and $g\ne g\inv$). For instance, every double coset of the form $\mu_f^B(B_f)\,a\,\mu_g^C(C_g)$ ($a\in A_e$) occurs at least twice in $E(\gr{D})$: as an $(f,g)$-edge and as an $(f\inv,g\inv)$-edge. The fixed point-free involution on $E(\gr{D})$ is the following: if $a\in A_e$ and $h = \mu_f^B(B_f)\,a\,\mu_g^C(C_g)$ is an element of $E_{f,g}$, then $h\inv$ is the element $\mu_{f\inv}^B(B_{f\inv})\,a\,\mu_{g\inv}^C(C_{g\inv}) = \mu_f^B(B_f)\,a\,\mu_g^C(C_g)$ in $E_{f\inv,g\inv}(\gr{D})$.
\end{rem}

Next, we fix representatives of the double cosets associated to the vertices and edges of $\gr{D}$. If $x \in V(\gr{D})$, we choose an element $d_x$ of the double coset $x$; and if $h \in E(\gr{D})$, we choose an element $d_h$ of the double coset $h$. The collection of choices is denoted $\bfd = ((d_x)_x, (d_h)_h)$. If $\mu^B$ and $\mu^C$ are morphisms of \emph{pointed} graphs of groups, the representative for the basepoint $x_0$ (see Equation~\eqref{eq: basepoint in pointed})
is always taken to be $1$, or rather, $d_{x_0} = 1$. There are no such restrictions when working with unpointed graphs of groups, that is, in $\GrGp$. Note that if $x$ is a $(v,w)$-vertex and $u = [v] = [w] \in V(\gr{A})$, then $d_x \in A_u$. Similarly, if $h$ is an $(f,g)$-edge and $e = [f] = [g] \in E(\gr{A})$, then $d_h \in A_e$.
As observed in Remark~\ref{rk: multiple double cosets}, $h$ and $h\inv$ are, set-theoretically, the same double coset, and we choose $d_{h^{-1}}$ to be equal to $d_h$.

In particular, if $h\in E_{f,g}(\gr{D})$, $e = [f] = [g] \in E(\gr{A})$ and $x = o(h) \in V_{v,w}(\gr{D})$, then $h = \mu_f^B(B_f)\,d_h\,\mu_g^C(C_g)$, and, in view of the definition of the incidence relations \eqref{eq: incidence in D} in $\gr{D}$, we have
$x = \mu_v^B(B_v)\,(f_\alpha\,\alpha_e(d_h)\,g_\alpha\inv)\,\mu_w^C(C_w)$. It follows that there exist $b_h\in B_v$ and $c_h\in C_w$ (not necessarily unique) such that
\begin{equation}\label{eq: old x tilde}
d_x = \mu_v^B(b_h)\,f_{\alpha}\,\alpha_e(d_h)\,g_{\alpha}^{-1}\,\mu_w^C(c_h)^{-1}. 
\end{equation}
Our Definition~\ref{defn: projection morphisms} (of the graph of groups data) below a priori depends on our choices above. However, in Lemma \ref{lem: uniqueness up to isomorphism} we show that different choices lead to isomorphic graphs of groups.

Recall that if $\phi\colon B\to A$ and $\psi\colon C \to A$ are group morphisms, then the pullback of $\phi$ and $\psi$ is the group
\[
B\times_AC = \{(b, c) \in B\times C \mid \phi(b) = \psi(c)\} \leqslant B\times C
\]
together with the natural projection morphisms $p^{B}$ and $p^{C}$ induced by the projection morphisms from $B\times C$ to $B$ and $C$. We adopt the following notational convention.

\paragraph{Notational convention}
Let $\phi\colon B\to A$ and $\psi\colon C \to A$ be group morphisms and let $a \in A$. When $\phi$ and $\psi$ are understood, in order to lighten notation, we denote by $\twpb BaAC$ the pullback of $\gamma_a\circ\phi$ and $\psi$, that is, the set of pairs $(b,c) \in B\times C$ such that $a\inv\,\phi(b)\, a = \psi(c)$. 

\begin{lem}\label{lem: DD is well defined new}
Let $h \in E(\gr D)$ and let $e$, $f$, $g$ be its images in $E(\gr A)$, $E(\gr B)$, $E(\gr C)$, respectively. Let also $x = o(h)$ and $u$, $v$, $w$ be its images in $V(\gr A)$, $V(\gr B)$, $V(\gr C)$. Then the morphism
\[
(\gamma_{b_h^{-1}}\circ\alpha_f, \gamma_{c_h^{-1}}\circ\alpha_g)\colon \twpb{B_f}{d_h}{A_e}{C_g} \to B_v\times C_w
\]
has image in $\twpb{B_v}{d_x}{A_u}{C_w}$.
\end{lem}

\begin{proof}
Let $z = (b, c)\in \twpb{B_f}{d_h}{A_e}{C_g}$. By definition, $\mu^B_f(b)^{d_h} = \mu_g^C(c)$ and we have 
\[
(\gamma_{b_h^{-1}}\circ\alpha_f, \gamma_{c_h^{-1}}\circ\alpha_g)(b, c) = (b_h\, \alpha_f(b)\, b_h^{-1}, c_h\,\alpha_g(c)\,c_h\inv).
\]
Since
\begin{align*}
\mu^B(b_h\, \alpha_f(b)\, b_h^{-1})^{d_x} &= \left(\mu^B(b_h)\, f_{\alpha}\, \alpha_e(\mu_f^B(b))\, f_{\alpha}^{-1}\, \mu^B(b_h)^{-1}\right)^{d_x} \textrm{ by definition of morphisms}\\
            &= \left(\mu^B(b_h)\, f_{\alpha}\, \alpha_e(d_h\, \mu_g^C(c)\, d_h^{-1})\, f_{\alpha}^{-1}\, \mu^B(b_h)^{-1}\right)^{d_x}\\
            &= \mu^C(c_h)\, g_{\alpha}\, \alpha_e(\mu^C_g(c))\, g_{\alpha}^{-1}\, \mu^C(c_h)^{-1} \textrm{ by Equation~\eqref{eq: old x tilde}}\\
            &= \mu^C(c_h\, \alpha_g(c)\, c_h^{-1})
\end{align*}
we certainly have that the image of $(\gamma_{b_h^{-1}}\circ\alpha_f, \gamma_{c_h^{-1}}\circ\alpha_g)$ lies in $\twpb{B_v}{d_x}{A_u}{C_w}$.
\end{proof}

\begin{defn}[Graph of groups data and projection morphisms]\label{defn: projection morphisms}
The graph of groups 
\[\DD = (\gr{D}, \{D_x\}, \{D_h\}, \{\alpha_h, \omega_h\})\] and its projections $\rho^B$ and $\rho^C$ to $\BB$ and $\CC$ are
given by the following data:
\begin{enumerate}[(1)]
\item The underlying graph $\gr{D}$ and the projection graph morphisms $\gr{D}\to \gr{B}$, $\gr{D}\to \gr{C}$ specified in Definition~\ref{def: underlying graph A-product}.
\item If $x\in V_{v, w}(\gr{D})$, the corresponding vertex group is
    \[
    D_x = \twpb{B_v}{d_x}{A_u}{C_w},
    \]
    a subgroup of $B_v \times C_w$. We then let $\rho_x^B\colon D_x\to B_v$ and $\rho_x^C\colon D_x\to C_w$ be the first and second coordinate projections.
\item Similarly, if $h\in E_{f, g}(\gr{D})$, the corresponding edge group is 
    \[
    D_h = \twpb{B_f}{d_h}{A_e}{C_g},
    \]
    and we let $\rho_h^B\colon D_h\to B_f$ and $\rho_h^C\colon D_h\to C_g$ be the first and second coordinate projections.
    
    \item If $h\in E_{f, g}(\gr{D})$ with origin $x\in V_{v, w}(\gr{D})$ and target $x'\in V_{v', w'}(\gr{D})$, then the corresponding edge maps $\alpha_h \colon D_h \to D_{o(h)} = D_x$ and $\omega_h \colon D_h \to D_{t(h)} = D_{x'}$ are given by
    \begin{align*}
        \alpha_h &= (\gamma_{b_h^{-1}}\circ\alpha_f, \gamma_{c_h^{-1}}\circ\alpha_g)\\
        \omega_h &= (\gamma_{b_{h^{-1}}^{-1}}\circ\omega_f, \gamma_{c_{h^{-1}}^{-1}}\circ\omega_g).
    \end{align*}
    Note that this is well-defined by Lemma \ref{lem: DD is well defined new}.
    \item The twisting elements for an edge $h$ relative to $\rho^B$ and $\rho^C$ are
    $$h_\alpha^B = b_h\textrm{ and }h_\alpha^C = c_h,\textrm{ and hence }h_\omega^B = b_{h\inv}\textrm{ and }h_\omega^C = c_{h\inv}.$$
\end{enumerate}
\end{defn}

A priori, our definition of $\DD$ and the maps $\rho^B$ and $\rho^C$ depend on the choices of representatives $d_x$, $d_h$, $b_h$ and $c_h$, for $x\in V(\gr D)$ and $h\in E(\gr D)$. We first examine the impact of different choices.

\begin{lem}\label{lem: uniqueness up to isomorphism}
Let $\mu^B\colon \BB\to \AA$ and $\mu^C\colon \CC\to \AA$ be morphisms of graphs of groups. Let $\DD$ and $\DD'$ be $\AA$-products over $\mu^B$ and $\mu^C$, with projection morphisms $\rho^B$ and $\rho^C$, and $\rho'^B$ and $\rho'^C$, respectively, determined by different choices of representatives. Then there exists an isomorphism (in the category $\GrGp$) $\sigma\colon \DD\to \DD'$ of graphs of groups such that $\rho'^B\circ\sigma\approx\rho^B$ and $\rho'^C\circ\sigma\approx\rho^C$.

The same statement holds in the case of morphisms of pointed graphs of groups, replacing isomorphisms in $\GrGp$ by isomorphisms in $\GrGp^*$ and each occurrence of $\approx$ with $\sim$.
\end{lem}

\begin{proof}
    Note first that $\gr D = \gr D'$. Let $d_x$, $d_h$, $b_h$ and $c_h$ (resp. $d_x$, ${d_h}'$, $b_h'$ and $c_h'$) be the representatives, for all $x\in V(\gr D)$ and $h \in E(\gr D)$, yielding $\DD$, $\rho^B$ and $\rho^C$ (resp. $\DD'$, $\rho'^B$ and $\rho'^C$). Since $\mu_v^B(B_v)\,d_x\,\mu_w^C(C_w) = \mu_v^B(B_v)\,d_x'\,\mu_w^C(C_w)$ and $\mu_f^B(B_f)\,d_h\,\mu_g^C(C_g) = \mu_f^B(B_f)\,d_h'\,\mu_g^C(C_g)$, there exist elements $\hat{b}_x\in B_v$, $\hat{c}_x\in C_w$, $\hat{b}_h\in B_f$ and $\hat{c}_h\in C_g$ such that
    \begin{align*}
        d_x' &= \mu^B(\hat{b}_x)\,d_x\,\mu^C(\hat{c}_x)^{-1},\\
        d_h' &= \mu^B(\hat{b}_h)\,d_h\,\mu^C(\hat{c}_h)^{-1}.
    \end{align*}
    Note that if $d_x = d_x' = 1$, then we may take $\hat{b}_x = 1$, $\hat{c}_x = 1$.
    We thus have
    \begin{align*}
    \twpb{B_v}{d_x}{A_u}{C_w} &= \{(b, c)\in B\times C \mid d_x^{-1}\mu^B(b)d_x = \mu^C(c)\}\\
                                    &= \{(b, c)\in B\times C \mid d_x'^{-1}\mu^B(\hat{b}_xb\hat{b}_x^{-1})d_x' = \mu^C(\hat{c}_xc\hat{c}_x^{-1})\}.
    \end{align*}
    In view of the above, we  may define  isomorphisms
    \begin{align*}
    \sigma_x\colon D_x = \twpb{B_v}{d_x}{A_u}{C_w} &\to \twpb{B_v}{d_x'}{A_u}{C_w} = D'_x\\
    (b, c) &\mapsto (\hat{b}_xb\hat{b}_x^{-1}, \hat{c}_xc\hat{c}_x^{-1})
    \end{align*}
    on the vertex groups. An identical argument yields isomorphisms:
    \begin{align*}
    \sigma_h\colon D_h = \twpb{B_f}{d_h}{A_e}{C_g} &\to \twpb{B_f}{d_h'}{A_e}{C_g} = D'_h\\
    (b, c) &\mapsto (\hat{b}_hb\hat{b}_h^{-1}, \hat{c}_hc\hat{c}_h^{-1})
    \end{align*}
    on the edge groups. We note that
    \begin{align*}
    \alpha_h\circ\sigma_h(b, c) &= \alpha_h(\hat{b}_hb\hat{b}_h^{-1}, \hat{c}_hc\hat{c}_h^{-1})\\
                                &= (b_h'\alpha_f(\hat{b}_hb\hat{b}_h^{-1})b_h'^{-1}, c_h'\alpha_g(\hat{c}_hc\hat{c}_h^{-1})c_h'^{-1}),\\
    \sigma_x\circ\alpha_h(b, c) &= \sigma_x(b_h\alpha_f(b)b_h^{-1}, c_h\alpha_g(c)c_h^{-1})\\
                                &= (\hat{b}_xb_h\alpha_f(b)b_h^{-1}\hat{b}_x^{-1}, \hat{c}_xc_h\alpha_g(c)c_h^{-1}\hat{c}_x^{-1}),
    \end{align*}
    and so, letting
    \[
    (\hat{h}_{\alpha}^B, \hat{h}_{\alpha}^C) = (\hat{b}_xb_h\alpha_f(\hat{b}_h)^{-1}b_h'^{-1}, \hat{c}_xc_h\alpha_g(\hat{c}_h)^{-1}c_h'^{-1}),
    \]
    we have
    \[
    \alpha_h\circ\sigma_h = \gamma_{(\hat{h}^B_{\alpha}, \hat{h}_{\alpha}^C)}\circ\sigma_x\circ\alpha_h.
    \]
    Hence, letting $(\hat{h}_{\alpha}^B, \hat{h}_{\alpha}^C)$ be the twisting element for $\sigma$ associated with the edge $h$, we see that $\sigma$ is a well-defined morphism of graphs of groups. Since it is an isomorphism on the underlying graph and each vertex and edge morphism is an isomorphism, $\sigma$ is an isomorphism.
    
    A direct computation, together with Proposition \ref{prop: change of data}, shows that $\rho'^B\circ\sigma \approx\rho^B$ and $\widehat\bfb \in P(\rho'^B\circ\sigma, \rho^B)$, and that $\rho'^C\circ\sigma\approx\rho^C$ and $\widehat\bfc \in P(\rho'^C\circ\sigma, \rho^C)$.  In the pointed case, since $\hat{b}_x = 1$ and $\hat{c}_x = 1$ when $x$ is the basepoint, Proposition \ref{prop: change of data} implies that $\rho'^B\circ\sigma\sim \rho^B$ and $\rho'^C\circ\sigma\sim \rho^C$.
\end{proof}

Lemma~\ref{lem: uniqueness up to isomorphism} states that different choices of representatives $d_x$, $d_h$, $b_h$ and $c_h$ ($x\in V(\gr D)$, $h \in E(\gr D)$) yield triples $(\DD, \rho^B, \rho^C)$ which are isomorphic in $\GrGp$. As a result, we will freely write $\BB\wtimes_\AA\CC$, $\rho^B$ and $\rho^C$ to denote any one of these isomorphic objects or arrows of $\GrGp$. Similarly, if $\mu^B$ and $\mu^C$ are morphisms of pointed graphs of groups, different choices of representatives (satisfying $d_{x_0} = 1$) yield tuples $(\DD, x_0, \rho^B, \rho^C)$ which are isomorphic in $\GrGp^*$, and we will freely write $(\BB\wtimes_\AA\CC, x_0)$, $\rho^B$ and $\rho^C$ to denote any one of these isomorphic objects or arrows of $\GrGp^*$.

We now verify that $\mu^B$ and $\mu^C$ can be replaced by equivalent morphisms without affecting the $\AA$-product.

\begin{lem}
\label{lem: uniqueness up to isomorphism_2}
Let $\mu^B, \nu^B\colon \BB\to \AA$ and $\mu^C, \nu^C\colon \CC\to \AA$ be morphisms of graphs of groups such that $\mu^B\approx\nu^B$ and $\mu^C\approx\nu^C$. If $\DD$ and $\DD'$ are the $\AA$-products over $\mu^B$ and $\mu^C$, and over $\nu^B$ and $\nu^C$, respectively, then there exists an isomorphism $\sigma\colon \DD\to\DD'$ of graphs of groups such that $\nu^B\circ\rho'^B\circ\sigma\approx \mu^B\circ\rho^B$ and $\nu^C\circ\rho'^C\circ\sigma\approx\mu^C\circ\rho^C$. See the commutative diagram in Figure~\ref{fig: AA-product for L36}.

If $\mu^B, \nu^B, \mu^C, \nu^C$ are morphisms of pointed graphs of groups such that $\mu^B\sim\nu^B$, $\mu^C\sim\nu^C$ and if $(\DD, x_0)$ and $(\DD', x'_0)$ are the pointed $\AA$-products over $\mu^B$ and $\mu^C$, and over $\nu^B$ and $\nu^C$, respectively, then there exists a pointed isomorphism $\sigma\colon (\DD, x_0)\to (\DD', x'_0)$ such that $\nu^B\circ\rho'^B\circ\sigma\sim\mu^B\circ\rho^B$ and $\nu^C\circ\rho'^C\circ\sigma\sim\mu^C\circ\rho^C$.
\end{lem}

\begin{figure}[H]
    \centering
\begin{tikzcd}
\DD' \arrow[rd, "\cong", dashed] \arrow[rrd, "\rho'^C", bend left] \arrow[rdd, "\rho'^B"', bend right] &[-15pt] &[10pt] \\[-15pt]
&[-15pt] \DD \arrow[r, "\rho^C"] \arrow[d, "\rho^B"']  &[10 pt] \CC \arrow[d, "\mu^C"', bend right = 20] \arrow[d, "\nu^C", bend left = 20] 
\arrow[d, "\approx" description, no head, phantom]\\[10pt]
&[-15pt] \BB \arrow[r, "\mu^B", bend left=15] \arrow[r, "\nu^B"', bend right = 15] \arrow[r, "\approx" description, no head, phantom]&[10pt] \AA        
\end{tikzcd}
    \caption{$\AA$-product and morphism equivalence.}
    \label{fig: AA-product for L36}
\end{figure}
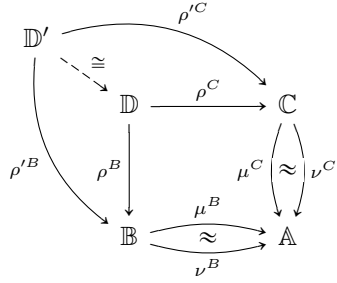

\begin{proof}
    Let $\bfa \in P(\mu^B, \nu^B)$ and $\bfa' \in P(\mu^C, \nu^C)$. For each edge $f\in E(\gr B)$ (resp. $g\in E(\gr C)$), let $f_\alpha$ and $f'_\alpha$ (resp. $g_\alpha$ and $g'_\alpha$) be its $\alpha$-twisting elements relative to $\mu^B$ and $\nu^B$ (resp. $\mu^C$ and $\nu^C$).
    
    In particular, if $f\in E(\gr B)$, $g \in E(\gr C)$ satisfy $\mu^B(f) = \mu^C(g) = e$, and if $o(f) = v$ and $o(g) = w$, then
    \begin{align*}
        a_v &= f_{\alpha}\,\alpha_e(a_f)\,f_{\alpha}'^{-1} \\
        a'_w &= g_{\alpha}\,\alpha_e(a'_g)\,g_{\alpha}'^{-1}.
    \end{align*}

   We define a morphism $\sigma\colon \DD\to \DD'$ as follows. If $x\in V_{v,w}(\gr D)$ and $h \in E_{f,g}(\gr D)$, we let
    \begin{align*}
        \sigma(x) &= \nu_v^B(B_{v})\,\left(a_v^{-1}\,d_x\,a'_w\right)\,\nu_w^C(C_{w}) \\
        \sigma(h) &= \nu_f^B(B_{f})\,\left(a_f^{-1}\,d_h\,a'_g\right)\,\nu_g^C(C_{g}).
    \end{align*}

    Note that $\sigma(x)$ is also a $(v,w)$-vertex, and $\sigma(h)$ is also an $(f,g)$-edge. By Equation~\eqref{eq: old x tilde}, if $o(h) = x$, then
    \begin{align*}
        d_x &= \mu^B_v(b_h)\, f_\alpha\, \alpha_e(d_h)\, g_\alpha^{-1}\, \mu_w^C(c_h)\inv\textrm{, and hence} \\
        a_v\inv\,d_x\, a'_w &= a_v^{-1}\,\mu^B_v(b_h)\,f_\alpha\, \alpha_e(a_f)\,\alpha_e(a_f\inv\, d_h\, a'_g)\,\alpha_e(a_g'^{-1})\, g_\alpha^{-1}\, \mu_w^C(c_h)\inv\, a'_w \\
        &= a_v^{-1}\,\mu^B_v(b_h)\,a_v\,f'_\alpha\, \alpha_e(a_f\inv\, d_h\, a'_g)\,g_\alpha'^{-1}\,a_w'^{-1}\, \mu_w^C(c_h)\inv\, a'_w \\
        &= \nu_v^B(b_h)\,f'_\alpha\, \alpha_e(a_f\inv\, d_h\, a'_g)\,g_\alpha'^{-1}\, \nu_v^B(c_h)\inv.
    \end{align*}
    In view of this equation and using Lemma \ref{lem: uniqueness up to isomorphism}, we may assume that $\DD'$ was defined using representatives 
    \begin{align*}
        d_{\sigma(x)} = a_v^{-1}\,d_x\,a'_w\enspace&\textrm{and}\enspace d_{\sigma(h)} = a_f^{-1}\,d_h\,a'_g\\
        b_{\sigma(h)} = b_h\enspace&\textrm{and}\enspace c_{\sigma(h)} = c_h.
    \end{align*}
Let $x\in V_{v,w}(\gr D)$. We note that the subgroups $D_x$ and $D'_{\sigma(x)}$ of $B_v \times C_w$ are equal. Indeed, $D_x = \twpb{B_v}{d_x}{A_u}{C_w}$ is the pullback of $\gamma_{d_x}\circ \mu^B$ and $\mu^C$, while $D'_{\sigma(x)} = \twpb{B_v}{d_{\sigma(x)}}{A_u}{C_w}$ is the pullback of $\gamma_{d_{\sigma(x)}}\circ \nu^B$ and $\nu^C$. Thus, if $(b,c) \in D_x$, we have $d_{x}\inv\, \mu^B_v(b)\, d_x = \mu^C_w(c)$, and hence
\begin{align*}
    d_{\sigma(x)}\inv\, \nu_v^B(b)\, d_{\sigma(x)} &= a_w'^{-1}\,d_x\inv\, a_v\, \nu_v^B(b)\, a_v\inv\, d_x\, a'_w \\
    &= a_w'^{-1}\,d_x\inv\, \mu_v^B(b)\, d_x\, a'_w \\
    &= a_w'^{-1}\,\mu_w^C(c)\, a'_w = \nu_w^C(c),
\end{align*}
that is, $(b,c) \in D'_{\sigma(x)}$. We then let $\sigma_x\colon D_x \to D'_{\sigma(x)}$ be the identity. Similarly, if $h\in E_{f,g}(\gr D)$, then $D_h = D'_{\sigma(h)} \le B_f \times C_g$, and we let $\sigma_h\colon D_h \to D'_{\sigma(h)}$ be the identity. We also take all the twisting elements for $\sigma$ equal to the identity. It is immediate that $\sigma$ is an isomorphism.  

Finally, we see that, if $x\in V_{v,w}(\gr D)$ and $(b,c)\in D_x$, then $\mu^B\circ\rho^B(b,c) = \mu^B_v(b)$, while $\nu^B\circ{\rho'}^B\circ \sigma(b,c) = \nu^B\circ{\rho'}^B(b,c) = \nu_v^B(b) = a_v\inv\,\mu^B_v(b)\, a_v$. Similarly, if $h\in E_{f,g}(\gr D)$ and $(b,c)\in D_h$, then $\nu^B\circ{\rho'}^B\circ \sigma(b,c) = a_f\inv\,(\mu^B\circ\rho^B(b,c))\, a_f$. Thus $\nu^B\circ{\rho'}^B\circ \sigma \approx \mu^B\circ\rho^B$. Symmetrically, $\nu^C\circ{\rho'}^C\circ \sigma \approx \mu^C\circ\rho^C$.

In particular, if $\mu^B \sim \nu^B$ (resp., $\mu^C \sim \nu^C$), then $a_{v_0}=1$ (resp., $a'_{w_0}=1$). Thus in a similar fashion as above, $\nu^B\circ{\rho'}^B\circ \sigma \sim \mu^B\circ\rho^B$ and $\nu^C\circ{\rho'}^C\circ \sigma \sim \mu^C\circ\rho^C$, which completes the proof.
\end{proof}

We now record some important properties of the diagram formed by $\mu^B$, $\mu^C$, $\rho^B$ and $\rho^C$.
\begin{figure}[H]
    \centering
    \begin{tikzcd}
        \DD \arrow[d, "\rho^B"'] \arrow[r, "\rho^C"] & \CC \arrow[d, "\mu^C"]\\
        \BB \arrow[r, "\mu^B"'] & \AA
    \end{tikzcd}
    \caption{The $\AA$-product $\DD = \BB\wtimes_\AA\CC$.}
    \label{fig: AA-product}
\end{figure}

\begin{lem}
\label{lem: composition}
The following properties hold, with reference to Figure~\ref{fig: AA-product}.
\begin{enumerate}[(1)]
\item\label{item: path level equivalence} Let $\mu^B\colon \BB \to \AA$ and $\mu^C\colon \CC \to \AA$ be morphisms of graphs of groups, and let $\DD = \BB \wtimes_\AA \CC$. If $d$ is a $\DD$-path from $x\in V(\gr{D})$ to $y\in V(\gr{D})$, then
    \[
    \mu^B\circ\rho^B(d) \sim_{\AA} d_x\ (\mu^C\circ\rho^C(d))\ d_y^{-1}.
    \]
    In particular, $\mu^B\circ\rho^B\approx \mu^C\circ\rho^C$ and $\bfd = ((d_x)_x, (d_h)_h) \in P(\mu^B\circ\rho^B, \mu^C\circ\rho^C)$.
\item\label{item: simple equivalence} Let $\mu^B\colon (\BB,v_0) \to (\AA,u_0)$ and $\mu^C\colon (\CC,w_0) \to (\AA,u_0)$ be morphisms of pointed graphs of groups, and let $(\DD,x_0)$ be their pointed $\AA$-product. If $d$ is a $\DD$-path from $x_0\in V(\gr{D})$ to $y\in V(\gr{D})$, then
    \[
    \mu^B\circ\rho^B(d) \sim_{\AA} (\mu^C\circ\rho^C(d))\ (d_y)^{-1}.
    \]
In particular, $\mu^B\circ\rho^B\sim \mu^C\circ\rho^C$ and $\bfd\in P^*(\mu^B\circ\rho^B, \mu^C\circ\rho^C)$.
\end{enumerate}
\end{lem}

\begin{proof}
Statement~\eqref{item: path level equivalence} needs only to be established for $\DD$-paths of the form $d = (d_0)$ where $d_0$ lies in some $D_x$, $x\in V(\gr D)$, or $d = (1,h,1)$, $h\in E(\gr D)$.

Let $x$ be a vertex of $\gr D$, let $d_0 = (b_0, c_0)\in D_x$ and let $d = (d_0)$. Let $v$ and $w$ be the images of $x$ in $\BB$ and $\CC$, respectively, and let $u = [v] = [w] \in V(\gr A)$. Then $\mu_v^B(b_0)^{d_x} = \mu_w^C(c_0)$ and we have $\rho^B(d) = (b_0)$ and $\rho^C(d) = (c_0)$. It follows that $\mu^C(\rho^C(d)) = (\mu_w^C(c_0)) = (\mu_v^B(b_0)^{d_x}) = \mu^B(\rho^B(d))^{d_x}$, as expected.

Let now $h$ be an edge of $\gr D$ from vertex $x$ to vertex $y$, and let $d = (1, h, 1)$. Let $v, v'\in V(\gr B)$, $w, w'\in V(\gr C)$ and $u, u'\in V(\gr A)$ such that $x$ is a $(v,w)$-vertex, $y$ is a $(v',w')$-vertex, $[v] = [w] = u$ and $[v'] = [w'] = u'$. Let also $e, f, g$ be in $E(\gr A)$, $E(\gr B)$ and $E(\gr C)$, respectively, such that $h$ is an $(f,g)$-edge and $e = [f] = [g]$. Then, since $b_h$ and $c_h$ are the twisting elements for $h$ relative to $\rho^B$ and $\rho^C$ (see Definition~\ref{defn: projection morphisms}), we have
\begin{align*}
    \rho^B(d) &= \left(b_h, \ f,\ b_{h^{-1}}\inv)\right),\\
    \rho^C(d) &= \left(c_h,\ g,\ c_{h^{-1}}\inv\right),
\end{align*}
and hence,
\begin{align*}
    \mu^B(\rho^B(d)) &= \left(\mu_v^B(b_h)\,f_\alpha,\ e,\ f_\omega\inv\,\mu_{v'}^B(b_{h^{-1}})\inv \right),\\
    \mu^C(\rho^C(d)) &= \left(\mu_w^C(c_h)\,g_\alpha,\ e,\ g_\omega\inv\,\mu_{w'}^C(c_{h^{-1}})\inv \right).
\end{align*}

Recall Equation~\eqref{eq: old x tilde}, applied to vertices $x$ and $y$:
\begin{align*}
d_x &= \mu_v^B(b_h)\,f_{\alpha}\,\alpha_e(d_h)\,g_{\alpha}^{-1}\,\mu_w^C(c_h)^{-1} \\
d_y &= \mu_{v'}^B(b_{h\inv})\,(f\inv)_{\alpha}\,\alpha_{e\inv}(d_{h\inv})\,(g\inv)_{\alpha}^{-1}\,\mu_{w'}^C(c_{h\inv})^{-1} \\
&= \mu_{v'}^B(b_{h\inv})\,f_\omega\,\omega_e(d_h)\,g_\omega^{-1}\,\mu_{w'}^C(c_{h\inv})^{-1}.
\end{align*}
Then we have
$$\mu^B(\rho^B(d)) = \left(d_x\ \mu_w^C(c_h)\,g_\alpha\,\alpha_e(d_h)\inv,\ e,\ \omega_e(d_h)\,g_\omega\inv\,\mu_w^C(c_{h\inv})\inv\ d_y\inv\right) \sim_\AA d_x\ \mu^C(\rho^C(d))\ d_y\inv,$$
as expected. The fact that $\mu^B\circ\rho^B\approx \mu^C\circ\rho^C$ with the claimed parameters follows directly from Proposition~\ref{prop: change of data}.

Statement~\eqref{item: simple equivalence} now follows from Statement~\eqref{item: path level equivalence} and the fact that $d_{x_0} = 1$.
\end{proof}

\begin{lem}\label{lem: rhos are folded 2}
If $\mu^C$ is an immersion, then $\rho^B$ is an immersion as well. Likewise, if $\mu^B$ is an immersion, then $\rho^C$ is an immersion as well.
\end{lem}

\begin{proof}
Assuming that $\mu^C$ is an immersion, we show that $\rho^B$ is an immersion, starting with the first condition in Definition~\ref{def: folded}. Let $h,h' \in E(\gr{D})$ be edges with the same $\rho^B$-image $f\in E(\gr{B})$ and the same initial vertex $x \in V(\gr{D})$. Let $g = \rho^C(h)$, $g' = \rho^C(h')$, $v = \rho^B(x)$, $w = \rho^C(x)$ and suppose that $(h')^B_\alpha = \rho^B_x(d)\,h^B_\alpha\,\alpha_f(b_1) = b \,h^B_\alpha\,\alpha_f(b_1)$ for some $d = (b, c)\in D_x$ and $b_1\in B_f$. We want to show that $h = h'$.

By definition of $\gr{D}$, $\mu^B(f) = \mu^C(g) = \mu^C(g') = e$ for some $e\in E(\gr{A})$, and by definition of $D_x$, we have $\mu_v^B(b)^{d_x} = \mu^C_w(c)$.

In view of Equation~\eqref{eq: old x tilde}, we have
\[
d_x = \mu_v^B(h^B_\alpha)\,f_{\alpha}\,\alpha_e(d_h)\,g_{\alpha}^{-1}\,\mu_w^C(h^C_\alpha)^{-1} = \mu_v^B((h')^B_\alpha)\,f_{\alpha}\,\alpha_e(d_{h'})\,(g')_{\alpha}^{-1}\,\mu_w^C((h')^C_\alpha)^{-1},
\]
and hence
\begin{align*}
    \mu^C_w(c) &= (\mu_v^B((h')^B_\alpha)\,f_{\alpha}\,\alpha_e(d_{h'})\,(g')_{\alpha}^{-1}\,\mu_w^C((h')^C_\alpha)^{-1})^{-1} \, \mu^B_v(b) \,\mu_v^B(h^B_\alpha)\,f_{\alpha}\,\alpha_e(d_{h})\,g_{\alpha}^{-1}\,\mu_w^C(h^C_\alpha)^{-1}\\
                &= \mu^C_w((h')^C_{\alpha})\, g'_{\alpha}\, \alpha_e(d_{h}')^{-1}\, f_{\alpha}^{-1}\, \mu^B_v(\alpha_f^B(b_1))^{-1}\, f_{\alpha}\, \alpha_e(d_{h})\, g_{\alpha}^{-1}\, \mu^C_w(h_{\alpha}^C)^{-1}\\
                &= \mu^C_w((h')^C_{\alpha})\, g'_{\alpha}\, \alpha_e(d_{h}'^{-1}\, \mu^B_{f}(b_1)^{-1}\, d_{h})\, g_{\alpha}^{-1}\, \mu^C_w(h_{\alpha}^C)^{-1}.
\end{align*}
Rearranging, we get that
\begin{align*}
    \mu^C_w(((h')_{\alpha}^C)^{-1}\, c\, h_{\alpha}^C)\, g_{\alpha}\, \alpha_e(d_{h}^{-1}\, \mu^B_f(b_1)\, d_{h}') = g_{\alpha}'.
\end{align*}
Since $\mu^C$ is an immersion, this implies that $g = g'$.

Now
$$\alpha_e(d_{h'}) = \mu^B_v\left(((h')^B_\alpha)\inv h^B_\alpha\right)^{f_\alpha}\,\alpha_e(d_h)\, \mu^C_w\left(((h')^C_\alpha)\inv h^C_\alpha\right)^{g_\alpha}.$$
The first factor is in $\alpha_e(\mu^B_f(B_f))$ and the last factor is in $\alpha_e(\mu^C_g(C_g))$. Using the injectivity of $\alpha_e$, it follows that $d_{h'} \in \mu^B_f(B_f)\,d_h\,\mu^C_g(C_g) = h$, and so $h = h'$. That is, the first condition in the definition of an immersion (Definition~\ref{def: folded}) is satisfied.

We now turn to the second condition in that definition. Let us assume that $d = (b, c)\in D_x$ satisfies $\rho^B_x(d)^{h_\alpha} \in \alpha_f(B_f)$, and let us show that $d \in \alpha_h(D_h)$. By definition, we have $\mu^B_v(b)^{d_x}  = \mu^C_w(c)$. Since $b^{h_\alpha} \in \alpha_f(B_f)$, there exists $y\in B_f$ such that $b^{b_h} = \alpha_f(y)$. It follows that
\begin{align*}
\mu^B_v(b)^{d_x} &= \mu^B(b^{b_h})^{f_\alpha\alpha_e(d_h)g_\alpha\inv\mu^C_w(c_h)\inv} = \mu^B_v(\alpha_f(y))^{f_\alpha\alpha_e(d_h)g_\alpha\inv\mu^C_w(c_h)\inv} \\
&=\alpha_e(\mu^B_f(y))^{\alpha_e(d_h)g_\alpha\inv\mu^C_w(c_h)\inv} =\enspace \alpha_e(\mu^B_f(y)^{d_h})^{g_\alpha\inv\mu^C_w(c_h)\inv}.
\end{align*}
Since $\mu^B(b)^{d_{x}} = \mu^C(c)$, we have that $\mu^C_w(c^{c_h})^{g_{\alpha}}= \alpha_e(\mu^B_f(y)^{d_h}) \in \alpha_e(A_e)$. Since $\mu^C$ is an immersion, it follows that $c^{c_h} = \alpha_g(z)$ for some $z\in C_g$ and
$$\alpha_e(\mu^B_f(y)^{d_h}) = \mu^C_w(c^{c_h})^{g_\alpha} = \mu^C_w(\alpha_g(z))^{g_\alpha} = \alpha_e(\mu^C_g(z)).$$
Therefore, $\mu^B_f(y)^{d_h} = \mu^C_g(z)$ and so $(y, z)\in \twpb{B_f}{d_h}{A_e}{C_g} = D_h$ and %
$$\rho^B_x(\alpha_h(y, z)) = \alpha_f(\rho^B_h(y, z))^{b_h\inv} = \alpha_f(y)^{b_h\inv} = b = \rho^B_x(d).$$
It follows, finally, that $d = \alpha_h(y, z) \in \alpha_h(D_h)$, which concludes the proof that $\rho^B$ is an immersion.

The proof that $\rho^C$ is an immersion is identical, swapping the roles of $B$ and $C$.
\end{proof}

\paragraph{Notational convention}
    The full notation of morphisms of graphs of groups used so far is rather heavy, and is not easy to parse when multiple morphisms interact. For the sake of legibility, we will often skip the vertex and edge subscripts when the context makes them unambiguous.

\subsection{Lifts to the $\AA$-product}
\label{sec: lifts}

In this section, we work towards the discussion of pullbacks in the categories $\GrGp^*$ and $\GrGp$, by considering the following situation. We fix morphisms $\mu^B$, $\mu^C$,  $\sigma^B$ and $\sigma^C$ as follows, see also Figure~\ref{fig: for the lifts section}.

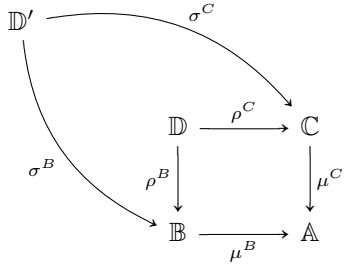
\begin{figure}[htbp]
    \centering
    \begin{tikzcd}
        \DD'  \arrow[rdd, "\sigma^B"', bend right] \arrow[rrd, "\sigma^C", bend left] \\
        & \DD \arrow[d, "\rho^B"'] \arrow[r, "\rho^C"] & \CC \arrow[d, "\mu^C"]\\
        & \BB \arrow[r, "\mu^B"'] & \AA
    \end{tikzcd}
    \caption{The morphisms $\mu^B$ and $\mu^C$, $\rho^B$ and $\rho^C$, and $\sigma^B$ and $\sigma^C$}
\label{fig: for the lifts section}
\end{figure}

\begin{itemize}
\item $\mu^B\colon \BB\to\AA$ and $\mu^C\colon\CC\to\AA$ are morphisms of graphs of groups, such that $\BB$ and $\CC$ are connected. We let $\DD = \BB\wtimes_\AA\CC$ be their $\AA$-product, with projection morphisms $\rho^B\colon \DD\to\BB$ and $\rho^C\colon \DD\to \CC$. 
\item $\sigma^B\colon \DD'\to\BB$ and $\sigma^C\colon \DD'\to \CC$ are morphisms of graphs of groups, such that $\DD'$ is connected and $\mu^B\circ\sigma^B \approx \mu^C\circ\sigma^C$. Let $\bfa \in P(\mu^B\circ\sigma^B, \mu^C\circ\sigma^C)$.
\end{itemize}

We shall completely characterise all morphisms of graphs of groups $\sigma\colon \DD'\to \DD$ such that $\rho^B\circ\sigma \approx \sigma^B$ and $\rho^C\circ\sigma \approx \sigma^C$. We call such morphisms \emph{lifts (of $\sigma^B, \sigma^C$)}. It will turn out that all lifts are of the form described in the following proposition.

\begin{prop}
\label{prop: existence 2}
For each $\bfz \in C(\mu^C\circ\sigma^C)$, there exists
 a lift $\sigma^{\bfz}\colon \DD' \to \DD$ such that
 \begin{itemize}
 \item for each $y\in V(\gr D')$, we have
  \begin{align*}
    \sigma^{\bfz}(y) &= \mu^B(B_{\sigma^B(y)})\ (a_y\,z_y)\ \mu^C(C_{\sigma^C(y)}),\\
    \sigma^{\bfz}(h) &= \mu^B(B_{\sigma^B(h)})\ (a_h\,z_h)\ \mu^C(C_{\sigma^C(h)});
  \end{align*}
  \item if $\bfb = ((b_y)_y, (b_h)_h)$ and $\bfc = ((c_y)_y, (c_h)_h)$ are tuples so that
  \begin{align*}
  d_{\sigma^{\bfz}(y)} &= \mu^B(b_y)\, (a_y\, z_y)\, \mu^C(c_y)^{-1},\\
  d_{\sigma^{\bfz}(h)} &= \mu^B(b_h)\, (a_h\, z_h)\, \mu^C(c_h)^{-1},
  \end{align*}
  then $\bfb \in P(\rho^B\circ\sigma^{\bfz}, \sigma^B)$ and $\bfc \in P(\rho^C\circ\sigma^{\bfz}, \sigma^C)$.
\end{itemize}
\end{prop}

\begin{proof}
For each edge $h\in E(\gr D')$, let $h^{\sigma^B}_\alpha$ and $h^{\sigma^C}_\alpha$ be its $\alpha$-twisting elements relative to $\sigma^B$ and $\sigma^C$. In particular, if the images of $h$ in $\BB$ and $\CC$ are $f$ and $g$, respectively, then the $\alpha$-twisting elements of $h$ relative to $\mu^B\circ\sigma^B$ and $\mu^C\circ\sigma^C$ are $\mu^B(h^{\sigma^B}_{\alpha})\,f_{\alpha}$ and $\mu^C(h^{\sigma^C}_{\alpha})\,g_{\alpha}$, respectively.

Let $p$ be a $\DD'$-path from $y$ to $z$. In view of Definition~\ref{def: mor_sim} and Proposition~\ref{prop: change of data}, we have 
\begin{equation}\label{eq: new ay vs ah}
\left.
\begin{aligned}
\mu^C\circ\sigma^C(p) &\sim_{\AA} a_y^{-1}\ \left(\mu^B\circ\sigma^B(p)\right)\ a_z \\
a_y &= \mu^B(h^{\sigma^B}_{\alpha})\,f_{\alpha}\, \alpha_e(a_h)\, g_\alpha\inv\, \mu^C(h^{\sigma^C}_{\alpha})\inv
\end{aligned}
\right\rbrace.
\end{equation}
Let $\bfz \in C(\mu^C\circ\sigma^C)$. We construct a morphism $\sigma^{\bfz}\colon \DD' \to \DD$ satisfying $\rho^B \circ \sigma^{\bfz} \approx \sigma^B$ and $\rho^C \circ \sigma^{\bfz} \approx \sigma^C$ and such that, if $y\in V(\gr D')$, $v = \sigma^B(y)$ and $w = \sigma^C(y)$, $h\in E(\gr{D}')$, $f = \sigma^B(h)$ and $g = \sigma^C(h)$, then
\begin{align*}
    \sigma^{\bfz}(y) &= \mu^B(B_{v})\ (a_y\,z_y)\ \mu^C(C_{w}),\\
    \sigma^{\bfz}(h) &= \mu^B(B_{f})\ (a_h\,z_h)\ \mu^C(C_{g}).
\end{align*}
First, we verify that this map is a graph morphism.

Since $\bfz\in C(\mu^C\circ\sigma^C)$, Definition~\ref{def: mor_sim} shows that
\begin{equation}
\label{eq: new k_h}
z_y = \mu^C(h_{\alpha}^{\sigma^C})\, g_{\alpha}\, \alpha_{e}(z_h)\,(\mu^C(h_{\alpha}^{\sigma^C})\, g_{\alpha})\inv.
\end{equation}
The origin of $\sigma^{\bfz}(h)$ is then
\begin{align*}
 o(\sigma^{\bfz}(h)) &= \mu^B(B_v)\, (\mu^B(h_{\alpha}^{\sigma^B})\, f_{\alpha})\, \alpha_e(a_h\,z_h)\, (\, \mu^C(h_{\alpha}^{\sigma^C})\,g_{\alpha})\inv\, \mu^C(C_w) \\
 &= \mu^B(B_v)\, a_y\, \mu^C(h_{\alpha}^{\sigma^C})\,g_{\alpha}\, \alpha_e(z_h)\, (\mu^C(h_{\alpha}^{\sigma^C})\,g_{\alpha})\inv\, \mu^C(C_w)\textrm{ by Equation~\eqref{eq: new ay vs ah}}\\
 &= \mu^B(B_v)\ (a_y\,z_y)\ \mu^C(C_w)\textrm{ by Equation~\eqref{eq: new k_h},}
 \end{align*}
that is, $o(\sigma^{\bfz}(h)) = \sigma^{\bfz}(y)$. The proof that $t(\sigma^{\bfz}(h)) = \sigma^{\bfz}(z)$ is similar. Thus $\sigma^{\bfz}$ is a graph morphism.

We now turn to the vertex group and edge group components of the morphism of graphs of groups $\sigma^{\bfz}$ whose existence we seek to establish.
For each vertex $y\in V(\gr{D}')$, with images $v$ and $w$ in $\BB$ and $\CC$, there exist elements $b_y\in B_v$ and $c_y\in C_w$ such that
\begin{equation}\label{eq: new ay by cy sigma tilde}
d_{\sigma^{\bfz}(y)} = \mu^B(b_y)\,(a_y\,z_y)\,\mu^C(c_y)^{-1}
\end{equation}
where $d_{\sigma^{\bfz}(y)}$ is the double coset representative associated with the vertex $\sigma^{\bfz}(y)$.

We let $\sigma^{\bfz}_y = (\gamma_{b_y}\inv\circ\sigma^B, \gamma_{c_y}\inv\circ\sigma^C)$. We need to verify that the range of this map is in $D_{\sigma(y)} = \twpb{B_v}{d_{\sigma(y)}}{A_u}{C_w}$ (it is clearly a morphism). Indeed, if $x\in D'_y$, then
 \begin{align*}
     d_{\sigma(y)}&\inv\, \mu^B\left(\sigma^B(x)^{b_y\inv}\right)\, d_{\sigma(y)} =\\
     &= \mu^C(c_y)\, z_y\inv\, a_y\inv\, \mu^B(b_y)\inv\, \mu^B(\sigma^B(x))^{\mu^B(b_y\inv)} \mu^B(b_y)\, a_y\, z_y\, \mu^C(c_y)\inv \\
     &= \mu^C(c_y)\, z_y\inv\, a_y\inv\, (\mu^B\circ\sigma^B)(x)\, a_y\, z_y\, \mu^C(c_y)\inv \\
     &= \mu^C(c_y)\, z_y\inv\, (\mu^C\circ\sigma^C)(x)\, z_y\, \mu^C(c_y)\inv  \textrm{, since $\bfa \in P(\mu^B\circ\sigma^B, \mu^C\circ\sigma^C)$}\\
     &= \mu^C(c_y)\, (\mu^C\circ\sigma^C)(x)\, \mu^C(c_y)\inv \textrm{, since $\bfz \in C(\mu^C\circ\sigma^C)$}\\
     &= \mu^C\left(\sigma^C(y)^{c_y\inv}\right).
 \end{align*}
 Similarly, for each edge $h\in E(\gr D')$ with images $e$, $f$ and $g$ in $\AA$, $\BB$ and $\CC$, there exist $b_h \in B_f$ and $c_h\in C_g$ such that 
$$d_{\sigma^{\bfz}(h)} = \mu^B(b_h)\, (a_h\,z_h)\,\mu^C(c_h)\inv,$$
where $d_{\sigma^{\bfz}(h)}$ is the double coset representative associated with the edge $\sigma^{\bfz}(h)$.
By the same reasoning as for vertices, now using Lemma~\ref{rem:edge_equivalence}, we find that the map $\sigma^{\bfz}_h = (\gamma_{b_h\inv}\circ\sigma^B_h, \gamma_{c_h\inv}\circ\sigma^C_h)$ defines a morphism from $D'_h$ to $D_{\sigma^{\bfz}(h)}$.

Now consider the following elements of $B_v$ and $C_w$:
$$b = b_{y}\, h_{\alpha}^{\sigma^B}\, \alpha_{f}(b_h)\inv\, (\sigma^{\bfz}(h)_{\alpha}^{B})^{-1} \quad\textrm{and}\quad c = c_{y}\, h_{\alpha}^{\sigma^C}\, \alpha_{g}(c_h)\inv\, (\sigma^{\bfz}(h)_{\alpha}^{C})^{-1}.$$
We let $h_\alpha = (b,c)$ be the twisting element associated to $h$ by $\sigma^{\bfz}$ (and $h_\omega = (h\inv)_\alpha$). In order to justify this choice, we need to verify that $(b,c) \in D_{\sigma^{\bfz}(y)}$. We first note that
Equations~\eqref{eq: new ay vs ah} and~\eqref{eq: new k_h} show that
\begin{equation}\label{eq: new ayky vs ahkh}
a_y\,z_y = \mu^B(h_\alpha^{\sigma^B})\, f_\alpha\, \alpha_e(a_h\,z_h)\, g_\alpha\inv\, \mu^C(h_\alpha^{\sigma^C})\inv.
\end{equation}
We now verify that $\mu^C(c) = \mu^B(b)^{d_{\sigma^{\bfz}(y)}}$:
\begin{align*}
\mu^C(c) &= \mu^C(c_{y})(a_{y}\,z_y)^{-1}\,\mu^B(h_{\alpha}^{\sigma^B})\,f_{\alpha}\,\alpha_e(a_h\,z_h)\,g_{\alpha}^{-1}\ \mu^C\left(\alpha_{g}(c_h)\inv\,(\sigma^{\bfz}(h)_{\alpha}^{C})^{-1}\right)\textrm{ by Eq.~\eqref{eq: new ayky vs ahkh}}\\
&=d_{\sigma^{\bfz}(y)}\inv\,\mu^B(b_y)\, \mu^B(h_{\alpha}^{\sigma^B})\,f_{\alpha}\,\alpha_e(a_h\,z_h)\,g_{\alpha}^{-1}\ \mu^C\left(\alpha_{g}(c_h)\inv\,(\sigma^{\bfz}(h)_{\alpha}^{C})^{-1}\right) \\
&\hskip 10cm\textrm{ by definition of $c_y$}\\
&= d_{\sigma^{\bfz}(y)}^{-1}\,\mu^B(b_y\, h_{\alpha}^{\sigma^B})\,f_{\alpha}\,\alpha_e\left(\mu^B(b_h)\inv\,d_{\sigma^{\bfz}(h)}\,\mu^C(c_h)\right)\,g_{\alpha}^{-1} \\
				&\hskip 4.5cm \mu^C\left(\alpha_{g}(c_h)\inv\,(\sigma^{\bfz}(h)_{\alpha}^{C})^{-1}\right)\quad \textrm{by definition of $b_h$ and $c_h$}\\
				&= d_{\sigma^{\bfz}(y)}^{-1}\,\mu^B(b_y\inv\, h_{\alpha}^{\sigma^B})\,\mu^B(\alpha_f(b_h)\inv)\,f_{\alpha}\,\alpha_e(d_{\sigma^{\bfz}(h)})\,g_{\alpha}^{-1}\, \mu^C(\sigma^{\bfz}(h)_{\alpha}^{C})^{-1} \\
				& \hskip 8cm \textrm{ since $\mu^C$ and $\mu^B$ are morphisms}\\
                &= d_{\sigma^{\bfz}(y)}\inv\, \mu^B(b_y\inv\, h_{\alpha}^{\sigma^B}\,\alpha_f(b_h)\inv\, (\sigma^{\bfz}(h)^B_\alpha)\inv)  \\
                &\hskip 5cm\mu^B(\sigma^{\bfz}(h)^B_\alpha)\,f_{\alpha}\,\alpha_e(d_{\sigma^{\bfz}(h)})\,g_{\alpha}^{-1}\, \mu^C(\sigma^{\bfz}(h)_{\alpha}^{C})^{-1}\\
                &= d_{\sigma^{\bfz}(y)}\inv\, \mu^B(b) \,\left(\mu^B(\sigma^{\bfz}(h)^B_\alpha)\,f_{\alpha}\,\alpha_e(d_{\sigma^{\bfz}(h)})\,g_{\alpha}^{-1}\, \mu^C(\sigma^{\bfz}(h)_{\alpha}^{C})^{-1}\right).
\end{align*}
Now Equation~\eqref{eq: old x tilde}, together with the choice of twisting elements in Definition~\ref{defn: projection morphisms} shows that
$$d_{\sigma^{\bfz}(y)} = \mu^B(\sigma^{\bfz}(h)^B_\alpha)\,f_{\alpha}\,\alpha_e(d_{\sigma^{\bfz}(h)})\,g_{\alpha}^{-1}\, \mu^C(\sigma^{\bfz}(h)_{\alpha}^{C})^{-1},$$
and hence $\mu^C(c) = \mu^B(b)^{d_{\sigma^{\bfz}(y)}}$, that is, $(b,c) \in D_{\sigma^{\bfz}(y)}$. This completes the definition of $\sigma^{\bfz}$.

There remains to verify that $\sigma^B \approx \rho^B\circ\sigma^{\bfz}$ and $\sigma^C \approx \rho^C\circ\sigma^{\bfz}$.
In view of Definition~\ref{defn: projection morphisms}, for each $x \in D'_y$ and $z\in D'_h$, we have
\begin{align*}
    \rho^B(\sigma^{\bfz}(x)) &= \rho^B(\gamma_{b_y}\inv\circ\sigma^B(x),\gamma_{c_y}\inv\circ\sigma^C(x)) = \gamma_{b_y}\inv\circ\sigma^B(x) \\
    \rho^B(\sigma^{\bfz}(z)) &= \rho^B(\gamma_{b_h}\inv\circ\sigma^B(z),\gamma_{c_h}\inv\circ\sigma^C(z)) = \gamma_{b_h}\inv\circ\sigma^B(z). 
\end{align*}
so $\sigma^B = \gamma_{b_y} \circ \rho^B\circ \sigma^{\bfz}$ on $D'_y$, and $\sigma^B = \gamma_{b_h} \circ \rho^B\circ \sigma^{\bfz}$ on $D'_h$.

We now only need to verify Condition~(3) in Definition~\ref{def: mor_sim}, namely that $b_y = h_{\alpha}^{\rho^B\circ\sigma^{\bfz}} \,\alpha_f(b_h)\ (h_\alpha^{\sigma^B})\inv$, where $h_{\alpha}^{\rho^B\circ\sigma^{\bfz}}$ is the $\alpha$-twisting element for $h$ relative to $\rho^B\circ\sigma^{\bfz}$. Since $(b,c)$ is the twisting element for $h$ relative to $\sigma^{\bfz}$, we have $h_{\alpha}^{\rho^B\circ\sigma^{\bfz}} = \rho^B(b,c)\,(\sigma^{\bfz}(h))^B_\alpha = b\, (\sigma^{\bfz}(h))^B_\alpha$. The expected equality then holds by definition of $b$.

Thus $\bfb = ((b_y)_y, (b_h)_h) \in P(\rho^B \circ \sigma^{\bfz}, \sigma^B)$ and $\bfc = ((c_y)_y, (c_h)_h) \in P(\rho^C \circ \sigma^{\bfz}, \sigma^C)$, which concludes the proof.
\end{proof}

\begin{prop}\label{prop: agree on graph}
Let $\sigma, \sigma'\colon \DD'\to \DD$ be two lifts of $(\sigma^B, \sigma^C)$, and suppose there exists a vertex $y_0\in V(\gr D')$ such that $\sigma(y_0) = \sigma'(y_0)$. Then $\sigma$ and $\sigma'$ agree on the underlying graph $\gr D'$.
\end{prop}

\begin{proof}
Let $\bfb \in P(\rho^B\circ \sigma, \sigma^B)$, $\bfb' \in P(\rho^B\circ \sigma', \sigma^B)$, $\bfc \in P(\rho^C\circ \sigma, \sigma^C)$ and $\bfc' \in P(\rho^C\circ \sigma', \sigma^C)$. We show that, if $h\in E(\gr{D}')$ is an edge with origin $y$ and if $\sigma(y) = \sigma'(y)$, then $\sigma(h) = \sigma'(h)$. This will imply that $\sigma$ and $\sigma'$ also agree on the end vertex of $h$. Since $\sigma(y_0) = \sigma'(y_0)$ and $\gr D'$ is connected, it will follow that $\sigma$ and $\sigma'$ agree on $\gr D'$.
 
Let $e$, $f$ and $g$ (resp. $u$, $v$ and $w$) be the images of $h$ (resp. $y$) in $\AA$, $\BB$ and $\CC$ under $\mu^C\circ\sigma^C$, $\sigma^B$ and $\sigma^C$. We have
\begin{align*}
    \rho^C\circ\sigma(1, h, 1) &= \rho^C(h_{\alpha}^{\sigma},\ \sigma(h),\  (h_{\omega}^{\sigma})^{-1})\\
                                &= (\rho^C(h_{\alpha}^{\sigma})\,\sigma(h)_{\alpha}^C,\ g,\ (\sigma(h)_{\omega}^C)^{-1}\, \rho^C(h_{\omega}^{\sigma})^{-1})\\
    \rho^C\circ\sigma'(1, h, 1) &= \rho^C(h_{\alpha}^{\sigma'},\  \sigma'(h),\  (h_{\omega}^{\sigma'})^{-1})\\
                                &= (\rho^C(h_{\alpha}^{\sigma'})\,\sigma'(h)_{\alpha}^C,\ g,\ (\sigma'(h)_{\omega}^C)^{-1}\, \rho^C(h_{\omega}^{\sigma'})^{-1})
\end{align*}
Since $\bfc\,(\bfc')\inv \in P(\rho^C\circ\sigma, \rho^C\circ\sigma')$ (Lemma~\ref{lem: double equivalence elements}), Proposition \ref{prop: sim vs equiv} shows that there exists $x^C_h \in C_g$ such that
\[
    \rho^C(h_{\alpha}^{\sigma})\, \sigma(h)_{\alpha}^C\, \alpha_g(x^C_h) = c_y\, (c_y')^{-1}\, \rho^C(h_{\alpha}^{\sigma'})\, \sigma'(h)_{\alpha}^C.
\]
Rearranging, we get
\begin{equation}\label{eq: new sigma(h)}
    \sigma(h)_{\alpha}^C = \rho^C(h_{\alpha}^{\sigma})^{-1}\, c_y\, (c_y')^{-1}\, \rho^C(h_{\alpha}^{\sigma'})\, \sigma'(h)_{\alpha}^C\, \alpha_g(x^C_h)\inv.
\end{equation}
Using the definition of the parameters associated with an $\approx$-equivalence, we have
\[
\alpha_g(c_h(c_h')^{-1}) = (\rho^C(h_{\alpha}^{\sigma})\, \sigma'(h)_{\alpha}^C)^{-1}\, c_y\, (c_y')^{-1}\, \rho^C(h_{\alpha}^{\sigma'})\, \sigma'(h)_{\alpha}^C.
\]
Substituting this into \eqref{eq: new sigma(h)}, and carrying out all the above for $B$ as well, we see that
\begin{align}
\label{eq: new sigma(h) 2}    \sigma(h)_{\alpha}^C &= \sigma'(h)_{\alpha}^C\, \alpha_g(c_h\, (c_h')^{-1}\, x^C_h)\\
\label{eq: new sigma(h) 3}    \sigma(h)_{\alpha}^B &= \sigma'(h)_{\alpha}^B\, \alpha_f(b_h\,(b_h')^{-1}\, x^B_h)
\end{align}
for some $x^B_h \in B_f$.
Now we have (by definition of the incidence relation in $\gr D$ and in view of Equation~\eqref{eq: old x tilde})
\begin{align*}
\mu^B(\sigma(h)_{\alpha}^B)\, f_{\alpha}\, \alpha_e(d_{\sigma(h)})\, g_{\alpha}^{-1}\, \mu^C(\sigma(h)_{\alpha}^C)^{-1} &= d_{\sigma(y)}\\
\mu^B(\sigma'(h)_{\alpha}^B)\, f_{\alpha}\, \alpha_e(d_{\sigma'(h)})\, g_{\alpha}^{-1}\, \mu^C(\sigma'(h)_{\alpha}^C)^{-1} &= d_{\sigma'(y)}.
\end{align*}
Using \eqref{eq: new sigma(h) 2} and \eqref{eq: new sigma(h) 3}, and our assumption that $\sigma(y) = \sigma'(y)$, we have
\begin{align*}
f_{\alpha}\, \alpha_e(d_{\sigma'(h)})\, g_{\alpha}^{-1} &= \mu^B(\alpha_f(b_h\, (b_h')^{-1}\, x^B_h))\, f_{\alpha}\, \alpha_e(d_{\sigma(h)})\, g_{\alpha}\inv\, \mu^C(\alpha_g(c_h\, (c_h')^{-1}\,x^C_h))^{-1}\\
        &= f_{\alpha}\,\alpha_e(\mu^B(b_h\,(b_h')^{-1}\,x^B_h)\, d_{\sigma(h)}\, \mu^C(c_h\, (c_h')^{-1}\, x^C_h)^{-1})\, g_{\alpha}^{-1}
\end{align*}
and so
\[
d_{\sigma'(h)} = \mu^B(b_h\, (b_h')^{-1}\, x^B_h)\ d_{\sigma(h)}\ \mu^C(c_h\, (c_h')^{-1}\, x^C_h)^{-1} \in \mu^B(B_f)\, d_{\sigma(h)}\, \mu^C(C_g).
\]
Thus, $\sigma$ and $\sigma'$ agree on $h$, and hence on $\gr D'$.
\end{proof}

\begin{prop}
\label{prop: equivalence characterisation}
Let $\sigma, \sigma'\colon \DD'\to \DD$ be two lifts of $(\sigma^B, \sigma^C)$, let $\bfc \in P(\rho^C\circ\sigma, \sigma^C)$ and $\bfc' \in P(\rho^C\circ\sigma', \sigma^C)$. Then $\sigma\approx\sigma'$ if and only if there exist $\bfz\in C(\mu^C\circ\sigma^C)$ and $\bfx \in \prod_yD_{\sigma(y)} \times \prod_hD_{\sigma(h)}$ such that
\begin{align*}
z_y &= \mu^C(c_y\inv\,\rho^C(x_y)\,c'_y) = \left(d_{\sigma(y)}\,\mu^C(c_y)\right)^{-1}\, \mu^B(\rho^B(x_y))\, \left(d_{\sigma'(y)}\,\mu^C(c_y')\right)\\
z_h &= \mu^C(c_h\inv\,\rho^C(x_h)\,c'_h) = \left(d_{\sigma(h)}\,\mu^C(c_h)\right)^{-1}\, \mu^B(\rho^B(x_h))\, \left(d_{\sigma'(h)}\,\mu^C(c_h')\right)
\end{align*}
In that case, $\bfx \in P(\sigma, \sigma')$.
\end{prop}

\begin{proof}
By Lemmas~\ref{lem: double equivalence elements} and~\ref{cor: composition associated elements 2}, $\bfc\,{\bfc'}\inv\in P(\rho^C\circ\sigma, \rho^C\circ\sigma')$ and $\mu^C(\bfc\,{\bfc'}\inv) \in P(\mu^C\circ\rho^C\circ\sigma, \mu^C\circ\rho^C\circ\sigma')$.

By Lemma \ref{lem: twist lemma}, $\sigma\approx\sigma'$ if and only if $\sigma$ and $\sigma'$ agree on the underlying graph $\gr D'$ and there exists $\bfz'\in C(\mu^C\circ\rho^C\circ\sigma)$ and $\bfx \in \prod_{y}D_{\sigma(y)}\times\prod_{h}D_{\sigma(h)}$ such that for each vertex $y$ and each edge $h$ of $\gr D'$,
\begin{align}
\label{eq: equivalent lifts vertices} z'_y\ \mu^C(c_y\,c_y'^{-1}) &= \mu^C\circ\rho^C(x_y) = \mu^B\circ\rho^B(x_y)^{d_{\sigma(y)}},\\
\label{eq: equivalent lifts edges} z'_h\ \mu^C(c_h\,c_h'^{-1}) &= \mu^C\circ\rho^C(x_h) = \mu^B\circ\rho^B(x_h)^{d_{\sigma(h)}},
\end{align}
in which case $\bfx \in P(\sigma, \sigma')$.

Let us first suppose that $\sigma \approx \sigma'$. Then
$\sigma$ and $\sigma'$ agree on $\gr D'$ and
there exist $\bfz' \in C(\mu^C\circ\rho^C\circ\sigma)$ and $\bfx \in \prod_yD_{\sigma(y)}\times\prod_hD_{\sigma(h)}$ such that Equations~\eqref{eq: equivalent lifts vertices} and \eqref{eq: equivalent lifts edges} hold.

Let $\bfz = ((z_y)_y, (z_h)_h)$ be given by $z_y = (z'_y)^{\mu^C(c_y)}$ and $z_h = (z_h')^{\mu^C(c_h)}$ for each vertex $y$ and edge $h$ of $\gr D'$. Then $\bfz \in C(\mu^C\circ\sigma^C)$
by Lemma~\ref{lem: almost uniqueness of associated elements}. Moreover, we have
\begin{align*}
    z_y &= \mu^C(c_y\inv\,\rho^C(x_y)\,c'_y)\textrm{ and} \\
    z_y &= \mu^C(c_y)\inv\, d_{\sigma(y)}\inv\, \mu^B(\rho^B(x_y))\, d_{\sigma(y)}\, \mu^C(c'_y) \\
    &= \mu^C(c_y)\inv\, d_{\sigma(y)}\inv\, \mu^B(\rho^B(x_y))\, d_{\sigma'(y)}\, \mu^C(c'_y) \enspace\textrm{(since $\sigma(y) = \sigma'(y)$),}
\end{align*}
as announced. The required equation for $z_h$ can be derived in the same way.

Conversely, suppose that
there exist $\bfz \in C(\mu^C\circ\sigma^C)$ and $\bfx \in \prod_yD_{\sigma(y)} \times \prod_hD_{\sigma(h)}$ such that
\begin{align}
\label{eq: equivalent lifts with prime vertices} z_y  &= \mu^C(c_y\inv\,\rho^C(x_y)\,c'_y) = \mu^C(c_y)\inv\, d_{\sigma(y)}\inv\, \mu^B(\rho^B(x_y))\, d_{\sigma'(y)}\, \mu^C(c_y'),\\
\label{eq: equivalent lifts with prime edges} z_h  &= \mu^C(c_h\inv\,\rho^C(x_h)\,c'_h) = \mu^C(c_h)\inv\, d_{\sigma(h)}\inv\, \mu^B(\rho^B(x_h))\, d_{\sigma'(h)}\, \mu^C(c_h').
\end{align}
It follows that
\begin{align*}
d_{\sigma(y)} &= \mu^B(\rho^B(x_y))\, d_{\sigma'(y)}\, \mu^C(\rho^C(x_y))\inv \in \mu^B(B_{\sigma^B(y)})\, d_{\sigma'(y)}\, \mu^C(C_{\sigma^C(y)}),\\
d_{\sigma(h)} &= \mu^B(\rho^B(x_h))\, d_{\sigma'(h)}\, \mu^C(\rho^C(x_h))\inv \in \mu^B(B_{\sigma^B(h)})\, d_{\sigma'(h)}\, \mu^C(C_{\sigma^C(h)}),
\end{align*}
that is, $\sigma(y) = \sigma'(y)$ and $\sigma(h) = \sigma'(h)$. Thus $\sigma$ and $\sigma'$ agree on the underlying graph.

Now, let $\bfz'$ be given by $z'_y = (z_y)^{\mu^C(c_y)\inv}, z'_h = (z_h)^{\mu^C(c_h)\inv}$ for each $y \in V(\gr D')$, $h\in E(\gr{D}')$. Then $\bfz' \in C(\mu^C\circ\rho^C\circ\sigma)$ (Lemma~\ref{lem: almost uniqueness of associated elements}) and Equation~\eqref{eq: equivalent lifts with prime vertices} yields
\begin{align*}
    &z'_y = \mu^C(c_y)\,z_y\, \mu^C(c_y)\inv = d_{\sigma(y)}^{-1}\, \mu^B(\rho^B(x_y))\, d_{\sigma(y)}\, \mu^C(c_y'\,c_y\inv)) \textrm{ and so}\\
    &z'_y\, \mu^C(c_y\,c_y'^{-1}) = \mu^C(\rho^B(x_y)) =  \mu^B(\rho^B(x_y))^{d_{\sigma(y)}}.
\end{align*}
A similar equation can be derived for the edges. Thus Equations~\eqref{eq: equivalent lifts vertices} and \eqref{eq: equivalent lifts edges} hold, and we have that $\sigma \approx \sigma'$ with $\bfx \in P(\sigma, \sigma')$.
\end{proof}

\begin{rem}
    Note that the existence condition in Proposition \ref{prop: equivalence characterisation} does not depend on the particular choice of parameters $\bfc$ and $\bfc'$, since all other parameters lie in $C(\sigma^C)\,\bfc$ and $C(\sigma^C)\,\bfc'$, and $\mu^C(C(\sigma^C))\leqslant C(\mu^C\circ\sigma^C)$.
\end{rem}

\begin{thm}
\label{thm: lift characterisation}
    Let $\sigma\colon \DD'\to \DD$ be a lift and let $\bfb \in P(\rho^B \circ \sigma, \sigma^B)$, $\bfc \in P(\rho^C\circ\sigma, \sigma^C)$ and denote by $\bfd' = ((d_{\sigma(y)})_y, (d_{\sigma(h)})_h)$. If $\bfz = \bfa^{-1}\, \mu^B(\bfb)\, \bfd'\, \mu^C(\bfc)$, then $\bfz\in C(\mu^C\circ\sigma^C)$ and $\sigma = \sigma^{\bfz}$, where $\sigma^{\bfz}$ denotes the lift from Proposition~\ref{prop: existence 2}.
    
    Let $y_0$ be a vertex of $\gr D'$, and let $\bfz^1, \bfz^2 \in C(\mu^C\circ\sigma^C)$. Then $\sigma^{\bfz^1}$ and $\sigma^{\bfz^2}$ agree on $\gr D'$ if and only if they agree on $y_0$, if and only if
 there exist $\hat b \in B_{\sigma^B(y_0)},$ and $\hat c \in C_{\sigma^C(y_0)}$ such that
 $$\mu^C(\hat c) = \left(a_{y_0}\, z^1_{y_0}\right)^{-1}\, \mu^B(\hat b)\, \left(a_{y_0}\, z_{y_0}^2\right).$$
    Finally, $\sigma^{\bfz^1}\approx\sigma^{\bfz^2}$ if and only if
    there exist $\bfz\in C(\mu^C\circ\sigma^C)$ and tuples $(\hat b_y)_y$ and $(\hat c_y)_y$ such that, for each vertex $y \in V(\gr D')$, $\hat b_y\in B_{\sigma^B(y)}$, $\hat c_y \in C_{\sigma^C(y)}$ and
    \begin{equation}\label{eq: eq d1 d2}
    z_y = \mu^C(\hat c_y) = \left(a_y\, z^1_y\right)^{-1}\, \mu^B(\hat b_y)\, \left(a_y\, z_y^2\right),
    \end{equation}
    if only if Equation~\eqref{eq: eq d1 d2} holds for at least one vertex $y$ of $V(\gr D')$. In that case, letting $x_y = (\hat b_y,\hat c_y)$ for each $y$ (an element of $D_{\sigma^{\bfz^1}(y)}$), we have $(x_y)_y \in P_V(\sigma^{\bfz^1}, \sigma^{\bfz^2})$. 
\end{thm}

\begin{proof}
By Corollary~\ref{cor: composition associated elements} and Lemma \ref{lem: composition}, we have $\bfd' \in P(\mu^B\circ\rho^B\circ\sigma, \mu^C\circ\rho^C\circ\sigma)$. Note that the equivalence $\mu^C \circ \sigma^C \approx \mu^C \circ \sigma^C$ is witnessed by the sequence of equivalences
$$\mu^C\circ\sigma^C \approx \mu^B\circ\sigma^B \approx \mu^B\circ\rho^B\circ\sigma \approx \mu^C\circ\rho^C\circ\sigma \approx \mu^C\circ\sigma^C.$$
It follows (see Lemma~\ref{lem: double equivalence elements}) that $\bfa\inv\, \mu^B(\bfb)\inv\, \bfd'\, \mu^C(\bfc) \in C(\mu^C \circ \sigma^C)$.

For each vertex $y$ and edge $h$, we have
\begin{align*}
d_{\sigma(y)} &= \mu^B(b_y)\, (a_y\, z_y)\, \mu^C(c_y)\inv \enspace\in\enspace \mu^B(B_{\sigma^B(y)})\, (a_y\, z_y)\, \mu^C(C_{\sigma^C(y)}) = \sigma^{\bfz}(y),\\
d_{\sigma(h)} &= \mu^B(b_h)\, (a_h\, z_h)\, \mu^C(c_h)\inv \enspace\in\enspace \mu^B(B_{\sigma^B(h)})\, (a_h\, z_h)\, \mu^C(C_{\sigma^C(h)}) = \sigma^{\bfz}(h).
\end{align*}
Thus, $\sigma$ and $\sigma^{\bfz}$ agree on $\gr D'$. From Proposition~\ref{prop: existence 2}, it follows that $\bfb \in P(\rho^B\circ\sigma^{\bfz}, \sigma^B)$ and $\bfc \in P(\rho^C\circ\sigma^{\bfz}, \sigma^C$). In particular, $\bfc$ sits in both $P(\rho^C\circ\sigma, \sigma^C)$ and $P(\rho^C\circ\sigma^{\bfz}, \sigma^C)$.

Now let $\bfx$ be given by $x_y = (1,1) \in D_{\sigma(y)}$ and $x_h = (1, 1)\in D_{\sigma(h)}$ for each vertex $y$ and edge $h$. Then Proposition~\ref{prop: equivalence characterisation} shows that $\bfx \in P(\sigma, \sigma^{\bfz})$, and hence $\sigma = \sigma^{\bfz}$ by Remark~\ref{rem: basics on equivalence}.

The statement about $\sigma^{\bfz^1}$, $\sigma^{\bfz^2}$ agreeing on vertices follows from Proposition \ref{prop: agree on graph} combined with the fact that there exist $\hat b \in B_{\sigma^B(y_0)}$ and $\hat c \in C_{\sigma^C(y_0)}$ such that
$$\mu^C(\hat c) = \left(a_{y_0}z^1_{y_0}\right)^{-1}\mu^B(\hat b)\left(a_{y_0}z_{y_0}^2\right)$$
if and only if $a_{y_0}z^1_{y_0} \in \mu^B(B_{\sigma^B(y_0)})\, (a_{y_0}z_{y_0}^2)\, \mu^C(C_{\sigma^C(y_0))})$,
that is, if and only if $\sigma^{\bfz^1}(y_0) = \sigma^{\bfz^2}(y_0)$.

For the last statement, first note that we just proved that, for each $y \in V(\gr D')$,
the existence of $\hat b \in B_{\sigma^B(y)}$ and $\hat c \in C_{\sigma^C(y)}$ such that $\mu^C(\hat c) = \left(a_{y}\, z^1_{y}\right)^{-1}\, \mu^B(\hat b)\, \left(a_{y}\, z_{y}^2\right)$ implies that $\sigma^{\bfz^1}$ and $\sigma^{\bfz^2}$ coincide on the underlying graph.

Then the result follows from Proposition \ref{prop: equivalence characterisation} and the fact that, for each vertex $y$ and edge $h$ and for $i = 1, 2$, $z_y^i = a_y^{-1}\mu^B(b_y^i)\inv d_{\sigma^{\bfz^i}(y)}\mu^C(c_y^i)$, $z_h^i = a_h^{-1}\mu^B(b_h^i)\inv d_{\sigma^{\bfz^i}(h)}\mu^C(c_h^i)$ (by Proposition \ref{prop: existence 2}), where $\bfb^i \in P(\rho^B\circ\sigma^{\bfz^i}, \sigma^B)$ and $\bfc^i \in P(\rho^C\circ\sigma^{\bfz^i}, \sigma^C)$.
\end{proof}

\subsection{Pullbacks of pointed graphs of groups}
\label{sec: final pointed pullbacks}

The main result of this section, Theorem~\ref{thm: pullbacks in GrGp*} below, establishes that pullbacks exist in the category $\GrGp^*$ of pointed graphs of groups with connected underlying graph, and gives an explicit construction for these pullbacks. This explicit construction is essential in the group-theoretic applications that are exploited in \cite{dllrw_fgip}.

\begin{thm}\label{thm: pullbacks in GrGp*}
Pullbacks exist in the category $\GrGp^*$. More precisely, if $\mu^B
\colon (\BB,v_0) \to (\AA,u_0)
$ and $\mu^C\colon (\CC,w_0) \to (\AA,u_0)$ are morphisms between pointed connected graphs of groups, then their pointed $\AA$-product is the pullback of $\mu^B$ and $\mu^C$.
\end{thm}

\begin{proof}
Let $(\AA,u_0)$, $(\BB,v_0)$, $(\CC, w_0)$ be connected pointed graphs of groups, and let $\mu^B\colon (\BB,v_0)\to (\AA,u_0)$ and $\mu^C\colon (\CC,w_0)\to (\AA,u_0)$ be morphisms. Let $(\DD,x_0) = (\BB,v_0) \wtimes_\AA (\CC,w_0)$ be their pointed $\AA$-product, with projection morphisms $\rho^B$ and $\rho^C$.

Let $(\DD',y_0)$ be a connected pointed graph of groups and let $\sigma^B\colon (\DD',y_0)\to (\BB,v_0)$ and $\sigma^C\colon (\DD',y_0)\to (\CC,w_0)$ be morphisms (as in Figure~\ref{fig: for the lifts section}) such that $\mu^B\circ\sigma^B \sim \mu^C\circ\sigma^C$. Let $\bfa \in P^*(\mu^B\circ\sigma^B, \mu^C\circ\sigma^C)$. 

We need to show that $(\sigma^B, \sigma^C)$ admits a lift, which is unique up to $\sim$-equivalence. Let $\bfz^1 = ((1)_y, (1)_h)$. Then $\bfz \in C(\mu^C\circ\sigma^C)$ and the morphism $\sigma^{\bfz}$ defined in Proposition~\ref{prop: existence 2}, is a lift of $(\sigma^B, \sigma^C)$. Moreover, we have $\sigma(y_0) = \mu^B(B_{v_0})\,(a_{y_0}\,z_{y_0})\, \mu^C(C_{w_0}) = \mu^B(B_{v_0})\, \mu^C(C_{w_0}) = x_0$.

Finally, if $\sigma\colon (\DD',y_0) \to (\DD, x_0)$ is another lift, we want to show that $\sigma \sim \sigma^{\bfz}$. Let $\bfb\in P^*(\rho^B \circ \sigma, \sigma^B)$ and $\bfc\in P^*(\rho^C\circ\sigma, \sigma^C)$. Let $\bfz' = \bfa\inv\, \mu^B(\bfb)\, \bfd'\, \mu^C(\bfc)$ where $\bfd' = ((d_{\sigma(x)})_x, (d_{\sigma(h)})_h)$. Theorem~\ref{thm: lift characterisation} establishes that $\bfz'\in C(\mu^C\circ\sigma^C)$ and, since $z'_{y_0} = 1$, Lemma~\ref{lem: trivial element of centraliser} shows that $\bfz' = \bfz$. Theorem~\ref{thm: lift characterisation} again states that $\sigma = \sigma^{\bfz}$ and this concludes the proof.
%
\end{proof}

Theorem~\ref{thm: pullbacks in GrGp*} also yields a description of pullbacks in the subcategory of pointed graphs of groups with immersions.

\begin{cor}\label{cor: it is a pullback}
Let $\mu^B\colon(\BB, v_0)\to(\AA, u_0)$ and $\mu^C\colon (\CC, w_0)\to(\AA, u_0)$ be immersions of pointed connected graphs of groups. Let $(\DD,x_0)$ be their pointed $\AA$-product (and their pullback in $\GrGp^*$), with projection morphisms $\rho^B$ and $\rho^C$.

Then $((\DD, x_0), \rho^B, \rho^C)$ is their pullback in the category of pointed connected graphs of groups with immersions.
\end{cor}

\begin{proof}
Let $\sigma^B\colon (\DD',y_0)\to(\BB,v_0)$ and $\sigma^C\colon (\DD',y_0)\to(\CC,w_0)$ be immersions such that $\mu^B\circ \sigma^B \sim \mu^C \circ \sigma^C$, and let $\sigma\colon(\DD',y_0) \to (\DD,x_0)$ be the morphism induced by the fact that $(\DD,x_0)$ is the pullback of $\mu^B$ and $\mu^C$ (in $\GrGp^*$). In particular, $\rho^B \circ \sigma \sim \sigma^B$ and $\rho^C \circ \sigma \sim \sigma^C$.

We already observed that $\rho^B$ and $\rho^C$ are immersions (Lemma~\ref{lem: rhos are folded 2}). It then follows from Corollary~\ref{cor: composition of folded} that $\sigma$ is an immersion as well, thus concluding the proof.
\end{proof}

Finally, we record an important application of the construction of pullbacks to the intersection of subgroups of $\pi_1(\AA, u_0)$. We start with a technical statement of independent interest, about lifting a pair of a $\BB$-path and a $\CC$-path to a $\DD$-path.

\begin{cor}\label{cor: lifting}
Let $\mu^B\colon \BB \to \AA$ and $\mu^C\colon \CC \to \AA$ be morphisms of graphs of groups, and let $\DD = \BB\wtimes_\AA\CC$ be their $\AA$-product. Let $u, v, w$ (resp. $u', v', w'$) be vertices of $\gr A$, $\gr B$ and $\gr C$, respectively, such that $[v] = [w] = u$ (resp. $[v'] = [w'] = u'$). Let $x$ be a $(v,w)$-vertex of $\gr D$ and let $x'$ be a $(v',w')$-vertex.

If $r$ is a $\DD$-path from $x$ to $x'$, then $\mu^B(\rho^B(r)) \sim_\AA (d_x)\,\mu^C(\rho^C(r))\, (d_{x'})\inv$. Conversely, if there exists a $\BB$-path $q^B$ from $v$ to $v'$ and a $\CC$-path $q^C$ from $w$ to $w'$ such that $\mu^B(q^B) \sim_\AA (d_x)\,\mu^C(q^C)\, (d_{x'})\inv$, then there exists a $\DD$-path $r$ from $x$ to $x'$ such that $q^B \sim_\BB \rho^B(r)$ and $q^C \sim_\CC \rho^C(r)$.
\end{cor}

\begin{proof}
 Let $r$ be a $\DD$-path from $x$ to $x'$. Then $\mu^B(\rho^B(r)) \sim_\AA (d_x)\,\mu^C(\rho^C(r))\,(d_{x'})\inv$ by Lemma~\ref{lem: composition}~\eqref{item: path level equivalence}.

Conversely, suppose that $q^B$ and $q^C$ satisfy $\mu^B(q^B) \sim_\AA (d_x)\,\mu^C(q^C)\, (d_{x'})\inv$. In particular, $q^B$ and $q^C$ have the same length $k$, say, $q^B = (b_0, f_1, b_1, \dots, f_k, b_k)$ and $q^C = (c_0, g_1, c_1, \dots, g_k, c_k)$.

If $k = 0$, we have $\mu^B(b_0)^{d_x} = \mu^C(c_0)$, and hence $d_0 = \mu^C(c_0)$ is an element of $D_x$. The path $r = (d_0)$ then satisfies $\rho^B(r) = q^B$ and $\rho^C(r) = q^C$.

Let us now assume that $k > 0$. Let $\DD'$ be the graph of groups whose underlying graph is a path of length $k$, say, $q = (h_1,\dots, h_k)$, with trivial vertex groups. We define a morphism $\sigma^B\colon \DD' \to \BB$ as follows: $\sigma^B$ maps the underlying graph of $\DD'$ to the underlying path of $q^B$; the twisting elements are $(h_i)^B_\alpha = b_{i-1}$ for every $i\in [1,k]$, $(h_i)^B_\omega = 1$ for every $i\in [1,k-1]$ and $(h_k)^B_\omega = b_k\inv$. We then have $\sigma^B(q) = q^B$. In the same way, we define a morphism $\sigma^C\colon \DD' \to \CC$ such that $\sigma^C(q) = q^C$. By applying Lemma \ref{lem: equivalent A-paths} to the equivalence $\sigma^B\circ\mu^B(q)\sim_{\AA} (d_x)\,(\sigma^C\circ\mu^C(q)) \,(d_{x'})\inv$ we see that $\sigma^B\circ\mu^B\approx\sigma^C\circ\mu^C$. By Proposition \ref{prop: existence 2}, if $\DD$ is the $\AA$-product of $\mu^B$ and $\mu^C$ (with projection morphisms $\rho^B$ and $\rho^C$), there exists a morphism $\sigma^{(1)}\colon \DD'\to\DD$ such that $\sigma^B = \rho^B\circ\sigma^{(1)}$ and $\sigma^C = \rho^C\circ\sigma^{(1)}$. Hence, the path $r = \sigma^{(1)}(q)$ satisfies $q^B\sim_{\BB}\rho^B(r)$ and $q^C\sim_{\CC}\rho^C(r)$ as required.
\end{proof}

This yields the following important corollary on the intersection of subgroups (for the purpose of which we can restrict our attention to immersions of graphs of groups).

\begin{cor}\label{cor: intersection of subgroups}
Let $\mu^B\colon \BB \to \AA$ and $\mu^C\colon \CC \to \AA$ be morphisms of graphs of groups, and let $\DD$ be the $\AA$-product of $\mu^B$ and $\mu^C$. 
\begin{enumerate}[(1)]
\item\label{item: intersection of conjugates} Let $v\in V(\gr B)$ and $w\in V(\gr C)$ such that $[v] = [w]$, and let $x\in V(\gr D)$ be a $(v,w)$-vertex. Then
    \[
    \left((\mu^B\circ\rho^B)_*(\pi_1(\DD, x))\right)^{d_x} = (\mu^C\circ\rho^C)_*(\pi_1(\DD, x)) \leqslant \left(\mu^B_*(\pi_1(\BB, v))\right)^{d_x} \cap \mu^C_*(\pi_1(\CC, w)),
    \]
    with equality if $\mu^B$ and $\mu^C$ are immersions.

\item\label{item: intersection of subgroups} Let $u_0$, $v_0$ and $w_0$ be vertices of $\gr A$, $\gr B$ and $\gr C$, respectively, such that $[v_0] = [w_0] = u_0$. Let $(\DD,x_0)$ be the $\AA$-product of $\mu^B$ and $\mu^C$, seen as morphisms of pointed graphs of groups. Then
    \[
    (\mu^B\circ\rho^B)_*(\pi_1(\DD, x_0)) = (\mu^C\circ\rho^C)_*(\pi_1(\DD, x_0)) \leqslant \mu^B_*(\pi_1(\BB, v_0))\cap \mu^C_*(\pi_1(\CC, w_0)),
    \]
    with equality if $\mu^B$ and $\mu^C$ are immersions.

\item \label{item: intersection_component} Let $v, w, x$ be as in \eqref{item: intersection of conjugates}, let $p_v$ be a reduced $\BB$-path from $v_0$ to $v$ and let $q_w$ be a reduced $\CC$-path from $w_0$ to $w$. Then we have
    \[
    \mu^C_*\circ\rho^C_*(\pi_1(\DD, x))^{\mu^C(q_w)^{-1}} \le \mu^B_*(\pi_1(\BB, v_0))^{\left(\mu^B(p_v)d_{x}\mu^C(q_w)^{-1}\right)}\cap \mu^C_*(\pi_1(\CC, w_0)),
    \]
    with equality if $\mu^B$ and $\mu^C$ are immersions.
\end{enumerate}
\end{cor}

\begin{proof}
We start with Statement~\eqref{item: intersection of conjugates}. By Lemma~\ref{lem: composition}~\eqref{item: path level equivalence}, we have $(\mu^B\circ\rho^B)_*(\pi_1(\DD,x))^{d_x} = (\mu^C\circ\rho^C)_*(\pi_1(\DD,x))$, which is clearly contained in $\mu^B_*(\pi_1(\BB,v))^{d_x} \cap \mu^C_*(\pi_1(\CC,w))$.

To establish the converse inclusion when $\mu^B$ and $\mu^C$ are immersions, we need to show that if $q^B$ is a reduced $\BB$-circuit at $v$ and $q^C$ is a reduced $\CC$-circuit at $w$ such that $\mu^B(q^B)^{d_x} =_\AA \mu^C(q^C)$, then there exists a reduced $\DD$-circuit $r$ at $x$ such that $\rho^B(r) =_\BB q^B$ and $\rho^C(r) =_\CC q^C$. Since $\mu^B$ and $\mu^C$ are immersions, the $\BB$- and $\CC$-paths $\mu^B(q^B)$ and $\mu^C(q^C)$ are reduced, and  $\mu^B(q^B)^{d_x} =_\AA \mu^C(q^C)$ implies $\mu^B(q^B)^{d_x} \sim_\AA \mu^C(q^C)$ (by Proposition~\ref{prop: sim vs equiv}~\eqref{eq: for reduced, sim is =}). We can now conclude, using Corollary~\ref{cor: lifting}.

Statement~\eqref{item: intersection of subgroups} is proved in the same fashion, with the additional consideration that ${d_{x_0} = 1}$.

Finally, we verify Statement~\eqref{item: intersection_component}. By~\eqref{item: intersection of conjugates}, we have
    \begin{align*}
    \mu_*^C\circ\rho_*^C(\pi_1(\DD, x)) &\le \mu^B_*(\pi_1(\BB, v))^{d_{x}}\cap \mu^C_*(\pi_1(\CC, w)) \\
    								&\le \mu^B_*(\pi_1(\BB, v_0))^{\mu^B(p_v)d_{x}}\cap \mu^C_*(\pi_1(\CC, w_0))^{\mu^C(q_w)}.
    \end{align*}
    The desired formula follows, once we conjugate both sides by $\mu^C(q_w)^{-1}$.
\end{proof}

\subsection{Pullbacks in the category of (unpointed) graphs of groups}\label{sec: final unpointed pullback}

In contrast with the category of pointed graphs of groups, pullbacks do not always exist in $\GrGp$, as we will see in Example~\ref{ex: pullbacks may not exist}. However, if $\mu^B\colon \BB \to \AA$ and $\mu^C\colon \CC \to \AA$ are morphisms of graphs of groups, we can still derive information about the pullback of $\mu^B$ and $\mu^C$, if it exists. Moreover, in specific instances we may use Theorem \ref{thm: lift characterisation} to determine precisely when it exists.

\begin{thm}
\label{thm: pullbacks_exist_sometimes}
Let $\mu^B\colon \BB\to \AA$ and $\mu^C\colon \CC\to \AA$ be
morphisms of graphs of groups, and let $\BB \wtimes_\AA\CC$ be their $\AA$-product, with projections $\rho^B$ and $\rho^C$ onto $\BB$ and $\CC$, respectively. If the pullback $(\BB\times_{\AA}\CC, \sigma^B, \sigma^C)$ of $\mu^B$ and $\mu^C$ exists in $\GrGp$, then it is a subgraph of groups of $\BB\wtimes\AA$.

More precisely, there exists a morphism of graphs of groups $\sigma\colon\BB\times_{\AA}\CC\to \BB\wtimes_{\AA}\CC$ such that $\sigma$ is an isomorphism onto its image and such that $\rho^B\circ\sigma\approx \sigma^B$ and $\rho^C\circ\sigma \approx\sigma^C$.
\end{thm}

\begin{proof} 
Since $\mu^B\circ\rho^B \approx \mu^C\circ\sigma^C$, there exists a unique morphism $\sigma'\colon\BB\wtimes_{\AA}\CC\to \BB\times_{\AA}\CC$ such that $\sigma^B\circ\sigma'\approx\rho^B$ and $\sigma^C\circ\sigma'\approx\rho^C$ (by the definition of pullbacks).

Proposition~\ref{prop: existence 2} establishes also the existence of a lift $\sigma\colon \BB\times_{\AA}\CC \to \BB\wtimes_{\AA}\CC$, such that $\rho^B\circ\sigma \approx \sigma^B$ and $\rho^C\circ \sigma \approx\sigma^C$.
It follows that $\sigma^B\circ(\sigma'\circ\sigma) \approx \sigma^B$ and $\sigma^C\circ(\sigma'\circ\sigma) \approx \sigma^C$: by the uniqueness in the definition of pullbacks, it follows that $\sigma'\circ\sigma \approx 1_{\BB\wtimes_{\AA}\CC}$ and, hence, $\BB\times_\AA\CC$ is isomorphic in $\GrGp$ to a subgraph of groups of $\BB\wtimes_{\AA}\CC$, as announced.
\end{proof}

\begin{exm}[Pullbacks sometimes do not exist in $\GrGp$]\label{ex: pullbacks may not exist}
    Let $\AA$ be the graph of groups with underlying graph having a single vertex $u$ and three edges $e_1, e_2, e_3$, vertex group $A_u = \langle a\rangle \cong \Z$, edge groups $A_{e_1}, A_{e_3} = 1$, $A_{e_2} = \Z$ and identity edge maps $\alpha_{e_2}=\omega_{e_2} = \id_{\ZZ}$. For convenience, we write $e_i$ for the element $(1,e_i,1)$ of $\pi_1(\AA,u)$. One can see directly that
    \[
    \pi_1(\AA, u) \cong \langle e_1, e_3 \mid\enspace \rangle *\langle a, e_2 \mid [a, e_2] = 1\rangle \cong F_2*\Z^2.
    \]
    Consider the following subgroups: 
    \begin{align*}
        B &= \langle a e_1 a^{-1}, e_2, e_3\rangle\\
        C &= \langle e_1, e_2, a e_3 a^{-1}\rangle.
    \end{align*}
    Let $\mu^B\colon \BB\to \AA$ and $\mu^C\colon \CC\to \AA$ be an immersion of graphs of groups whose fundamental group images are $B$ and $C$, respectively. Both $\BB$ and $\CC$ have single vertices $v$ and $w$ with $[v] = [w] = u$ and three edges $f_1, f_2, f_3$ and $g_1, g_2, g_3$ with $[f_i] = [g_i] = e_i$ for $i = 1, 2, 3$. 

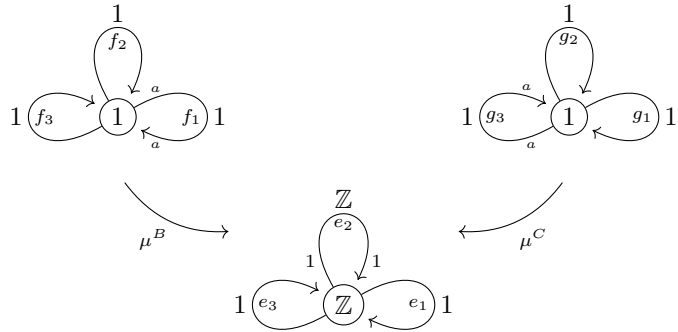
\begin{figure}[htbp]
\centering
\begin{tikzpicture}[shorten >=3pt, node distance=2cm and 1.75cm, on grid,auto,-latex]

\begin{scope}[yshift=1 cm]
    \node[state] (vA) {$\ZZ$};
    \node (ab) [above left = 1 and 1.3 of vA]{};
    \node (ac) [above right = 1 and 1.3 of vA]{};

    \node[] (Aab) [above = 0.25 of vA]{};
    \path[->] (vA) edge[loop right,min distance=15mm,in=330, out=30]
    node[left] {$e_1$}
    node[right] {$\text{\normalsize$1$}$} (vA);

    \path[->] (vA) edge[loop above,min distance=15mm,in=60, out=120]
    node[above] {$\text{\normalsize$\ZZ$}$}
    node[below] {$e_2$}
    node[pos=0.9,right] {$1$}
    node[pos=0.1,left] {$1$}
            (vA);

    \path[->] (vA) edge[loop left,min distance=15mm,in=150, out=210]
    node[left] {$\text{\normalsize$1$}$}
    node[right] {$e_3$} (vA);
\end{scope}

\begin{scope}[xshift=-3cm,yshift=3.5 cm]
    \node[state] (vA) {$1$};
    \node (b) [below = 0.75 of vA]{};
    \node (bb) [above left = 1.5 of vA]{};
    
    \node[] (Aab) [above = 0.25 of vA]{};
    \path[->] (vA) edge[loop right,min distance=15mm,in=330, out=30]
    node[right] {$\text{\normalsize$1$}$}
    node[left] {$f_1$}
    node[pos=0.92,below] {$\scriptstyle{a}$}
    node[pos=0.08,above] {$\scriptstyle{a}$}
            (vA);

    \path[->] (vA) edge[loop above,min distance=15mm,in=60, out=120]
    node[above] {$\text{\normalsize$1$}$}
    node[below] {$f_2$} (vA);

    \path[->] (vA) edge[loop left,min distance=15mm,in=150, out=210]
    node[left] {$\text{\normalsize$1$}$}
    node[right] {$f_3$} (vA);
\end{scope}

\begin{scope}[xshift=3cm,yshift=3.5cm]
    \node[state] (vA) {$1$};
    \node (c) [below = 0.75 of vA]{};
    \node (cc) [above right = 1.5 of vA]{};
    
    \node[] (Aab) [above = 0.25 of vA]{};
    \path[->] (vA) edge[loop right,min distance=15mm,in=330, out=30]
    node[right] {$\text{\normalsize$1$}$}
    node[left] {$g_1$} (vA);

    \path[->] (vA) edge[loop above,min distance=15mm,in=60, out=120]
    node[above] {$\text{\normalsize$1$}$}
    node[below] {$g_2$} (vA);

    \path[->] (vA) edge[loop left,min distance=15mm,in=150, out=210]
    node[left] {$\text{\normalsize$1$}$}
    node[right] {$g_3$}
    node[pos=0.9,above] {$\scriptstyle{a}$}
    node[pos=0.08,below] {$\scriptstyle{a}$} (vA);
\end{scope}

\path[->] (b) edge[bend right]
    node[below left] {$\mu^B$}
            (ab);
\path[->] (c) edge[bend left]
    node[below right] {$\mu^C$}
            (ac);
\end{tikzpicture}
\vspace{-15pt}
\caption{Immersions $\mu^B$ and $\mu^C$  realizing $B$ and $C$}
\label{fig: immersions}
\end{figure}
    
    All vertex and edge groups are also trivial and the twisting elements are $(f_1)_{\alpha} = (f_1)_{\omega} = (g_3)_{\alpha} = (g_3)_{\omega} = a^{-1}$ with all others are trivial. It is clear that $B \cong C \cong F_3$.

Now consider the $\AA$-product $\DD = \BB\wtimes_{\AA}\CC$. Its vertex set is the set of double cosets $x_i = \mu^B_v(B_v)\,a^i\,\mu^C_w(C_w) = \{a^i\}$ ($i\in \Z$), since $\mu^B_v(B_v)$ and $\mu^C_w(C_w)$ are trivial. The $(f_2,g_2)$-edges of $\DD$ are the $h_i = \{a^i\}$ ($i\in \Z$), where $h_i$ is a loop at $x_i$. In addition, $\DD$ has a single $(f_1,g_1)$-edge $h'_{-1}$ which is a loop at $x_{-1}$, and a single $(f_3,g_3)$-edge $h'_1$, which is a loop at $x_1$. Finally, the vertex and edge groups of $\DD$ are all trivial. It follows that $\DD$ has connected components $\DD_i$ ($i\in \Z$), each with a single vertex (namely $x_i$) and that $\pi_1(\DD_i,x_i)$ is freely generated by $(1,h_i,1)$ if $i\ne \pm1$, by $(1,h_1,1)$ and $(1,h'_1,1)$ if $i = 1$, and by $(1,h_{-1},1)$ and $(1,h'_{-1},1)$ if $i = -1$ (see Figure~\ref{fig: immersions and product}).

By Theorem~\ref{thm: pullbacks_exist_sometimes}, if the pullback of $\mu^B$ and $\mu^C$ exists in the unpointed category of graphs of groups, then it is a subgraph of groups of $\DD$ and there is a unique morphism $\sigma\colon \DD\to \BB\times_\AA\CC$ such that $\rho^B\circ\sigma \approx \rho^B$ and $\rho^C\circ\sigma \approx \rho^C$. In particular, $\DD_1$ and $\DD_{-1}$ must be in $\BB\times_\AA\CC$, since these components admit a unique morphism into $\DD$, compatible with $\rho^B$ and $\rho^C$. On the other hand, if $i \ne \pm1$ and $j\in \Z$, there is a morphism from $\DD_i$ to $\DD_j$, since all the $h_j$ map to $f_2$, $g_2$ and $e_2$. As a result, for there to be a unique morphism $\sigma$ from $\DD_i$ to $\BB\times_\AA\CC$ seen as a subgraph of groups of $\DD$, $\DD$ cannot contain both $h_1$ and $h_{-1}$, a contradiction. Thus $\mu^B$ and $\mu^C$ do not admit a pullback in $\GrGp$.

\begin{figure}[htbp]
\centering
\begin{tikzpicture}[shorten >=3pt, node distance=2cm and 1.75cm, on grid,auto,-latex]

\begin{scope}[yshift=1 cm]
    \node[state] (vA) {$\ZZ$};
    \node (ab) [above left = 1 and 1.3 of vA]{};
    \node (ac) [above right = 1 and 1.3 of vA]{};

    \node[] (Aab) [above = 0.25 of vA]{};
    \path[->] (vA) edge[loop right,min distance=15mm,in=330, out=30]
    node[left] {$e_1$}
    node[right] {$\text{\normalsize$1$}$} (vA);

    \path[->] (vA) edge[loop above,min distance=15mm,in=60, out=120]
    node[above] {$\text{\normalsize$\ZZ$}$}
    node[below] {$e_2$}
    node[pos=0.9,right] {$1$}
    node[pos=0.1,left] {$1$} (vA);

    \path[->] (vA) edge[loop left,min distance=15mm,in=150, out=210]
    node[left] {$\text{\normalsize$1$}$}
    node[right] {$e_3$} (vA);
\end{scope}

\begin{scope}[xshift=-3cm,yshift=3.5 cm]
    \node[state] (vA) {$1$};
    \node (b) [below = 0.75 of vA]{};
    \node (bb) [above left = 1.5 of vA]{};
    
    \node[] (Aab) [above = 0.25 of vA]{};
    \path[->] (vA) edge[loop right,min distance=15mm,in=330, out=30]
    node[right] {$\text{\normalsize$1$}$}
    node[left] {$f_1$}
    node[pos=0.9,below] {$\scriptstyle{a}$}
    node[pos=0.1,above] {$\scriptstyle{a}$}
            (vA);

    \path[->] (vA) edge[loop above,min distance=15mm,in=60, out=120]
    node[above] {$\text{\normalsize$1$}$}
    node[below] {$f_2$} (vA);

    \path[->] (vA) edge[loop left,min distance=15mm,in=150, out=210]
    node[left] {$\text{\normalsize$1$}$}
    node[right] {$f_3$} (vA);
\end{scope}

\begin{scope}[xshift=3cm,yshift=3.5cm]
    \node[state] (vA) {$1$};
    \node (c) [below = 0.75 of vA]{};
    \node (cc) [above right = 1.5 of vA]{};
    
    \node[] (Aab) [above = 0.25 of vA]{};
    \path[->] (vA) edge[loop right,min distance=15mm,in=330, out=30]
    node[right] {$\text{\normalsize$1$}$}
    node[left] {$g_1$} (vA);

    \path[->] (vA) edge[loop above,min distance=15mm,in=60, out=120]
    node[above] {$\text{\normalsize$1$}$}
    node[below] {$g_2$} (vA);

    \path[->] (vA) edge[loop left,min distance=15mm,in=150, out=210]
    node[left] {$\text{\normalsize$1$}$}
    node[right] {$g_3$}
    node[pos=0.9,above] {$\scriptstyle{a}$}
    node[pos=0.08,below] {$\scriptstyle{a}$} (vA);
\end{scope}

\path[->] (b) edge[bend right]
    node[below left] {$\mu^B$}
            (ab);
\path[->] (c) edge[bend left]
    node[below right] {$\mu^C$}
            (ac);

\begin{scope}[yshift= 6.5 cm]
    \node[state] (x0) {$1$};
    
    \node[state] (x1) [right = 1.5 of x0]{$1$};
    \node[state] (x2) [right = 1.5 of x1]{$1$};
    \node (d1) [right = 1 of x2]{$\cdots$};
    \node[state] (xi) [right = 2 of x2]{$1$};

    \node (pc) [below = 0.75 of xi]{};
    \node (d2) [right = 1 of xi]{$\cdots$};

    \node[state] (-x1) [left = 1.5 of x0]{$1$};
    \node[state] (-x2) [left = 1.5 of -x1]{$1$};
    \node (-d1) [left = 1 of -x2]{$\cdots$};
    \node[state] (-xi) [left = 2 of -x2]{$1$};

    \node (pb) [below = 0.75 of -xi]{};
    \node (-d2) [left = 1 of -xi]{$\cdots$};

    \path[->] (x0) edge[loop above,min distance=15mm,in=60, out=120]
    node[below] {$h_0$}
    node[above] {$1$} (x0);

    \path[->] (x1)
    edge[loop above,min distance=15mm,in=60, out=120]
    node[below] {$h_1$}
    node[above] {$1$} (x1);

    \path[->] (x1)
    edge[loop below,min distance=15mm,out=-60, in=-120]
    node[above] {$h'_1$}
    node[below] {$1$} (x1);

    \path[->] (x2) edge[loop above,min distance=15mm,in=60, out=120]
    node[below] {$h_2$}
    node[above] {$1$} (x2);

    \path[->] (xi) edge[loop above,min distance=15mm,in=60, out=120]
    node[below] {$h_i$}
    node[above] {$1$} (xi);

    \path[->] (-x1)
    edge[loop above,min distance=15mm,in=60, out=120]
    node[below] {$h_{\text{-1}}$}
    node[above] {$1$} (-x1);

    \path[->] (-x1)
    edge[loop below,min distance=15mm,out=-60, in=-120]
    node[above] {$h'_{\text{-}1}$}
    node[below] {$1$} (-x1);

    \path[->] (-x2) edge[loop above,min distance=15mm,in=60, out=120]
    node[below] {$h_{\text{-}2}$}
    node[above] {$1$} (-x2);

    \path[->] (-xi) edge[loop above,min distance=15mm,in=60, out=120]
    node[below] {$h_{\text{-}i}$}
    node[above] {$1$} (-xi);

\end{scope}

\path[->] (pb) edge[bend right]
    node[below left] {$\rho^B$}
            (bb);
\path[->] (pc) edge[bend left]
    node[below right] {$\rho^C$}
            (cc);

\end{tikzpicture}
\vspace{-15pt}
\caption{Immersions $\mu^B$ and $\mu^C$  realizing $B$ and $C$ and their $\AA$-product}
\label{fig: immersions and product}
\end{figure}
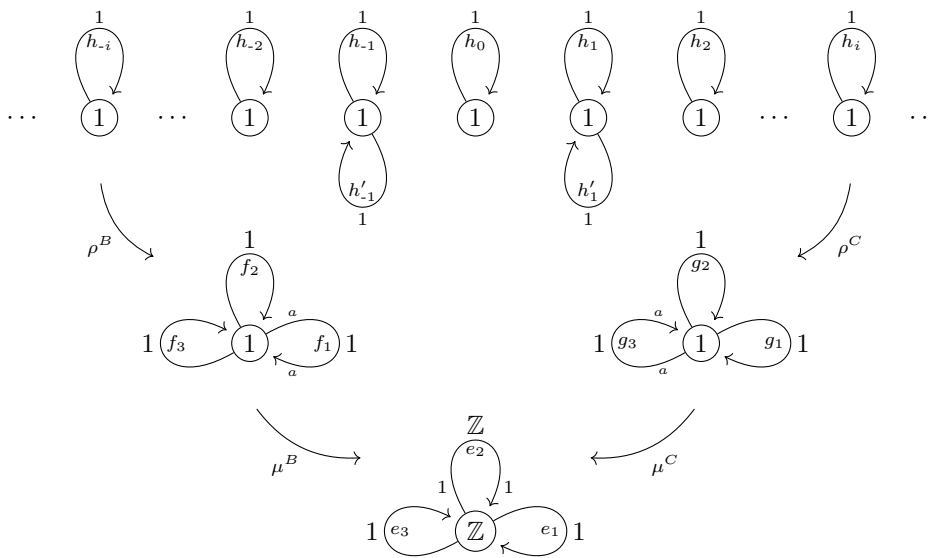
\end{exm}

\subsection{Example of an application of Corollary \ref{cor: intersection of subgroups}}\label{sec: worked out example}

A \emph{Baumslag--Solitar group} is a group of the form
\begin{equation}\label{eq: bs}\bs(m,n) = \langle a,t \mid t\inv a^mt = a^n\rangle,\end{equation}
where $m, n \in \ZZ \setminus\{ 0 \}$ are non-zero integers. 
Recall that $\bs(m,n) \cong \bs(m',n')$ if and only if $(m',n') \in \{(m,n),(n,m),(-m,-n),(-n,-m)\}$ \cite{mol91}.
Hence, without loss of generality, we can assume that $|n|\geq m \geq 1$.

It is a simple fact that $\bs(m,n)$ splits as a graph of groups with a single vertex and a single edge, with infinite cyclic vertex and edge group, see Figure~\ref{fig: BS}.
\begin{figure}[H]
\centering
\begin{tikzpicture}[shorten >=3pt, node distance=2cm and 1.75cm, on grid,auto,-latex]
\begin{scope}
    \node[state] (vA) {$\ZZ$};
    \node[] (Aab) [above = 0.25 of vA]{};
    \node[] (eA) [right = of vA]{$\ZZ$};
    \node[] (Aac) [above = 0.25 of eA]{};
    \path[->] (vA) edge[loop right,min distance=19mm,in=335, out=25]
    node[left] {$e$}
    node[pos=0.9,below] {$n$}
    node[pos=0.1,above] {$m$}
            (vA);
\end{scope}
\end{tikzpicture}
\vspace{-15pt}
\caption{Splitting of $\bs(m,n)$; for every $i\in \Z$, $\alpha_e(i) = mi$ and $\omega_e(i) = ni$}
\label{fig: BS}
\end{figure}
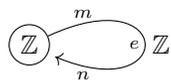

It is also well-known that $\bs(1, n)$ has the \fgip, that is, the intersection of two finitely generated subgroups of $\bs(1,n)$ is finitely generated, see \cite{mol68} for instance. Paramantzoglou showed in \cite{pa12} that all other Baumslag--Solitar groups contain $F_2\times\Z$ as a subgroup and hence do not have the \fgip\ We provide an alternative proof of this fact via an explicit computation of the pullback over the graph of groups from Figure~\ref{fig: BS}.


\begin{prop}{\cite[Proposition~1]{pa12}}
\label{prop: fgip BS}
    A Baumslag--Solitar group has the finitely generated intersection property if and only if it is of the form ${\bs(1,n)}$.
\end{prop}

\begin{proof}
Given Moldavanskiĭ's result \cite{mol68} mentioned earlier, which states that $\bs(1, n)$ has the finitely generated intersection property, we only need to show that $A = \bs(m,n)$ does not have this property if $2\leq m \leq |n|$. Indeed, consider the subgroups of $A$ given by
    \begin{align*}
        B &= \langle t, ata\inv, \ldots, a^{m-1}ta^{-(m-1)}\rangle\\
        C &= \langle at, ta\rangle.
    \end{align*}
    We claim that $B\cap C$ is not finitely generated.
    
Let $\AA$ be the graph of groups in Figure~\ref{fig: BS}, with a single vertex $u$ and a single edge $e$.
We fix a generator $a$ of $A_u = \langle a \rangle \cong \ZZ$ and denote the operation of $A_u$ multiplicatively: then, for every $i\in \ZZ$, we let $\alpha_e(i) = a^{mi}$ and $\omega_e(i) = a^{ni}$. With reference to Equation~\eqref{eq: bs}, the generator $t$ is $t = (1,e,1)$.

Now let $\BB$ be the graph of groups with a single vertex $v$, $m$ edges $f_0, \ldots,f_{m-1}$, and trivial vertex and edge groups. Let also $\mu^B$ be the morphism from $\BB$ to $\AA$, which maps each edge of $\gr B$	 to $e$, and such that $(f_j)_\alpha = (f_j)_\omega = a^j$ for each $j \in [0,m-1]$. It is immediate that $\pi_1(\BB,v)$ is freely generated by the $b_j = (1,f_j,1)$ ($j \in [0,m-1]$). Moreover, $\mu^B_*(b_j) = ((f_j)_\alpha,e,(f_j)_\omega\inv) = (a^j,e,a^{-j}) = a^jta^{-j}$ for each $j \in [0,m-1]$. So $\mu^B(\pi_1(\BB,v)) = B$.
    \begin{figure}[H]
\centering
\begin{tikzpicture}[shorten >=3pt, node distance=2cm and 1.75cm, on grid,auto,-latex]
\begin{scope}
    \node[state] (vA) {$\ZZ$};
    \node[] (Aab) [above left = 0.35 of vA]{};
    \node[] (Ab) [below right = 1 and 0.75 of vA]{$\AA$};
    \node[] (eA) [right = of vA]{$\ZZ$};
    \node[] (Aac) [above right = 0.35 of eA]{};
    
    \path[->] (vA) edge[loop right,min distance=19mm,in=335, out=25]
    node[left] {$e$}
    node[pos=0.85,below] {$n$}
    node[pos=0.15,above] {$m$} (vA);
\end{scope}

\begin{scope}[xshift=-1.5cm,yshift=2 cm]
    \node[state] (vB) {$1$};
    \node[] (Ba) [above = 0.8 of vB]{$\BB$};
    \node[] (Bb) [below = 0.5 of vB]{};
    \node[] (eB) [right = of vB]{$1$};
    \node[] (lB) [left = 1.25 of vB]{$\scriptstyle{j \,\in\, [0\,,\,m-1]}$};

    \path[->] (vB) edge[loop right,min distance=20mm,in=25, out=335]
    node[left] {$f_j$}
    node[pos=0.95,above] {$a^{j}$}
    node[pos=0.05,below] {$a^{j}$} (vB);
\end{scope}
\begin{scope}[xshift=3cm,yshift=2 cm]
    \node[state] (vC) {$1$};
    \node[] (Ca) [above = 0.8 of vC]{$\CC$};
    \node[] (Cb) [below = 0.5 of vC]{};
    \node[] (eC) [right = of vC]{$1$};
    \node[] (eeC) [left = of vC]{$1$};
    
    \path[->] (vC) edge[loop right,min distance=20mm,in=25, out=335]
    node[left] {$g_{_+}$}
    node[pos=0.95,above] {$a\inv$}
    node[pos=0.05,below] {} (vC);
     \path[->] (vC) edge[loop left,min distance=20mm,in=205, out=155]
    node[right] {$g_{_-}$}
    node[pos=0.95,below] {}
    node[pos=0.05,above right] {$a$} (vC);
\end{scope}
 \path[->] (Bb) edge[bend right] 
    node[pos=0.5,below left]{$\mu^{B}$}
    (Aab);
 \path[->] (Cb) edge[bend left] 
    node[pos=0.5,below right]{$\mu^{C}$}
    (Aac);
\end{tikzpicture}
\caption{Morphisms of graphs of groups for subgroups $B$ and $C$ of $\bs(m,n)$}
\label{fig: hom bs(m,n)}
\end{figure}
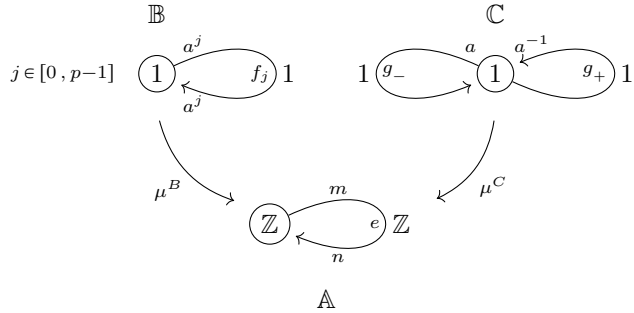
Finally, let $\CC$ be the graph of groups with a single vertex $w$, two edges $g_+$ and $g_-$, and trivial vertex and edge groups. Again, $\pi_1(\CC,w)$ is freely generated by $c_{-} = (1,g_-,1)$ and $c_{+} = (1,g_+,1)$. Let now $\mu^C$ be the morphism from $\CC$ to $\AA$, which maps $g_-$ and $g_+$ to $e$, and such that $(g_{-})_{\alpha} = a$, $(g_+)_\omega = a\inv$ and $(g_-)_\omega = (g_+)_\alpha = 1$. Then $\mu^C_*(c_{-}) = \mu^C_*((1,g_-,1)) = (a,e,1) = at$, and $\mu^C_*(c_{+}) = \mu^C_*((1,g_+,1)) = (1,e,a) = ta$. Thus $\mu^C_*(\pi_1(\CC,w)) = C$. In addition, both morphisms $\mu^B$ and $\mu^C$ are immersions.

By definition, the vertex set of $\BB\wtimes_\AA\CC$ is the set of $(1,1)$-double cosets in $A_u$, which we can identify with $A_u = \langle a \rangle \cong \ZZ$.
For simplicity, we denote the vertices of $\BB\wtimes_\AA\CC$ as integers.
Similarly, the edge set of $\BB\wtimes_\AA\CC$ can be identified with
\begin{align*}
E(\gr B\times_{\gr{A}}\gr C)\times (\dblcoset{1_A}{A_e}{1_A})
= E(\gr B) \times E(\gr C) \times A_e
= \{f_0, \dots, f_{m-1}\} \times \{g_-, g_+\} \times \Z.
\end{align*}
In view of~\eqref{eq: incidence in D}, the incidence relation in the underlying graph is given, for every $0\le j < m$ and $i\in\Z$, by
 \begin{align*}
o(f_j, g_{+}, i) &= {mi + j} &&o(f_j, g_{-}, i) = {mi + j - 1} \textrm{ and}\\
t(f_j, g_{+}, i) &= {ni + j + 1} &&t(f_j, g_{-}, i) = {ni + j}.
\end{align*}
And the vertex and edge groups of $\BB\wtimes_\AA\CC$ are trivial.

Let us say that an edge is a $g_+$-edge if it is of the form $(f_j,g_+,i)$, and a $g_-$-edge if it is of the form $(f_j,g_-,i)$. For each $x \in \Z$, denote by $q(x)$ and $r(x)$ the quotient and remainder in the Euclidean division of $x$ by $m$, so that $x = m\,q(x) + r(x)$ and $0 \le r(x) < m$. From each vertex $x\in \Z$, there originates a unique $g_+$-edge and a unique $g_-$-edge, whose parameters are determined by $q(x)$ and $r(x)$.

In particular, for every $j \in [0,m-2]$ and $i\in \ZZ$, the graph of $\BB\wtimes_\AA\CC$ has a $g_+$-edge and a $g_-$-edge from ${mi+j}$ to ${ni+j+1}$. Thus, every integer $z$ not congruent to $m-1$ mod $m$ sits in a non-trivial circuit of length  $1$ (if $mi+j = ni+j+1$, that is, if $n = m+1$ and $i = -1$) or $2$, see Figure~\ref{fig: length 2 circuitt}.
\begin{figure}[htb]
\centering
\begin{tikzpicture}[shorten >=3pt, node distance=2cm and 1.75cm, on grid,auto,-latex]
\begin{scope}
    \node (1) {$mi+j$};
    \node (2) [right = 4 of 1]{$ni+j+1$};

    \path[->] (1) edge[bend left =15] 
    node[pos=0.5,above]{$(f_{j},g_{+},i)$}
    (2);
    \path[->] (1) edge[bend left=-15] 
    node[pos=0.5,below]{$(f_{j+1},g_{-},i)$}
    (2);
\end{scope}
\end{tikzpicture}
\caption{Circuit in $\BB\wtimes_\AA\CC$, where $i\in \Z$ and $j\in [0,m-2]$ (degenerate if $n = m+1$ and $i = -1$)}
\label{fig: length 2 circuitt}
\end{figure}

We first show that every vertex $z$ congruent to $m-1$ modulo $m$ sits within distance at most $1$ from another which is not congruent to $m-1$ modulo $m$, and hence $z$ lies at distance at most 1 from a cyclically reduced circuit (of length $1$ or $2$).

Indeed, let $z = m\,q(z) + r(z)$ with $r(z) = m-1$. The $g_+$-edge out of $z$ leads to $z^+ = q(z)\,n+m$. Suppose that $z^+$ is congruent to $m-1$ modulo $m$. Then $q(z)\,n \equiv -1 \bmod m$. In particular, $n$ and $m$ are relatively prime and $-q(z)$ is the inverse of $n$ modulo $m$. Now the $g_-$-edge out of $z$ goes to $z^- = (q(z)+1)\,n$. If this is also $-1$ modulo $m$, then $-(q(z)+1)$ is the inverse of $n$ modulo $m$, a contradiction. Thus at most one of $z^+$ and $z^-$ is congruent to $m-1$ modulo $m$

We now show that the pullback $\DD = \BB\times_{\AA}\CC$ contains an embedded infinite path starting at $x_0 = 0$. Since each of its vertices sits on a circuit of length at most $4$, this implies that $\pi_1(\gr{D}, x_0)$ is an infinite rank free group. Since $\pi_1(\DD, x_0)$ surjects $\pi_1(\gr{D}, x_0)$ (this is well-known, see \cite[p. 42]{ser80} for instance), it follows that $\pi_1(\DD, x_0) \cong B\cap C$ is not finitely generated.

First consider the path $(x_k)_k$, starting at $x_0 = 0$ and using only $g_+$-edges. In particular, for each $k\ge 0$, we have $x_{k+1} = n\,q(x_k) + r_k + 1$.  We show that, for most  values of $m$ and $n$, this path is infinite.

\underline{Case $n \ge m$}. \enspace If $x_k \ge 0$, then $q(x_k)\ge 0$, so $x_{k+1} > x_k$. In particular, since $x_0 \ge 0$, the sequence $(x_k)_k$ takes pairwise distinct values.

\underline{Case $n < -m$}.\enspace If $x_k < 0$, then $q(x_k) \le -1$ and $x_{k+1} = n\,q(x_k) + r(x_k) + 1 > 0$. Then $|x_{k+1}|-|x_k| = x_{k+1} + x_k = (n+m)\,q(x_k) + 2\,r(x_k) + 1 \ge 2$, and hence $|x_{k+1}| > |x_k|$.

Suppose now that $0 \le x_k < m-1$. Then $q_k = 0$, $x_{k+1} = x_k +1$ and $x_k+2 = x_k+2$. If $x_k = m-1$, then $x_{k+1} = m$, $x_{k+2} = n+1 < 0$, and hence $|x_{k+2}|-|x_k| = -(n+1) - (m - 1) = -(n+m) \ge  1$. 

Finally, suppose that $x_k \ge m$. Then $q(x_k) \ge 1$ and $x_{k+1} = n\,q(x_k) + r(x_k) + 1 = m\,q(x_{k+1}) + r(x_{k+1})$. Since $n < -m$, we have $m\,q(x_{k+1}) + r(x_{k+1}) < -m\,q(x_k) + r(x_k) + 1$, and hence $m\, (q(x_{k+1}) + q(x_k)) < r(x_k) + 1 - r(x_{k+1}) \le m$. It follows from the inequality $m\, (q(x_{k+1}) + q(x_k)) < m$ that $q(x_{k+1}) + q(x_k) \le 0$, that is, $q(x_{k+1}) \le -q(x_k)$. 

In particular, $m\, (q(x_{k+1}) + q(x_k)) < m$ and it follows that $q(x_{k+1}) + q(x_k) \le 0$, that is, $q(x_{k+1}) \le -q(x_k)$. This first yields
$$x_{k+2} = n\, q(x_{k+1}) + r(x_{k+1}) + 1 \ge -n\, q(x_k)  + r(x_{k+1}) + 1 > m\, q(x_k)  + r(x_{k+1}) + 1 \ge 3,$$
so that $x_{k+2} > 0$, and $|x_{k+2}| - |x_k| = x_{k+2} - x_k$. Next we have
\begin{align*}
    x_{k+2} - x_k &= (n-m)\, q(x_{k+1}) + (m\,q(x_{k+1}) + r(x_{k+1}) + 1) - x_k \\
    &= (n-m)\, q(x_{k+1}) + x_{k+1} + 1  - x_k \\
    &= (n-m)\, q(x_{k+1}) + n\, q(x_k) + r(x_k) + 2 - (m\, q(x_k) + r(x_k)) \\
    &= (n-m)\, (q(x_{k+1})+q(x_k)) + 2 \ge 2.
\end{align*}
As a result, the sequence $(|x_k|)_k$ has a strictly increasing subsequence, and hence $(x_k)_k$ takes infinitely many values.

\underline{Case $n = -m$}.\enspace 
Let $q\in \Z$. Following $m-1$ successive $g_+$-edges from $mq$ leads to $(-1)^{m-1}mq + m-1$, and an additional $g_+$-edge leads to $m((-1)^m q + 1)$. If $m$ is even, this is $m\,(q+1)$. In particular, the path out of $x_0 = 0$ following only $g_+$-edges is infinite. If instead $m$ is odd, the $g_-$-edge from $(-1)^{m-1}mq + m-1 = m\,(q+1)-1$ leads to $-m\,(q+1)$. Repeating this sequence of $m-1$ successive $g_+$-edges followed by a $g_-$-edge leads from $-m\,(q+1)$ to $m(q+2)$. The repetition of this length $2m$ edges leads from $x_0 = 0$ to $2m$, $4m$, $6m$, etc. In particular, it yields a path which visits infinitely many vertices.

Thus we have constructed, in all cases, a sequence of vertices with infinitely many values, as announced.
\end{proof}

\section{The subcategory of core graphs of groups with immersions}\label{sec: subcategory of core gog}

As discussed in the introduction, a major motivation for considering morphisms of graphs of groups and their pullbacks is to explore the intersection of subgroups of the fundamental group of a graph of groups. In this context, it is natural to restrict our attention to immersions of graphs of groups (see Corollary~\ref{cor: folded morphisms}). There is one further restriction that we shall impose, \emph{coreness}, inspired by Stallings' treatment of the case of graphs in \cite{sta83}. Indeed, the overarching theme of this section is to demonstrate that, in the subcategory of core graphs of groups with immersions, there are completely analogous statements to those established by Stallings \cite{sta83} for immersions of core graphs.

The structure of this section is as follows. We shall begin by defining the core of a graph of groups and establishing a criterion for the existence of the pullback in the category of core graphs of groups with immersions. Then, we shall convert the theory of coverings of graphs of groups, as developed by Bass \cite{bas93}, into our language. This allows us to define canonical immersions of core graphs of groups which are in correspondence with (conjugacy classes of) subgroups. We then use this to explicitly interpret the components of a given $\AA$-product $\BB\wtimes_{\AA}\CC$ (over $\mu^B$ and $\mu^C$) in terms of $(B,C)$-double cosets in $A = \pi_1(\AA, u)$, where $B = \mu^B_*(\pi_1(\BB, v))$ and $C = \mu^C_*(\pi_1(\CC, w))$: each component of $\core(\BB\wtimes_{\AA}\CC)$ corresponds to a double coset $B\,g\,C$ such that $B^g\cap C$ contains at least one element that is not conjugate into a vertex group. Finally, we close the paper with further improvements for core acylindrical graphs of groups with immersions: in this case, pullbacks always exist.

\subsection{Core graphs of groups and the pullback}\label{sec: core gog and pullbacks}

The notion of a core graph extends naturally to graphs of groups as follows.

\begin{defn}
    Let $\AA$ be a graph of groups, let $\gr A$ be its underlying graph and let $u \in V(\gr A)$.
    \begin{itemize}
        \item An $\AA$-circuit $p = (a_0,\, e_1,\, a_1\, \dots\, e_k,\, a_k)$ is \emph{cyclically reduced} if it is reduced and, when $e_k = e_1\inv$, we additionally have $a_ka_0\notin \alpha_{e_1}(A_{e_1})$. Every $\AA$-circuit $p$ has a \emph{cyclic reduction}; that is, a cyclically reduced $\AA$-circuit $q$ so that $p \sim_{\AA} r^{-1}\,q\,r$ for some $\AA$-path connecting the origin of $q$ with that of $p$ (as in Proposition~\ref{prop: sim vs equiv}~\eqref{eq: reduce path}).
        
        \item The \emph{core of $\AA$ at $u$}, written $\core(\AA,u)$, is the graph of groups consisting of all the vertices and edges of $\gr{A}$ (and the corresponding vertex and edge groups and structure maps) that occur in reduced $\AA$-circuits at $u$.

        \item The (unpointed) \emph{core of $\AA$}, written $\core(\AA)$, is the graph of groups consisting of all the vertices and edges of $\gr{A}$ (and the corresponding vertex and edge groups and structure maps) that occur in cyclically reduced $\AA$-circuits  of positive length (at any vertex).
        
        \item We say that $\AA$ is \emph{core} if it is equal to $\core(\AA)$, and that it is \emph{core at $u$} if it is equal to $\core(\AA,u)$. Finally, we say that the pointed graph $(\AA, u)$ is \emph{core} if $\AA$ is core at $u$.
    \end{itemize}
\end{defn}

Note that when $\AA$ is a graph of trivial groups, the underlying graph of $\core(\AA)$ coincides with Stallings' core of the underlying graph $\gr{A}$.

The form of the core of a graph of groups is not too dissimilar to the case of graphs, as the following lemma demonstrates.

\begin{lem}
\label{lem: form_of_core}
    Let $\AA'$ be a graph of groups and suppose that $\AA = \core(\AA')\neq\emptyset$. Then $\AA'$ contains subtrees of groups $\mathbb{T}_v$ for each vertex $v\in V(\gr{A})$, such that $\gr{T}_v\cap \gr{A} = v$, $\gr{T}_v\cap \gr{T}_w = \emptyset$ for all $v\neq w\in V(\gr{A})$ and 
    \[
    \gr{A}' = \gr{A} \cup \bigcup_{v\in V(\gr{A})}\gr{T}_v.
    \]
Furthermore, for each edge $e\in E(\gr{T}_v)$ oriented towards $v$, the edge map $\alpha_e\colon A_e\to  A_{o(e)}$ is an isomorphism.
\end{lem}

\begin{proof}
    The fact that $\gr{A}'$ has the claimed form follows from the fact that any cyclically reduced circuit in $\gr{A}'$ gives rise to a cyclically reduced $\AA$-circuit of positive length.
            
    Let $v\in V(\gr{A})$, $q$ a cyclically reduced $\AA$-circuit based at $v$, $e_1$ be an edge in $\gr T_v$, oriented towards $v$ and let $a_1$ be an arbitrary element of $A_{o(e_1)}$. Since $\gr T_v$ is a tree, there exist $e_2,\dots,e_n \in E(\gr T_v)$ such that $(e_1, \ldots, e_n)$ is a geodesic path ending at $v$. Consider the $\AA'$-circuit \[p = (a_1, e_1, 1, \ldots, e_n, 1)\, q\, (1, e_n^{-1}, \ldots, e_1^{-1}, 1).\] Then $p$ is cyclically reduced if and only if $a_1 \notin  \alpha_{e_1}(A_{e_1})$. But $e_1$ is not in $\AA = \core(\AA')$, so we must have $a_1 \in \alpha_{e_1}(A_{e_1})$. Since this holds for every $a_1\in A_{o(e_1)}$, $\alpha_{e_1}$ is surjective and hence an isomorphism.
\end{proof}

We note the following useful special case.

\begin{rem}\label{rk: pointed core vs core}
    Let $(\AA',u_0)$ be a pointed core graph of groups and let $\AA = \core(\AA')$. If $\AA$ is empty, then $\gr A'$ is a finite line.

    If $\AA$ is not empty, Lemma~\ref{lem: form_of_core} shows that $\AA'$ has underlying graph a finite line $\gr L$ connecting the basepoint with $\AA$, such that for each edge $e\in E(\gr{L})$ pointing away from the basepoint (and hence towards $\AA$), the map $\alpha_e$ is an isomorphism.
\end{rem}

Theorem~\ref{thm: pullbacks in GrGp*} and Corollary~\ref{cor: it is a pullback} established the existence of pullbacks in $\GrGp^*$ and in its subcategory, restricted to immersions. This easily extends to the subcategory restricted to core pointed graphs of groups and immersions. 

\begin{cor}\label{cor: core pullback}
Let $\mu^B\colon(\BB, v_0)\to(\AA, u_0)$ and $\mu^C\colon (\CC, w_0)\to(\AA, u_0)$ be immersions between core pointed graphs of groups. Let $(\DD,x_0)$ be their pointed $\AA$-product, with projection morphisms $\rho^B$ and $\rho^C$.

Then $(\core(\DD, x_0), \rho^B, \rho^C)$ is the pullback of $\mu^B$ and $\mu^C$ in the category of core pointed graphs of groups with immersions (up to $\sim$-equivalence).
\end{cor}

\begin{proof}
Let $\sigma^B\colon (\DD',y_0) \to (\BB,v_0)$ and $\sigma^C\colon (\DD',y_0) \to (\CC,w_0)$ be immersions such that $(\DD',y_0)$ is core and $\mu^B \circ \sigma^B \sim \mu^C \circ \sigma^C$. Since $((\DD, x_0), \rho^B, \rho^C)$ is the pullback of $\mu^B$ and $\mu^C$ in the category of pointed graphs of groups with immersions (Corollary~\ref{cor: it is a pullback}), there exists an immersion $\sigma\colon (\DD',y_0) \to (\DD,x_0)$ such that $\rho^B\circ\sigma \sim \sigma^B$ and $\rho^C\circ\sigma \sim \sigma^C$.

Since $\sigma$ is an immersion, it maps reduced paths to reduced paths (Proposition~\ref{folded}), and hence its range $\sigma(\DD')$ is core at $\sigma(y_0) = x_0$. Thus $\sigma$ is in fact an immersion from $(\DD',y_0)$ to $(\core(\DD,x_0), x_0)$, which concludes the proof.
\end{proof}

We saw that pullbacks do not always exist in the category $\GrGp$ of unpointed graphs of groups (Example~\ref{ex: pullbacks may not exist}). This is also the case if we restrict $\GrGp$ to core graphs of groups and immersions between them. However, with a technical condition on a given immersion $\mu^C\colon \CC\to \AA$, we can at least guarantee that pullbacks with $\mu^C$ always exist.

\begin{cor}
\label{thm: pullbacks_exist_sometimes1}
    Let $\mu^C\colon \CC\to \AA$ be an immersion of core graphs of groups such that for each cyclically reduced $\CC$-circuit $c$, there is no non-trivial element $a\in A_{\mu^C(o(c))}$ such that ${[a, \mu^C(c)] = 1}$. Then the following holds:

    If $\mu^B\colon \BB\to \AA$ is an immersion of core graphs of groups, then the pullback of $\mu^B$ and $\mu^C$ exists in the category of core graphs of groups with immersions and is precisely $(\core(\BB\wtimes_{\AA}\CC), \rho^B, \rho^C)$.
\end{cor}

\begin{proof}
    Let $\sigma^B \colon\DD'\to \BB$ and $\sigma^C\colon\DD'\to \CC$ be immersions of core graphs of groups such that $\mu^B\circ\sigma^B\approx\mu^C\circ\sigma^C$. Let $d$ be a cyclically reduced $\DD'$ circuit at $y$ and let $\bfa \in C(\mu^C\circ\sigma^C)$. By definition, we have $(\mu^C\circ\sigma^C(d))\,a_y = a_y\, (\mu^C\circ\sigma^C(d))$, that is, $[a_y, \mu^C(\sigma^C(d))] = 1$. Since $\sigma^C$ is an immersion, $\sigma^C(d)$ is cyclically reduced. Hence, our assumption on $\mu^C$ implies that $a_y = 1$, and hence $\bfa = ((1),(1))$ (Lemma~\ref{lem: trivial element of centraliser}). Now Theorem \ref{thm: lift characterisation} implies that there is a unique lift $\sigma\colon \DD'\to \core(\BB\wtimes_{\AA}\CC)$ (up to $\approx$-equivalence) such that $\rho^B\circ\sigma\approx\sigma^B$ and $\rho^C\circ\sigma\approx\sigma^C$. It follows that $(\core(\BB\wtimes_{\AA}\CC), \rho^B, \rho^C)$ is the pullback of $\mu^B$ and $\mu^C$, as announced.
\end{proof}

\subsection{Coverings of graphs of groups}
\label{sec: covers}

Let $(\AA,u_0)$ be a pointed graph of groups and let $A = \pi_1(\AA,u_0)$. We saw before that if $\mu\colon (\BB,v_0) \to (\AA, u_0)$ is a morphism of pointed graphs of groups, then $\mu_*\colon \pi_1(\BB,v_0) \to \pi_1(\AA,u_0)$ is a group morphism (see Corollary~\ref{cor: mu* is a group morphism})\footnote{In category-theoretic terms, $\pi_1$ can be viewed as a functor, and $\mu_*$ is the image of $\mu$ under $\pi_1$.}, so that $\mu$ can be thought of as \emph{representing} the subgroup $B = \mu_*(\pi_1(\BB,v_0))$ of $A$. We also saw that $\sim$-equivalent morphisms represent the same subgroup (Proposition~\ref{prop: change of data}). Finally, we saw that if $\mu$ is an immersion, then $\mu_*$ is injective (Corollary~\ref{cor: folded morphisms}), that is, $\pi_1(\BB,v_0)$ is isomorphic to $B$. It turns out that every subgroup of $A$ can be represented in this way, and one can identify a canonical morphism representing it using the notion of coverings, see Theorem~\ref{thm: uniqueness of covers} below. 

\begin{defn}\label{def: coverings}
    Let $\mu\colon \BB\to \AA$ be a morphism of graphs of groups. We say that $\mu$ is a \emph{covering} of $\AA$ if it is an immersion and if, for each vertex $v\in V(\gr B)$, the following map 
\begin{align*}
    \Star(v) &\longrightarrow \bigsqcup_{e\in \Star(\mu(v))} \dblcoset{\mu(B_v)}{A_{\mu(v)}}{\alpha_e(A_e)} \\
    f &\longmapsto \left(\mu(B_v)\right)\, f_{\alpha}\, \left(\alpha_{\mu(f)}(A_{\mu(f)})\right).
\end{align*}
is bijective.

A morphism $\mu\colon (\BB, v_0) \to (\AA, u_0)$ of pointed graphs of groups is a \emph{pointed covering} if the underlying morphism $\mu\colon \BB \to \AA$ is a covering.
\end{defn}

\begin{rem}\label{rk: equivalent to Bass}
    Bass \cite[Section 2]{bas93} defines a covering of graphs of groups to be a morphism $\mu\colon \BB\to \AA$ such that for each vertex $v\in V(\gr{B})$, the following map,
    \begin{align*}
        \bigsqcup_{f\in \Star(v)}B_v/\alpha_f(B_f) &\to \bigsqcup_{e\in\Star(\mu(v))}A_{\mu(v)}/\alpha_{e}(A_{e})\\
        b\, \alpha_f(B_f) &\mapsto \mu_v(b)\, f_{\alpha}\, \alpha_{\mu(f)}(A_{\mu(f)}),
    \end{align*}
    is bijective. This can be seen to be equivalent to Definition~\ref{def: coverings}, see Remark~\ref{rem: folded}.
\end{rem}

We note the following sharpening of Lemma~\ref{lem: rhos are folded 2}, in the case of coverings.

\begin{lem}\label{lem: coverings yield coverings}
    If $\mu^B\colon \BB\to \AA$ is a covering and $\mu^C\colon \CC\to \AA$ is a morphism, then $\rho^C\colon \BB\wtimes_{\AA}\CC\to \CC$ is a covering. Likewise, if $\mu^C$ is a covering, then $\rho^B$ is one too.
\end{lem}

\begin{proof}
    Let $\DD = \BB\wtimes_\AA\CC$. By Lemma~\ref{lem: rhos are folded 2}, $\rho^C$ is an immersion. Thus, we just need to check that the map given in Definition~\ref{def: coverings} is a bijection. Explicitly, we want to show that for each vertex $x$ of  $\gr{D}$, say $x\in V_{v,w}(\gr{D})$, each edge $g\in \Star(w)$ and each $c\in C_w$, there is an edge $h\in \Star(x)$ such that $\rho^C(h) = g$ and
    \[
    \rho^C(D_x)\, h_{\alpha}^C\, \alpha_g(C_g) = \rho^C(D_x)\, c\, \alpha_g(C_g).
    \]
    Let $e = \mu^C(g)$. Since $\mu^B$ is a covering, there exists $f\in \Star(v)$ such that $\mu^B(f) = e$ and
    \[
    \mu^B(B_v)\,f_{\alpha}\,\alpha_e(A_e) = \mu^B(B_v)\,(d_x\, \mu^C(c)\, g_\alpha)\, \alpha_e(A_e).
    \]
    This implies the existence of $b\in B_v$ and $a\in A_e$ such that
    \begin{equation}\label{eq: f from covering}
        f_\alpha = \mu^B(b\inv)\,d_x\, \mu^C(c)\, g_\alpha\, \alpha_e(a\inv).
    \end{equation}
    Now let $h = \mu^B(B_f)\,a\,\mu^C(C_g) \in E_{f,g}(\gr D)$. We choose $d_h = a$. By Definition~\ref{def: underlying graph A-product}, the origin of $h$ is the double coset
    $$\mu^B(B_v)\, (f_\alpha\,\alpha_e(a)\, g_\alpha\inv)\, \mu^C(C_w) = \mu^B(B_v)\, (\mu^B(b)\inv\,d_x\, \mu^C(c))\, \mu^C(C_w) = \mu^B(B_v)\, d_x\, \mu^C(C_w) = x,$$
    that is, $h \in \Star(x)$.
    Finally, Equations~\eqref{eq: f from covering} and~\eqref{eq: old x tilde}, together with Definition~\ref{defn: projection morphisms}, show that we can choose $h_\alpha^B = b$ and $h_\alpha^C = c$. It follows that $\rho^C(D_x)\, h_{\alpha}^C\, \alpha_g(C_g) = \rho^C(D_x)\, c\, \alpha_g(C_g)$, as announced.
\end{proof}

Just like in the usual covering space theory, there are natural lifting results for coverings of graphs of groups. We begin with lifting paths.

\begin{lem}
\label{lem: pointed segment lift}
    Let $\mu\colon (\BB, v_0)\to (\AA, u_0)$ be a pointed covering of graphs of groups and let $p$ be an $\AA$-path from $u_0$ to a vertex $u$. Then there is a $\BB$-path $q$ leading out from $v_0$ and an element $a\in A_u$ so that $\mu(q)\,a\sim_{\AA} p$.

    Moreover, if $p$ is an $\AA$-circuit at $u_0$ and $p\in \mu_*(\pi_1(\BB, v_0))$, then $q$ is a loop and we may take $a = 1$.
\end{lem}

\begin{proof}
    Let $p = (a_0, e_1, \ldots, e_k, a_k)$. The proof proceeds by induction on the length of $p$. The result is clearly true for $k = 0$. Now suppose that $k\geqslant 1$ and that the result is true for all $\AA$-paths of length less than $k$.
    
    Let $p' = (a_0, e_1, \ldots, e_{k-1}, a_{k-1})$ and $u' = t(p')$. Let $q'$ be a $\BB$-path leading out from $v_0$ and $a'\in A_{u'}$ be given by the induction hypothesis, such that $\mu(q')\,a'\sim_{\AA} p'$. Let $v' = t(q')$.
    Since $\mu$ is a covering, there is an edge $f\in \Star(v')$ such that $\mu(f) = e_k$ and $\mu(B_{v'})\,f_{\alpha}\,\alpha_{e_k}(A_{e_k}) = \mu^B(B_{v'})\,a'\,\alpha_{e_k}(A_{e_k})$. In particular, there exist $b\in B_{v'}$ and $a''\in A_{e_k}$ such that $f_\alpha = \mu(b)\,a'\,\alpha_{e_k}(a'')$.
    Now let $q = q'\, (b\inv, f, 1)$. Then
    \begin{align*}
        \mu(q) = \mu(q')\, (\mu(b)\inv\,f_\alpha, e_k, 1) &\sim_\AA p'\, a'^{-1}\, (\mu(b)\inv\,f_\alpha, e_k, 1) \\ 
        &\sim_\AA p'\,(\alpha_{e_k}(a''), e_k, 1)
        \\
        &\sim_\AA p'\,(1, e_k, \omega_{e_k}(a'')) = p\, (a_k\inv\,\omega_{e_k}(a'')).
    \end{align*}
    This proves that $p \sim_\AA \mu(q)\,a$, where $a = \omega_{e_k}(a'')\inv\,a_k$.
    
    Now suppose that $p$ is an $\AA$-circuit at $u_0$ and $p\in \mu_*(\pi_1(\BB, v_0))$, say $p = \mu_*(q')$ for some $\BB$-circuit $q'$ at $v_0$. Then, as above, there exists a $\BB$-path $q$ and an element $a\in A_{u_0}$ such that $\mu(q)\,a\sim_{\AA}p =_\AA \mu(q')$. It follows that $a =_\AA \mu(q\inv\,q')$. Let $r$ be a reduced $\BB$-path such that $q\inv\,q' \sim_\BB r$. Then $\mu(r)$ is reduced as well, and has length 0, that is, $r\in B_{v_0}$, and we have $a' = \mu(r)$. In particular, $q\inv$ starts at $v_0$, that is, $q$ is a $\BB$-circuit at $v_0$. Now let $\hat q = q\,r$: this $\BB$-circuit at $v_0$ satisfies $\mu(\hat q) = \mu(q)\,\mu(r) \sim_\AA \mu(q)\,a' \sim_\AA p$, which completes the proof.
\end{proof}

With a slight modification to the proof of Lemma \ref{lem: pointed segment lift}, we may also obtain an unpointed statement.

\begin{lem}
\label{lem: segment_lift}
    Let $\mu\colon \BB\to \AA$ be a covering of graphs of groups and let $p$ be an $\AA$-path from vertex $u$ to vertex $u'$. Let $v \in V(\gr B)$ such that $\mu(v) = u$.  Then there exists a $\BB$-path $q$ from $v$ to a vertex $v'$ and elements $a\in A_u$ and $a'\in A_{u'}$ such that $a\,\mu^B(q)\,a'\sim_{\AA}p$.

    Moreover, if $p$ is an $\AA$-circuit in $\mu_*(\pi_1(\BB, o(q)))$, then $q$ is a $\BB$-circuit and we may take $a' = a^{-1}$.
\end{lem}

We may now lift arbitrary morphisms, using Lemma \ref{lem: pointed segment lift} and our results on the pullback.

\begin{thm}
\label{thm: lifting_pointed_morphisms}
    Let $\mu^B\colon (\BB, v_0)\to (\AA, u_0)$ be a pointed covering and let $\mu^C\colon (\CC, w_0)\to (\AA, u_0)$ be a morphism of pointed graphs of groups, where $\gr A$, $\gr B$ and $\gr C$ are connected. There exists a pointed lift $\sigma\colon (\CC, w_0)\to (\BB, v_0)$ such that $\mu^B\circ\sigma \sim \mu^C$, if and only if $\mu^C_*(\pi_1(\CC, w_0))\leqslant \mu^B_*(\pi_1(\BB, v_0))$. Moreover, this lift is unique up to $\sim$-equivalence.
\begin{figure}[H]
    \centering
    \begin{tikzcd}[column sep=4em,row sep=3em]
          & (\BB,v_0) \arrow[d, "\mu^B"] \\
         (\CC,w_0) \arrow[ur,dashed, "\sigma"]\arrow[r, "\mu^C"'] & (\AA,u_0)
    \end{tikzcd}
    
\end{figure}
    In this situation, $(\CC, w_0)$ (with projection morphisms $1_\CC$ and $\sigma$) is the pullback of $(\BB, v_0)$ and $(\CC, w_0)$.
\end{thm}

\begin{proof}
If such a lift exists, then certainly $\mu^C_*(\pi_1(\CC, w_0))\leqslant \mu^B_*(\pi_1(\BB, v_0))$. Now suppose that $\mu^C_*(\pi_1(\CC, w_0))$ is a subgroup of $ \mu^B_*(\pi_1(\BB, v_0))$. Consider the pointed pullback $(\DD, x_0)$ of the morphisms $\mu^B$ and $\mu^C$, with projection morphisms $\rho^B$ and $\rho^C$. Since $\mu^C_*(\pi_1(\CC, w_0))\leqslant \mu_*^B(\pi_1(\BB, v_0))$, Theorem \ref{cor: intersection of subgroups} shows that $(\mu^C\circ\rho^C)_*(\pi_1(\DD, x_0)) = \mu^C_*(\pi_1(\CC, w_0))$.

By Lemma \ref{lem: pointed segment lift}, for each $\CC$-circuit $c$ at $w_0$, there is a $\BB$-circuit $b$ at $v_0$ so that $\mu^B(b) \sim_{\AA} \mu^C(c)$. By Corollary \ref{cor: lifting}, this implies that $\rho^C_*$ is surjective. By Lemma~\ref{lem: coverings yield coverings}, $\rho^C$ is a covering. In particular, it is an immersion and so $\rho^C_*$ is an isomorphism from $\pi_1(\DD, x_0)$ to $\pi_1(\CC, w_0)$. By Lemma~\ref{lem:immersion_inclusion} it follows that $\rho^C$ is an inclusion of the underlying graphs, and an isomorphism on each vertex and edge group. Combined with the fact that $\rho^C$ is a covering, this establishes that $\rho^C$ is an isomorphism of graphs of groups. Let $\sigma\colon (\CC,w_0) \to (\BB,v_0)$ be given by $\sigma = \rho^B \circ {\rho^C}\inv$. Then, as announced, $\sigma$ is a lift of $\mu^C$, and $(\CC,w_0)$, with morphisms $1_\CC$ and $\sigma$, is the pullback of $\mu^B$ and $\mu^C$ (up to isomorphism in $\GrGp^*$). The uniqueness of pullbacks then yields the lift uniqueness statement: if $\sigma'\colon (\CC,w_0) \to (\BB,v_0)$ is a lift of $\mu^C$, then $\sigma' \sim \sigma$.
\end{proof}

An unpointed lifting theorem can be proved by combining Theorem \ref{thm: lifting_pointed_morphisms} with Lemma \ref{lem: pointed segment lift}.

\begin{thm}
\label{thm: lifting_unpointed_morphisms}
    Let $\mu^B\colon \BB\to \AA$ be a connected covering and let $\mu^C\colon \CC\to \AA$ be a morphism of connected graphs of groups. Let $w_0 \in V(\gr C)$ and $u_0 = \mu^C(w_0)$. The following are equivalent:
    \begin{enumerate}[(1)]
        \item There exists a lift $\mu\colon \CC\to \BB$ such that $\mu^B\circ\mu \approx \mu^C$. \label{item: exists unpointed lift}
        \item There exists $v \in (\mu^B)\inv(u_0)$ and $a \in \pi_1(\AA,u_0)$ such that $\mu^C_*(\pi_1(\CC, w_0))^a \le \mu^B_*(\pi_1(\BB, v))$. \label{item: cjug in the big group}
        \item There exists $v \in (\mu^B)\inv(u_0)$ and $a \in A_{u_0}$ such that $\mu^C_*(\pi_1(\CC, w_0))^a \le \mu^B_*(\pi_1(\BB, v))$. \label{item: cjug in the small group}
    \end{enumerate}
    Moreover, for each vertex $v\in (\mu^B)\inv(u_0)$, the lifts $\mu\colon \CC\to \BB$ with $\mu(w_0) = v$ are in bijection with the double cosets $C_{w_0}(\mu^C)\, a\, \mu^B(B_{v})\in C_{w_0}(\mu^C)\backslash A_{u_0}/ \mu^B(B_{v})$ so that $\mu^C_*(\pi_1(\CC, w_0))^a \le \mu^B_*(\pi_1(\BB, v))$.
\end{thm}

\begin{proof}
    Item~\eqref{item: cjug in the small group} trivially implies Item~\eqref{item: cjug in the big group}. Let us verify that the converse holds. Let $C$ be a subgroup of $\pi_1(\AA,u_0)$, and suppose that $v_0 \in (\mu^B)\inv(u_0)$ and $a \in \pi_1(\AA,u_0)$ satisfy $C^a \le \mu^B_*(\pi_1(\BB, v))$. Let $p$ be an $\AA$-circuit at $u_0$ representing $a\inv$. By Lemma~\ref{lem: pointed segment lift}, there exists a reduced $\BB$-path $q$ starting at $v_0$ and ending at some vertex $v_1$, and there exists an element $a'\in A_{u_0}$ such that $\mu^B(q)\,a' \sim_\AA p$. In particular, $\mu^B(v_1) = u_0$. Now we have
    \[
    C^a = C^{p\inv} = (C^{a'^{-1}})^{\mu^B(q\inv)} \le \mu^B_*(\pi_1(\BB, v_0)),
    \]
    and hence $C^{a'^{-1}} \le (\mu^B_*(\pi_1(\BB, v)))^{\mu^B(q)} = \mu^B(\pi_1(\BB, v_0)^{q}) = \mu^B(\pi_1(\BB, v_1))$. This establishes the equivalence of Items~\eqref{item: cjug in the big group} and~\eqref{item: cjug in the small group}.
    
    Let us now assume that $\mu\colon \CC\to\BB$ is a lift, that is, $\mu^B\circ\mu\approx\mu^C$. Let $v_0 = \mu(w_0)$. Then $\mu^B(v_0) = u_0$ and $\mu^C_*(\pi_1(\CC, w_0))$ conjugates into $\mu^B_*(\pi_1(\BB, v_0))$ via an element of $A_{u_0}$ by Proposition~\ref{prop: change of data}. Thus~\eqref{item: exists unpointed lift} implies~\eqref{item: cjug in the small group}.

    Conversely, suppose that $v_0\in V(\gr B)$ satisfies $\mu^B(v_0) = u_0$ and $\mu^C_*(\pi_1(\CC, w_0))^a \le \mu^B_*(\pi_1(\BB, v_0))$, for some $a\in A_{u_0}$. Let $\nu^C\colon \CC\to \AA$ be the morphism which differs from  $\mu^C$ only on the $\alpha$-twisting elements of edges starting at $w_0$ (and $\omega$-twisting elements of edges into $w_0$): if $g$ is an edge starting at $w_0$ with $\alpha$-twisting element $g_\alpha$ relative to $\mu^C$, its $\alpha$-twisting element relative to $\nu^C$ is $a\inv\,g_\alpha$. It is directly verified that $\nu^C \approx \mu^C$ and $\nu^C(\CC,w_0) = \mu^C(\CC,w_0)^a \le \mu^B_*(\pi_1(\BB, v_0))$. 

    Now fix a vertex $v\in (\mu^B)^{-1}(u_0)$. Let $\mu, \nu\colon (\CC, w_0)\to (\BB, v)$ be two pointed morphisms so that $\mu^C\approx\mu^B\circ\mu$ and $\mu^C\approx\mu^B\circ\nu$. Let $\bft^1 \in P(\mu^C, \mu^B\circ\mu)$ and $\bft^2 \in P(\mu^C, \mu^B\circ\nu)$. Then $\bft = (\bft^1)^{-1}\, \bft^2 \in P(\mu^B\circ\mu, \mu^B\circ\nu)$. By Lemma \ref{lem: twist lemma}, $\mu\approx\nu$ if and only if $\mu$ and $\nu$ agree on the underlying graph and $\bft = \bfz\,\mu^B(\bfb)$ for some $\bfz\in C(\mu^B\circ\mu)$ and $\bfb\in \prod_wB_{\mu(w)}\times\prod_gB_{\mu(g)}$. Since $\mu^B$ is an immersion and $\mu, \nu$ agree on $w_0$, Lemma \ref{lem: graph_determined} establishes that $\mu\approx\nu$ if and only if $t_{w_0} = z_{w_0}\,\mu^B(b_{w_0})$ for some $\bfz \in C(\mu^B\circ\mu)$ and some $b_{w_0}\in B_{v}$. But since $C(\mu^B\circ\mu) = C(\mu^C)^{\bft^1}$, we see that $\mu\approx\nu$ if and only if $t_{w_0}^2 = z_{w_0}'\,t_{w_0}^1\,\mu^B(b_{w_0})$ for some $b\in B_v$ and where $\bfz' = \bfz^{(\bft^1)^{-1}}\in C(\mu^C)$.
\end{proof}

Continuing the analogy with covering space theory for topological spaces, every subgroup and every conjugacy class of subgroups of $\pi_1(\AA, u_0)$ is captured by a covering. These can be defined and constructed intrinsically or as quotients of the universal cover, the Bass--Serre tree, by the action of a subgroup (see Appendix \ref{appendix}). A direct and explicit construction \cite[Section 4]{bas93} is, by now, standard. For the sake of completeness, we briefly describe it in our language.

Let $(\AA, u_0)$ be a pointed graph of groups and $A = \pi_1(\AA, u_0)$. We denote by $\pi^\AA[u_0, u]$ the $=_\AA$-equivalence classes of $\AA$-paths connecting $u_0$ with $u$, so that $A = \pi^\AA[u_0,u_0]$. If $B\leqslant A$ is any subgroup, consider the graph $\gr{B}$ with
\begin{align*}
    V(\gr{B}) &= \bigsqcup_{u\in V(\gr{A})} \dblcoset B{\pi^{\AA}[u_0, u]}{A_u}\\
    E(\gr{B}) &= \bigsqcup_{e\in E(\gr{A})} \dblcoset B{\pi^{\AA}[u_0, o(e)]}{\alpha_e(A_e)},
\end{align*}
such that, if $p \in \pi^{\AA}[u_0, o(e)]$, then
\begin{align*}
    (B\, p\, \alpha_e(A_e))^{-1} & = B\, (p\, e)\, \alpha_{e\inv}(A_{e\inv}) = B\, (p\, e)\, \omega_e(A_e) \\
    o(B\, p\, \alpha_e(A_e)) &= B\, p\, A_{o(e)}\\
    t(B\, p\, \alpha_e(A_e)) &= B\, (p\,e)\, A_{t(e)}
\end{align*}
It can be checked that this is well-defined. For each vertex $v\in V(\gr{B})$, say, $v = B\, p\, A_u$ with $u\in V(\gr A)$ and $p \in \pi^\AA[u_0,u]$, we choose an element $p_v \in v$ and set 
\[
B_v = B^{p_v}\cap A_u.
\]
For the edge groups, we first choose an orientation $E^+$ of $\gr A$. If $f\in E^+$, say, $f = B\, p\, \alpha_e(A_e)$ with $e\in E(\gr A)$ and $p\in \pi^\AA[u_0, o(e)]$, we choose a representative $p_f \in B\,p\,\alpha_e(A_e)$, and we let $p_{f\inv} =_{\AA} p_f\,e$. Since $p =_{\AA} b\,p_f\, \alpha_e(a)$ for some $a\in A_e$ and some $b\in B$, we have $p\,e =_{\AA} b\,p_f\, \alpha_e(a)\, e =_{\AA} b\,(p_f\,e)\, \omega_e(a) =_{\AA} b\,(p_f\,e)\,\alpha_{e^{-1}}(a)$, and hence $p_{f\inv} \in f\inv$. Then we set, for each edge $f\in E(\gr B)$,
$$B_f = \alpha_e^{-1}(B^{p_f}\cap \alpha_e(A_e)).$$
We note that $B_{f\inv} = B_f$. Indeed, if $a\in B_f$, then $\alpha_e(a) = p_f\inv\, b\,p_f$ for some $b\in B$, and hence
$$\alpha_{e\inv}(a) = \omega_e(a) =_\AA e\inv\, \alpha_e(a)\, e = e\inv\, p_f\inv\, b\, p_f\, e = p_{f\inv}\inv\, b\, p_{f\inv}.$$

If $v = o(f) \in V(\gr B)$, then $Bp_vA_{o(e)} = Bp_fA_{o(e)}$ and so there are elements $b_f\in B$ and $a_f\in A_{o(e)}$ such that $p_v = b_fp_fa_f$. We let $f_{\alpha} = (f^{-1})_{\omega} = p_v\inv\,b_f\,p_f = a_f^{-1}$. Then we have
\[
B_f = \alpha_e^{-1}(B_v^{f_{\alpha}}\cap \alpha_e(A_e)) = \alpha_e^{-1}(B_v^{p_v^{-1}b_fp_f}\cap \alpha_e(A_e)) = \alpha_e^{-1}(B^{p_f}\cap\alpha_e(A_e))
\]
and we let $\alpha_f\colon B_f\to B_v$ be the restriction of $\gamma_{f_{\alpha}\inv}\circ\alpha_e$ to $B_f$. We also let $\omega_f = \alpha_{f\inv}\colon B_f \to B_{v'}$, where $v' = t(f)$.

With this data, we may form the graph of groups $\BB = (\gr{B}, \{B_v\}, \{B_f\}, \{\alpha_f, \omega_f\})$ and a morphism $\BB\to \AA$, mapping vertex $B\,p\, A_u$ to $u$ and edge $B\, p\, \alpha_e(A_e)$ to $e$,
with the vertex and edge group morphisms given by inclusion and with twisting elements $\{f_{\alpha}, f_{\omega}\}$. Note that by replacing $B$ with any conjugate $B^a$, the above construction yields a graph of groups $\BB'$ and a natural isomorphism $\BB'\to \BB$ of graphs of groups commuting with the corresponding morphisms to $\AA$. Thus, denote by 
\[
\mu^{[B]}\colon \BB\to \AA
\]
the morphism constructed above, where $[B]$ denotes the conjugacy class of $B$ in $A$. The morphism $\mu^{[B]}$ is certainly a covering and it follows from Theorem \ref{thm: lifting_unpointed_morphisms} that it is uniquely defined up to isomorphism in the category $\GrGp$. Indeed, if $\nu\colon \BB'\to \AA$ is another connected covering such that $\nu_*(\pi_1(\BB',v'_0))$ is conjugated to $B$ in $A$, then Theorem~\ref{thm: lifting_unpointed_morphisms} shows that there exist morphisms $\mu\colon \BB' \to \BB$ and $\mu'\colon \BB \to \BB'$ such that $\mu\circ\mu'\approx1_{\BB'}$, $\mu'\circ\mu\approx1_{\BB}$, $\mu^{[B]} \circ \mu \approx \nu$ and $\nu \circ \mu' \approx \mu^{[B]}$. This is exactly the definition of an isomorphism in the category $\GrGp$. We call $\mu^{[B]}\colon \BB \to \AA$ the \emph{covering associated with the conjugacy class $[B]$}.

The graph $\gr{B}$ has a natural basepoint, namely $v_0 = BA_{u_0}$. It can be checked that the subgroup $\mu_*^{[B]}(\pi_1(\BB, v_0))\leqslant \pi_1(\AA, u_0)$ is conjugate to $B$. Choosing $p_{v_0} = 1$ and denoting the resulting pointed morphism by 
\[
\mu^B\colon(\BB, v_0)\to (\AA, u_0),
\]
we see instead that $\mu_*^B(\pi_1(\BB, v_0)) = B$. This morphism is uniquely defined up to $\sim$-equivalence by Theorem \ref{thm: lifting_pointed_morphisms}. We call $\mu^B\colon (\BB, v_0)\to (\AA, u_0)$ the \emph{pointed covering associated with the subgroup $B$}.

When $B = 1$, the covering $\mu^{[1]}\colon \BB\to \AA$ is the \emph{universal covering}. By \cite[Theorem 1.17]{bas93}, the universal covering is a tree, also known as the \emph{Bass--Serre tree for $(\AA, u_0)$}. See Appendix \ref{appendix} for the details on the functor between the category of graphs of groups and the category of trees with group actions.

Pushing further the analogy with the theory of covering spaces, we have the following bijective correspondences.

\begin{thm}\label{thm: uniqueness of covers}
    Let $(\AA,u_0)$ be a pointed graph of groups and $A = \pi_1(\AA,u_0)$.
    
    Mapping a morphism $\mu\colon (\BB, v_0) \to (\AA, u_0)$ to the subgroup $B = \mu_*(\pi_1(\BB,v_0)) \le A$, establishes a bijection between the set of pointed coverings of $(\AA, u_0)$ (taken up to isomorphism in $\GrGp^*$)
    and the set of subgroups of $A$.

    Similarly, mapping a morphism $\mu\colon \BB\to \AA$ to the conjugacy class of $\mu_*(\pi_1(\BB,v_0))$ (where $v_0$ is an arbitrary vertex of $\gr B$ such that $\mu(v_0) = u_0$), establishes a bijection between the set of coverings of $\AA$ (taken up to isomorphism in $\GrGp$)
    and the set of conjugacy classes of subgroups of $A$.
\end{thm}

\begin{proof}
    The construction of the $B$-cover and the $[B]$-cover outlined above shows that the maps in the statement are surjective.
    
    Let $\mu\colon (\BB,v_0) \to (\AA,u_0)$ and $\nu\colon (\CC,w_0) \to (\AA,u_0)$ be pointed coverings such that $\mu_*(\pi_1(\BB,v_0)) = \nu_*(\pi_1(\CC,w_0))$. Theorem~\ref{thm: lifting_pointed_morphisms} shows the existence of morphisms $\mu'\colon (\CC,w_0)\to (\BB,v_0)$ and $\nu'\colon (\BB, v_0) \to (\CC, w_0)$ such that $\nu \sim \mu\circ\mu'$ and $\mu \sim \nu\circ\nu'$. Note that these $\sim$-equivalences imply that $\mu'_*(\pi_1(\CC,w_0)) = \pi_1(\BB,v_0)$ and $\nu'_*(\pi_1(\BB,v_0)) = \pi_1(\CC,w_0)$ (using the fact that $\mu_*$ and $\nu_*$ are injective).

    It follows that $(\nu'\circ\mu')_*(\pi_1(\CC,w_0)) = \pi_1(\CC,w_0)$. The uniqueness statement of Theorem~\ref{thm: lifting_pointed_morphisms} then implies that $\nu'\circ\mu' \sim 1_{\BB}$. Similarly, $\mu'\circ\nu' \sim 1_{\CC}$ and hence $(\BB,v_0)$ and $(\CC,w_0)$ are isomorphic in $\GrGp^*$.

    The proof in the unpointed case follows the same pattern.
\end{proof}

Since core subgraphs of groups are sent to core subgraphs of groups under immersions of pointed graphs of groups, we obtain the following corollaries to Theorems \ref{thm: lifting_pointed_morphisms} and \ref{thm: lifting_unpointed_morphisms}, in the pointed and the unpointed case.

\begin{cor}
\label{cor: pointed_core_lifting_morphism_1}
    Let $\mu^B\colon (\BB, v_0)\to (\AA, u_0)$ be an immersion of pointed graphs of groups with $(\BB, v_0)$ core, and let $B = \mu^B_*(\pi_1(\BB, v_0))$. Then $\mu^B$ is (isomorphic in $\GrGp^*$ to) the restriction of the pointed $B$-covering to its pointed core subgraph of groups.
    
    If $\mu^C\colon (\CC, w_0)\to (\AA, u_0)$ is any immersion of graphs of groups with $(\CC, w_0)$ core, then there exists a pointed lift $\mu\colon (\CC, w_0)\to (\BB, v_0)$ such that $\mu^B\circ\mu \sim \mu^C$, if and only if $\mu^C_*(\pi_1(\CC, w_0))\leqslant \mu^B_*(\pi_1(\BB, v_0))$. Moreover, such a lift, should it exist, is unique up to $\sim$-equivalence.
\end{cor}

\begin{proof}
    Let $(\BB', v_0')\to (\AA, u_0)$ be the pointed $B$-covering. By Theorem~\ref{thm: lifting_pointed_morphisms} there is a unique lift $\sigma^B\colon(\BB, v_0)\to (\BB', v_0')$. Since $\mu^B_*(\pi_1(\BB, v_0)) = B$, the morphism $\sigma^B_*$ is an isomorphism. Then Lemma \ref{lem:immersion_inclusion} implies that $\sigma^B$ is an inclusion of graphs of groups. Finally, the image of $\sigma^B$ must be the core of $(\BB', v_0')$ since $(\BB, v_0)$ is core and $\sigma_*^B$ is an isomorphism.

    For the second statement, Theorem~\ref{thm: lifting_pointed_morphisms} implies that there is a unique lift $\sigma^C\colon (\CC, w_0)\to (\BB', v_0')$ to the pointed $\BB$-covering. Since $\CC$ is core, $\sigma^C$ has image the core of $(\BB', v_0')$ which is precisely (the image of) $(\BB, v_0)$.
\end{proof}

An unpointed version of Corollary \ref{cor: pointed_core_lifting_morphism_1} may be proved analogously.

\begin{cor}
\label{cor: core_lifting_morphism_1}
    Let $\mu^B\colon \BB\to \AA$ be an immersion of graphs of groups with $\BB$ core and connected, let $v_0 \in V(\gr B)$ and let $B = \mu^B_*(\pi_1(\BB, v_0))$. Then $\mu^B$ is (isomorphic in $\GrGp$ to) the restriction of the $[B]$-covering to its core subgraph of groups. 
    
    If $\mu^C\colon \CC\to \AA$ is any immersion of graphs of groups with $\CC$ core and connected, then there exists a lift $\mu\colon \CC\to \BB$ such that $\mu^B\circ\mu \approx \mu^C$, if and only if there exists a vertex $w_0$ of $\gr C$ such that $\mu^C_*(\pi_1(\CC, w_0))$ conjugates into $\mu^B_*(\pi_1(\BB, v_0))$ within $\pi_1(\AA)$.
\end{cor}

Corollaries~\ref{cor: pointed_core_lifting_morphism_1} and~\ref{cor: core_lifting_morphism_1} in turn justify the following (pointed and unpointed) corollaries of Theorem~\ref{thm: uniqueness of covers}.

\begin{cor}\label{cor: pi_1-functor}
Let $(\AA,u_0)$ be a pointed graph of groups and $A = \pi_1(\AA,u_0)$. The map
    $$\left[\mu\colon(\BB, v_0)\to(\AA, u_0)\right]_\sim \enspace \longmapsto \enspace \mu_*(\pi_1(\BB,v_0))$$
    establishes a bijection from the set of immersions $\mu\colon(\BB, v_0)\to(\AA, u_0)$ where $(\BB, v_0)$ is core, taken up to isomorphism in $\GrGp^*$, to the set of subgroups of $\pi_1(\AA,u_0)$. The inverse map is given by sending the subgroup $B\leqslant \pi_1(\AA, u_0)$ to the restriction of the pointed $B$-covering to its pointed core subgraph of groups.
\end{cor}

Before stating the unpointed version of Corollary~\ref{cor: pi_1-functor}, we define locally elliptic subgroups.

\begin{defn}\label{def: locally elliptic}
Let $\AA$ be a graph of groups and let $u \in V(\gr A)$. A subgroup $B$ of $\pi_1(\AA,u)$ is said to be \emph{locally elliptic} if each of its elements is conjugate, in the groupoid $\pi_1(\AA)$, to a vertex group element. In other words, $B$ is locally elliptic if, for each $b\in B$, there exists an $\AA$-path $p$ starting at $u$, such that $p\inv\, b\, p \in A_{t(p)}$.
\end{defn}

Equivalently, locally elliptic subgroups are precisely the (conjugacy classes of) subgroups whose associated cover has empty core. In other words, they are the subgroups in which each element acts elliptically on the Bass--Serre tree, see Appendix \ref{appendix}. We remark that finitely generated locally elliptic subgroups are elliptic (in the sense that they fix a vertex in the Bass--Serre tree) by \cite[Corollary 7.3]{bas93}, justifying the name.

\begin{cor}
\label{cor: core_lifting_morphism_2}
    Let $\AA$ be a connected graph of groups. Letting $\mathcal L$ be the set of conjugacy classes $[B]$ of non-locally elliptic subgroups $B$ of the groupoid $\pi_1(\AA)$, the map
    $$\left[\mu\colon \BB \to \AA\right]_\approx \enspace \longmapsto \enspace \Big\{[\mu_*(\pi_1(\BB,v))] \mid v\in V(\gr B)\Big\}$$
    establishes a bijection from the set of immersions $\mu\colon \BB \to \AA$ where $\BB$ is core, taken up to isomorphism in $\GrGp$, to $\mathcal P(\mathcal L)$, the power set of $\mathcal L$.
\end{cor}

\subsection{Double cosets and components of the $\AA$-product}\label{sec: double cosets}

In this section, we describe the connected components of the $\AA$-product of two immersions into a graph of groups $\AA$, in terms of certain double cosets of $\pi_1(\AA, u_0)$ (for some vertex $u_0$ of $\gr A$), Theorems~\ref{thm: double_cosets} and~\ref{thm: localy_elliptic_double_cosets} below.

Throughout this section, $\mu^B\colon \BB\to \AA$ and $\mu^C\colon \CC\to \AA$ are immersions of graphs of groups and we let $\DD = \BB\wtimes_{\AA}\CC$ be their $\AA$-product, with projection morphisms $\rho^B$ and $\rho^C$. We fix vertices $u_0\in V(\gr A)$, $v_0\in V(\gr B)$ and $w_0\in V(\gr C)$ such that $[v_0] = [w_0] = u_0$, and we let $A = \pi_1(\AA, u_0)$, $B = \mu^B_*(\pi_1(\BB, v_0))$, $C = \mu^C_*(\pi_1(\CC, w_0))$. Recall that since $\mu^B$ and $\mu^C$ are immersions, $\mu^B_*$ and $\mu^C_*$ are injective on $\pi_1(\BB, v_0)$ and $\pi_1(\CC, w_0)$, respectively.

\def\calC{\mathcal{C}}

We define a map $\calC$ from $V(\DD)$ to $\dblcoset BAC$ as follows. For each vertex $v\in V(\gr B)$, fix a $\BB$-path $p_v$ from $v_0$ to $v$. Similarly, for each $w\in V(\gr C)$, fix a $\CC$-path $q_w$ from $w_0$ to $w$. If $[v] = [w]$, the map $\calC_{v,w}$ is given by
\begin{align*}
\calC_{v,w}\colon\enspace V_{v,w}(\gr D) &\enspace\longrightarrow\enspace \dblcoset BAC \\
\mu^B(B_v)\, a\, \mu^C(C_w) &\enspace\longmapsto\enspace B\, (\mu^B(p_v)\, a\, \mu^C(q_w)^{-1})\, C.
\end{align*}
Since $V(\gr D)$ is the disjoint union of the $V_{v,w}(\gr D)$, we define $\calC(x) = \calC_{v,w}(x)$ if $x\in V_{v,w}(\gr D)$.

We first observe that $\calC$ does not depend on the choice of the $\BB$-paths $p_v$. Indeed, if $p_v$ and $p_v'$ are $\BB$-paths from $v_0$ to $v$, then $p_v'\, p_v^{-1}\in \pi_1(\BB, v_0) = B$ and so 
\[
B\, (\mu^B(p_v')\, a\, \mu^C(q_w)^{-1})\, C = B\, (\mu^B(p_v'\, p_v^{-1}))\, (\mu^B(p_v)\, a\, \mu^C(q_w)^{-1})\, C = B\, (\mu^B(p_v)\, a\, \mu^C(q_w^{-1}))\, C.
\]
Thus, different choices of $\BB$-paths lead to the same $(B,C)$-double coset. The same argument is valid for the choice of $\CC$-paths.

We now show that $\calC$ induces an injective map on the set $\pi_0(\DD)$ of connected components of $\DD$, which is surjective if at least one of $\mu^B$ and $\mu^C$ is a covering.

\begin{lem}
\label{lem: same_double_coset}
    Two vertices $x_1 = \mu^B_{v_1}(B_{v_1})\,d_{x_1}\,\mu^C_{w_1}(C_{w_1})$ and $x_2 = \mu^B_{v_2}(B_{v_2})\, d_{x_2}\, \mu^C_{w_2}(C_{w_2})$ in $V(\BB\wtimes_{\AA}\CC)$ lie in the same connected component if and only if $\calC(x_1) =\calC(x_2)$. Hence, the map $\mathcal{C}$ descends to an injective map:
    \[
    \calC_0\colon \pi_0(\DD) \to B\backslash A/C.
    \]
\end{lem}

\begin{proof}
Suppose first that $x_1, x_2$ lie in the same component and let $d$ be a reduced $\DD$-path connecting $x_1$ with $x_2$. By Lemma~\ref{lem: composition}, we have that 
\[
d_{x_1} =_{\AA} (\mu^B\circ\rho^B(d))\, d_{x_2}\, (\mu^C\circ\rho^C(d))^{-1}.
\]
Let $p_{v_1}$ and $p_{v_2}$ be reduced $\BB$-paths from $v_0$ to $v_1$ and $v_2$, respectively, and let $q_{w_1}$ and $q_{w_2}$ be reduced $\CC$-paths from $w_0$ to $w_1$ and $w_2$, respectively. Then
\begin{align*}
\mu^B(p_{v_1})\, d_{x_1}\, \mu^C(q_{w_1})^{-1} &=_{\AA} \mu^B(p_{v_1})\, (\mu^B\circ\rho^B(d))\, d_{x_2}\, (\mu^C\circ\rho^C(d))^{-1}\, \mu^C(q_{w_1})^{-1}\\
&=_{\AA} \mu^B(p_{v_1}\,\rho^B(d))\, d_{x_2}\, \mu^C(\rho^C(d)^{-1}q_{w_1}^{-1}).
\end{align*}
The $(B,C)$-double coset determined by this element is $\calC(x_1)$ by definition and, since $p_{v_1}\,\rho^B(d)$ is a $\BB$-path from $v_0$ to $v_2$ and $q_{w_1}\,\rho^C(d)$ is a $\CC$-path from $w_0$ to $w_2$, it is also equal to $\calC(x_2)$.

Now suppose that $\calC(x_1) = \calC(x_2)$. Then there exists a reduced $\BB$-circuit $p_B$ at $v_0$ and a reduced $\CC$-circuit $q_C$ at $w_0$ such that
\[
\mu^B(p_{v_1})\, d_{x_1}\, \mu^C(q_{w_1})^{-1} =_{\AA} \mu^B(p_B\,p_{v_2})\, d_{x_2}\, \mu^C(q_{w_2}^{-1}\, q_C^{-1}).
\]
Rearranging, we obtain that 
\[
\mu^B(p_{v_1}^{-1}\, p_{v_2}\, p_B) =_{\AA} d_{x_1}\, \mu^C(q_{w_1}^{-1}\, q_{w_2}\, q_C)\, d_{x_2}^{-1}
\]
and hence, by Corollary~\ref{cor: lifting}, there exists a reduced $\DD$-path connecting $x_1$ with $x_2$.
\end{proof}

The map $\calC_0$ from Lemma \ref{lem: same_double_coset} may not always be surjective. However, when $\mu^B$ is a covering, it always is.

\begin{lem}
\label{lem: double_cosets_surjection}
If $\mu^B\colon \BB\to \AA$ is a covering, then the map $\calC_0$ is a bijection.
\end{lem}

\begin{proof}
Let $g \in A$, represented by a reduced $\AA$-circuit $p$ at $u_0$. By Lemma \ref{lem: pointed segment lift}, there exists a reduced $\BB$-path $b$ starting at $v_0$ and an element $a\in A_{u_0}$ such that $\mu^B(b)\,a \sim_\AA p$. In particular, $\mu^B(b)$ is an $\AA$-circuit at $u_0$, and hence $B\,g\,C = B\,(\mu^B(b)\,a)\,C$. Let $v$ be the end vertex of $b$. Then $[v] = u_0$ and $x = \mu^B(B_v)\,a\,\mu^C(C_{w_0})$ is a $(v,w_0)$-vertex of $\DD$, and we have $\calC(x) = B\,g\,C$. Thus $\calC$ is surjective as claimed.
\end{proof}

This directly yields the following result, when one of the two morphisms is a covering.

\begin{thm}
\label{thm: double_cosets}
    Let $\mu^B\colon \BB\to \AA$ be a covering and let $\mu^C\colon \CC\to \AA$ be an immersion of graphs of groups. If $\DD\subset \BB\wtimes_{\AA}\CC$ is the union of the components with non-trivial fundamental group, then the map $\calC_0$ from Lemma \ref{lem: same_double_coset} induces a bijection:
    \[
    \pi_0(\DD)  \enspace\longrightarrow\enspace \Big\{B\,g\,C \mid B^g\cap C\neq 1\Big\}.
    \]
\end{thm}

\begin{proof}
    The map $\calC_0$ is a bijection on $\pi_0(\BB\wtimes_{\AA}\CC)$ by Lemma \ref{lem: double_cosets_surjection}. Corollary~\ref{cor: intersection of subgroups}~\eqref{item: intersection_component} now implies the result.
\end{proof}

The analogous statement in the case of general immersions is a little more complex.

\begin{thm}
\label{thm: localy_elliptic_double_cosets}
    Let $\mu^B\colon \BB \to \AA$ and $\mu^C\colon \CC\to \AA$ be immersions of graphs of groups. If $\DD = \BB\wtimes_{\AA}\CC$, then the map $\calC_0$ from Lemma \ref{lem: same_double_coset} induces a bijection:
    \[
    \pi_0(\core(\DD)) \enspace\longrightarrow\enspace \big\{B\,g\,C \mid B^g\cap C \text{ is not locally elliptic}\big\}.
    \]
\end{thm}

\begin{proof}    
Let $x\in V(\core(\DD))$. By definition,  $\calC(x) = B\,(\mu^B(p_v)\,d_x\,\mu^C(q_w)\inv)\,C$. We show that the intersection $B^{\mu^B(p_v)\,d_x\,\mu^C(q_w)\inv} \cap C$ is not locally elliptic. Indeed, there exists a cyclically reduced $\DD$-circuit $r$ at $x$ with positive length. Since immersions send cyclically reduced elements to cyclically reduced elements, this implies that $\rho^C(r)$ is a cyclically reduced $\CC$-circuit at $w$, and $\rho^C(r)^{q_w\inv} \in C$ has a cyclic reduction of positive length (that is, it is not conjugated to a vertex group element). By Corollary~\ref{cor: intersection of subgroups}~\eqref{item: intersection_component}, we have
$$\mu^C_*(\rho^C_*(r)^{p_w\inv}) \enspace=_{\AA}\enspace \mu^C_*(\rho^C_*(r))^{\mu^C(q_w)\inv} \enspace\in\enspace B^{\mu^B(p_v)\,d_x\,\mu^C(q_w)\inv} \cap C.$$
Thus $B^{\mu^B(p_v)\,d_x\,\mu^C(q_w)\inv} \cap C$ is not locally elliptic and the codomain of $\calC_0$ is as claimed.

Now we show that the restriction of $\calC$ to $V(\core(\DD))$ surjects the set of double cosets $BgC$ such that $B^g\cap C$ is not locally elliptic. Since $\calC$ descends to an injection on $\pi_0(\DD)$ by Lemma \ref{lem: same_double_coset}, this will imply that the required map is a bijection.

Consider a double coset $B\,g\,C$ such that $B^g\cap C$ is not locally elliptic. Since $B^g \cap C$ contains a non-locally elliptic element, there exists a $\BB$-circuit $p_B$ at $v_0$ and a $\CC$-circuit $q_C$ at $w_0$ such that $\mu^B(p_B)^g =_{\AA} \mu^C(q_C)$ and $\mu^C(q_C)$ has a cyclic reduction with positive length. There exists a $\BB$-path $r'_B$ and a cyclically reduced $\BB$-circuit 
\[
p_B' = (1, f_1, b_1, \ldots, f_n, b_n)
\]
so that $p_B \sim_{\BB} r_B'^{-1}\,p_B'\,r_B'$. Similarly, there exists a $\CC$-path $r'_C$ and a cyclically reduced $\CC$-circuit 
\[
q_C' = (1, g_1, c_1, \ldots, g_m, c_m)
\]
so that $q_C\sim_{\CC}r_C'^{-1}\, q_C'\, r'_C$. Now we have
\[
\mu^B(p_B')^{\mu^B(r'_B)g\mu^C(r'_C)^{-1}} =_{\AA} \mu^C(q_C').
\]
Applying Theorem \ref{Collins'}, we see that one of the following holds:
\begin{enumerate}
    \item There is a $\BB$-suffix $r_B$ of $(p_B')^{\ell}$ for some $\ell\geqslant 1$ and an element $a\in \omega_{[g_m]}(A_{[g_m]})$ such that
    \[
    \mu^B(r'_B)\,g\, \mu^C(r'_C)^{-1} =_{\AA} \mu^B(r_B)^{-1}\,(f_i)_{\alpha} a\, (g_1)^{-1}_{\alpha}
    \]
    where here $f_i$ is the first edge in $r_B$. Hence,
    \[
    B\, g\, C = B(f_i)_{\alpha}\, a\, (g_1)_{\alpha}^{-1}C = \calC(\mu^B(B_v)\, (f_i)_{\alpha}\, a\, (g_1)_{\alpha}^{-1}\, \mu^C(C_w))
    \]
    where $v = o(r_B)$ and $w = o(r'_C)$. Since (the reduction of) $(p_B')^{r_B^{-1}}$ and $p'_C$ are cyclically reduced, by Corollary \ref{cor: lifting} there is a cyclically reduced circuit at 
    $x = \mu^B(B_v)(f_i)_{\alpha}a(g_1)_{\alpha}^{-1}\mu^C(C_w)\in V(\gr{D}))$ and so $x\in V(\core(\DD))$ with $\calC(x) = BgC$ as required.
    \item There is a $\BB$-prefix $r_B$ of $(p_B')^{\ell}$ for some $\ell\geqslant 1$ and an element $a\in \alpha_{[g_1]}(A_{[g_1]})$ such that
    \[
    \mu^B(r'_B)\, g\, \mu^C(r'_C)^{-1} =_{\AA} \mu^B(r_B)\,(f_i)_{\omega}\, a\, (g_1)^{-1}_{\alpha}
    \]
    where here $f_i$ is the last edge in $r_B$. Hence,
    \[
    B\, g\, C = B(f_i)_{\omega}\, a\, (g_1)_{\alpha}^{-1}C = \calC(\mu^B(B_v)\, (f_i)_{\omega}\, a\, (g_1)_{\alpha}^{-1}\, \mu^C(C_w))
    \]
    where $v = t(r_B)$ and $w = o(r'_C)$. Since (the reduction of) $(p_B')^{r_B}$ and $p'_C$ are cyclically reduced, by Corollary \ref{cor: lifting} there is a cyclically reduced circuit at 
    $x = \mu^B(B_v)(f_i)_{\omega}\, a\, (g_1)_{\alpha}^{-1}\mu^C(C_w)\in V(\gr{D}))$ and so $x\in V(\core(\DD))$ with $\calC(x) = BgC$ as required.
\end{enumerate}
This completes the proof.
\end{proof}

We close this section with the following proposition, which should be thought of as an explanation of the mismatch between Theorems~\ref{thm: double_cosets} and~\ref{thm: localy_elliptic_double_cosets}.

\begin{prop}
\label{prop: conjugate_into_edge_group}
   Suppose that $\BB$ is core. If $B\,g\,C\notin \calC(V(\BB\wtimes_{\AA}\CC))$, then $B^g\cap C$ conjugates into an edge group of $\AA$.
\end{prop}

\begin{proof}
Let $\BB'$ be the $[B]$-cover. By Corollary \ref{cor: core_lifting_morphism_1}, we can identify $\BB$ with $\core(\BB')$, that is, with a subgraph of groups of $\BB'$, and hence also $\BB\wtimes_{\AA}\CC$ with a subgraph of groups of $\BB'\wtimes_{\AA}\CC$. See Corollary \ref{cor: core_lifting_morphism_1}.

Let $B\,g\,C$ be a double coset that is not in $\calC(V(\BB\wtimes_{\AA}\CC))$. By Lemma~\ref{lem: double_cosets_surjection}, $B\,g\,C = \calC(x)$ for some vertex $x$ of $\BB'\wtimes_{\AA}\CC$, not in $\BB\wtimes_{\AA}\CC$. Say that $x = \mu^B_{v'}(B_{v'})\, a\, \mu^C_w(C_w)$ is a $(v',w)$-vertex. Then (with the notation of Lemma~\ref{lem: form_of_core}), there exists a vertex $v\in V(\gr{B})$ such that $v'$ is a vertex of $\gr{T}_v - \{v\}$.

By Lemma~\ref{lem: same_double_coset}, the connected component $\DD_0$ of $\BB'\wtimes_{\AA}\CC$ containing the vertex $x$ does not meet $\BB\wtimes_\AA\CC$, and hence $\rho^B(\DD_0)$ is contained in $\mathbb{T}_v - v$. Since $\gr{D}_0$ is connected, $\rho^B(\DD_0)$ is a (sub)tree of groups, and there exists a unique edge $f$ of $\mathbb T_v - \rho^B(\DD_0)$ with its origin in $\rho^B(\DD_0)$ and pointing towards $v$. Let $x\in V(\gr{D}_0)$ such that $\rho^B(x) = o(f)$. By Lemma~\ref{lem: form_of_core}, $\alpha_f$ is an isomorphism. Furthermore, for each edge $f'$ in $\rho^B(\DD_0)$ pointing towards $o(f)$, we also have that $\alpha_{f'}$ is an isomorphism. As a result, since $\rho^B$ is an immersion, every reduced $\DD_0$ circuit at $x$ must actually have length $0$. Hence, $\rho^B_*(\pi_1(\DD_0, x)) \leqslant \alpha_{f}(B'_{f})$. This implies that $(\mu^B\circ\rho^B)_*(\pi_1(\DD_0, x))$ conjugates into an edge group of $\AA$ and so we are done.
\end{proof}

\subsection{Acylindrical graphs of groups}\label{sec: acylindrical gog}

We now use the previous results to investigate the case of acylindrical graphs of groups, and to show that pullbacks of morphisms on such graphs of groups always exist.

The reader is invited to consult Appendix \ref{appendix} for the correspondence between $G$-trees and graphs of groups. Recall that a $G$-tree $\gr{T}$ is \emph{$k$-acylindrical} if every segment of length $k$ in $\gr{T}$ has trivial stabiliser. It is \emph{acylindrical} if it is $k$-acylindrical for some $k$. For graphs of groups, the equivalent definition is as follows. A graph of groups $\AA$ is \emph{$k$-acylindrical} if for every reduced $\AA$-path $p$ of length $k$ and every non-trivial element $a\in A_{t(p)}$, the $\AA$-path $p\,a\,p^{-1}$ does not reduce to an element of $A_{o(p)}$. Again, a graph of groups is \emph{acylindrical} if it is $k$-acylindrical for some~$k$. We record the fact that these two definitions are equivalent in a lemma.

\begin{lem}
\label{lem: acylindrical_characterisation}
    Let $\AA$ be a graph of groups and let $k\geqslant 0$ be an integer. The following are equivalent:
    \begin{enumerate}[(1)]
        \item The Bass--Serre tree $\gr{T}$ for $\AA$ is $k$-acylindrical.
        \item $\AA$ is $k$-acylindrical.
    \end{enumerate}
    Moreover, if $\mu\colon \BB\to \AA$ is an immersion and $\AA$ is $k$-acylindrical, then so is $\BB$.
\end{lem}

\begin{proof}
    We prove the contrapositive of the equivalence statement. Let $S\subset\gr{T}$ be a segment. Let $u_0\in V(\gr{A})$ be the image of a vertex at one end of $S$ and let $u$ be the image of the vertex at the other end of $S$ under the map $\gr{T}\to \gr{A}$ given by quotienting $\gr{T}$ by the action of $\pi_1(\AA, u_0)$. By construction of $\gr{T}$ (see Appendix \ref{appendix}), there is an element $p\in \pi^{\AA}[u_0, u]$ such that the segment $S$ connects the vertex $A_{u_0}$  with the vertex $p\cdot A_{u}$. Since the stabiliser of the vertex $A_{u_0}$ is precisely $A_{u_0}$ and the stabiliser of the vertex $p\cdot A_u$ is precisely $p\cdot A_u\cdot p^{-1}$, anything stabilising $S$ must lie in $A_{u_0}\cap p\cdot A_u\cdot p^{-1}$. Thus, $S$ has non-trivial stabiliser if and only if there is some element $a\in A_u$ such that $pap^{-1}$ reduces to an element of $A_{u_0}$. This implies that $\gr{T}$ is not $k$-acylindrical if and only if $\AA$ is not $k$-acylindrical.

    Since immersions of graphs of groups send reduced paths to reduced paths of the same length (Proposition~\ref{folded}), if $\AA$ is $k$-acylindrical, so is $\BB$.
\end{proof}

We note a corollary of Corollary~\ref{thm: pullbacks_exist_sometimes1}.

\begin{cor} \label{cor: acylindrical pullback}
    If $\AA$ is an acylindrical graph of groups, then, within the category of core graphs of groups with immersions, pullbacks of morphisms to $\AA$ exist.
\end{cor}

\begin{proof}
    Let $\mu^B \colon \BB \to \AA$ and $\mu^C\colon \CC\to \AA$ be immersions of core graphs of groups with $\AA$ acylindrical. Let $c$ be a cyclically reduced $\CC$-circuit at a vertex $w$ and let $a\in A_u$ where $u = \mu^C(c)$. If $[a, \mu^C(c)] = 1$, then $[a, \mu^C(c^n)] = 1$ for all $n\geqslant 1$. In particular, since $c$ is cyclically reduced and $\mu^C$ is an immersion, $\mu^C(c^n)$ has length at least $n$. Thus, for $n$ greater than the acylindricity constant, we see that $\mu^C(c^{-n})a\mu^C(c^n) =_{\AA} a$ and so $a = 1$. Now we may apply Corollary \ref{thm: pullbacks_exist_sometimes1} to conclude that the pullback of $\mu^B$ and $\mu^C$ exists in the category of core graphs of groups with immersions. Since $\mu^B$ and $\mu^C$ are immersions and $\rho^B$ and $\rho^C$ are immersions by Lemma \ref{lem: rhos are folded 2}, we see that $\BB$, $\CC$ and the pullback are also acylindrical by Lemma \ref{lem: acylindrical_characterisation}. Hence, since $\mu^B$ and $\mu^C$ were arbitrary, pullbacks exist in the category of core acylindrical graphs of groups with immersions.
\end{proof}

Under additional hypotheses, we obtain a refinement of Theorem~\ref{thm: double_cosets} for acylindrical graphs of groups.
We first need a definition and a technical proposition. A group $A$ is said to have the \emph{strong finitely generated intersection property} (henceforth \emph{\sfgip}) \emph{relative to a subgroup $B\leqslant A$} if for any finitely generated subgroup $C\leqslant A$, we have that $B\cap C$ is finitely generated and $B^g\cap C = 1$ for all but finitely many double cosets $B\, g\, C\subset A$.

\begin{prop}
\label{prop: acylindrical}
    Let $\AA$ be a graph of groups in which $A_{o(e)}$ has the \sfgip\ relative to $\alpha_e(A_e)$ for each edge $e\in E(\gr{A})$. Let $B\leqslant \pi_1(\AA, u_0)$ be a subgroup and let $\mu\colon \BB\to \AA$ be the $[B]$-covering. If $\BB$ is acylindrical, $\core(\BB)$ has finite underlying graph and each vertex group in $\core(\BB)$ is finitely generated, then all but finitely many vertex and edge groups of $\BB$ are trivial and all vertex and edge groups of $\BB$ are finitely generated.
\end{prop}

\begin{proof}
    Since $\core(\BB)$ has finitely generated vertex groups and since each vertex group in $\AA$ has the \sfgip\ relative to its adjacent edge groups, it follows that each edge group in $\BB_0 = \core(\BB)$ is also finitely generated. By Lemma \ref{lem: form_of_core}, for each vertex $v\in V(\gr{B}_0)$, there is a tree of groups $\mathbb{T}_v\subset \BB$ such that $\gr{T}_v\cap\gr{B}_0 = v$, $\gr{T}_v\cap \gr{T}_w = \emptyset$ for all $v\neq w\in V(\gr{B}_0)$ and such that
    \[
    \gr{B} = \gr{B}_0 \cup \bigcup_{v\in V(\gr{B}_0)}\gr{T}_v.
    \]
    Moreover, for every $v \in  V(\gr{B}_0)$ and every edge $e$ in $\gr{T}_v$ oriented towards $v$, the map $\alpha_e\colon B_e\to B_{o(e)}$ is an isomorphism.

    For each integer $n\geqslant 0$, denote by $\BB_n\subset \BB$ the subgraph of groups containing the core $\BB_0$ and all reduced $\BB$-paths of length at most $n$ leading out of a vertex in the core. We will show by induction on $n$ that each vertex and edge group in $\BB_n$ is finitely generated and that $\BB_n$ has finitely many vertices and edges with non-trivial associated group. We have already shown this for the base case $n = 0$. Now we assume that the claim holds true for some $n$ and we show that it also holds for $n+1$. By our \sfgip\ assumption, for every vertex $v\in V(\gr{B}_n)$ and for every edge $e\in E(\gr{A})$ such that $o(e) = \mu(v)$, there are finitely many double cosets $\mu_v(B_v)\, a\, \alpha_e(A_e)$ such that $\mu_v(B_v)^a\cap \alpha_e(A_e) \neq 1$. Moreover, each such intersection is finitely generated. Hence, by definition of coverings and the fact that $\alpha_e\circ\mu_f(B_f) = \mu_v(B_v)^{f_{\alpha}}\cap \alpha_e(A_e)$ for all edges $f\in E(\gr{B})$ with $o(f) = v$ and $\mu(f) = e$, it follows that each vertex in $\gr{B}_n$ has finitely many outgoing edges in $\gr{B}_{n+1}$ with non-trivial edge group. Moreover, each edge group is finitely generated. Since every vertex in $\gr{B}_{n+1} - \gr{B}_n$ has associated group isomorphic to an adjacent edge group, it follows that there are also finitely many vertices in $\gr{B}_{n+1}$ with non-trivial associated group and each such group is finitely generated, thus completing the proof of the claim.
    
    Finally, we use the acylindricity of $\BB$ to finish the proof. Suppose that $k$ is the acylindricity constant, let $v\in V(\gr{B}_0)$ and let $w\in V(\gr{T}_v)$ be a vertex at distance $k$ from $v$. Let $p$ be a $\mathbb{T}_v$-path from $v$ to $w$ and let $b\in B_w$. Since for each $e\in E(\gr{T}_v)$ pointing towards $v$ the edge map $\alpha_e$ is an isomorphism, it follows that $p\, b\, p^{-1}$ is a $\BB$-path which reduces to an element in $B_v$. Thus, by Lemma~\ref{lem: acylindrical_characterisation} we have that $b = 1$. Hence, $B_w = 1$ and so every vertex at distance at least $k$ from $\gr{B}_0$ has trivial associated group. In particular, this implies, combined with all the above, that all but finitely many vertex and edge groups of $\BB$ are trivial.
\end{proof}

\begin{thm}
\label{thm: acylindrical}
    Let $\AA$ be a graph of groups in which $A_{o(e)}$ has the \sfgip\ relative to $\alpha_e(A_e)$ for each edge $e\in E(\gr{A})$. Let $B\leqslant \pi_1(\AA, u_0)$ be a subgroup and let $\mu\colon \BB\to \AA$ be the $[B]$-covering. If $\BB$ is acylindrical, $\core(\BB)$ has finite underlying graph and each vertex group in $\core(\BB)$ is finitely generated, then there is a finite subgraph of groups $\BB'\subseteq\BB$ with finitely generated vertex and edge groups containing $\core(\BB)$ such that the following holds. 
    
    If $\mu^C\colon (\CC, w_0)\to (\AA, u_0)$ is any immersion of graphs of groups and if $\DD$ is the union of all the connected components of $\BB'\wtimes_{\AA}\CC$ with non-trivial fundamental group, then the map $\calC_0$ from Lemma \ref{lem: same_double_coset} induces a bijection between $\pi_0(\DD)$ and the set of double cosets $\left\{B\, g\, C \mid B^g\cap C \neq 1\right\}$, where $C = \mu^C_*(\pi_1(\CC, w_0))$.
\end{thm}

\begin{proof}
    By Proposition \ref{prop: acylindrical}, there exists a finite subgraph of groups $\BB'\subseteq\BB$ containing $\core(\BB)$ with finitely generated vertex and edge groups that contains all vertices with non-trivial associated group from $\BB$. Each component of $\BB\wtimes_{\AA}\CC$ which has non-trivial fundamental group either has a vertex with a non-trivial vertex group, or supports a cyclically reduced cycle in its underlying graph. In the first instance, the component must have a vertex of the form $\mu^B_v(B_v)\, a\, \mu^C_w(C_w)$ for some $v\in V(\BB')$ as all others have trivial vertex group by definition of the vertex groups of $\BB\wtimes_{\AA}\CC$. In the second instance, the cyclically reduced cycle in the underlying graph maps to cyclically reduced cycles in the underlying graphs of $\BB$ and $\CC$. Thus, it projects to cycles in $\core(\BB)$ and $\core(\CC)$. In particular, such a component of $\BB\wtimes_{\AA}\CC$ must contain a vertex of the form $\mu^B_v(B_v)\, a\, \mu^C_w(C_w)$ for some $v\in V(\core(\BB))$. Since $\core(\BB)\subseteq \BB'$, the components of $\DD$ with non-trivial fundamental group are in bijection with the components of $\BB\wtimes_{\AA}\CC$ with non-trivial fundamental group via the inclusion $\BB'\wtimes_{\AA}\CC\subset \BB\wtimes_{\AA}\CC$. Applying Theorem~\ref{thm: double_cosets} completes the proof.
\end{proof}



\renewcommand*{\bibfont}{\small}
\printbibliography


\appendix

\section{The $\GpT$ category and the $\Gp$-tree functor}\label{sec: pi1 functor}
\label{appendix}

References for the theory of groups acting on trees, known as Bass--Serre theory, include \cite{ser80,
bas93}. We shall always assume that actions do not invert edges. This is not a real restriction as we can always subdivide edges in a tree to obtain an action that does not invert edges.

If $G$ is a group, a $G$-tree is a tree $\gr{T}$ which comes with an action of $G$. If $\gr{T}$ is a $G$-tree and $\gr{S}$ is an $H$-tree, then a $\Gp$-tree morphism between $\gr{T}$ and $\gr{S}$ is a map $\psi\colon \gr{T}\to \gr{S}$ together with a group homomorphism $\phi\colon G\to H$ such that $\psi(g\cdot x) = \phi(g)\cdot \psi(x)$ for all $g\in G$ and $x\in \gr{T}$. A morphism of $\Gp$-trees is an isomorphism (respectively, monomorphism, epimorphism) if the associated homomorphism is an isomorphism (respectively, monomorphism, epimorphism) and the associated morphism of trees is bijective (respectively, injective, surjective).

A pointed $G$-tree $(\gr{T}, x_0)$ is a $G$-tree $\gr{T}$ together with a choice of vertex $x_0\in V(\gr{T})$. A morphism between pointed $\Gp$-trees $(\gr{T}, x_0) \to (\gr{S}, y_0)$ is a $\Gp$-tree morphism that sends $x_0$ to $y_0$. The category of pointed $\Gp$-trees $\GpT^*$ is the category whose objects are the pointed $\Gp$-trees and whose morphisms are pointed $\Gp$-tree morphisms. The fundamental theorem of Bass--Serre theory states that the categories of pointed $\Gp$-trees and pointed graphs of groups are the same. We will roughly explain this correspondence in the next section. The reader is directed to \cite[Section 4]{bas93} for more details.

Two morphisms of $\Gp$-trees $(\psi_1, \varphi_1)$, $(\psi_2, \varphi_2)$ are $\approx$-equivalent, written $(\psi_1, \phi_1)\approx(\psi_2, \varphi_2)$, if there is some $h\in H$ such that $\psi_1 = h^{-1}\cdot \psi_2$ and $\varphi_1 = \gamma_h\circ\varphi_2$. In other words, if there exists a $\Gp$-tree isomorphism induced by conjugation making the following diagram commute:
\[
\begin{tikzcd}
\mathsf{T} \arrow[r, "{(\psi_1, \varphi_1)}"] \arrow[rd, "{(\psi_2, \varphi_2)}"'] & \mathsf{S} \arrow[d, "\cong", dotted] \\
                                                                                  & \mathsf{S}                           
\end{tikzcd}
\]
The category of $\Gp$-trees $\GpT$ is the category of $\Gp$-trees with $\approx$-equivalence classes of $\Gp$-tree morphisms. We will see that there is a version of the fundamental theorem of Bass--Serre theory also for the unpointed category.

\subsection{General construction of $\Gp$-trees and morphisms}

\def\GG{\mathbb{G}}

Let $(\GG, x_0)$ be a pointed graph of groups and $G = \pi_1(\GG,x_0)$. We denote by $\pi^\GG[x_0, x]$ the $=_\GG$-equivalence classes of $\GG$-paths connecting $x_0$ with $x$, so that $G = \pi^\GG[x_0,x_0]$. If $E^+(\gr{G})\subset E(\gr{G})$ is an orientation, we let $\gr{T}$ be the graph with
\begin{align*}
    V(\gr{T}) &= \bigsqcup_{x\in V(\gr{G})}\pi^{\mathbb{G}}[x_0, x]/G_x\\
    E(\gr{T}) &= \bigsqcup_{e\in E^+(\gr{G})}\pi^{\mathbb{G}}[x_0, o(e)]/\alpha_e(G_e)
\end{align*}
where adjacency here is given by inclusion of the corresponding cosets. By \cite[Theorem 1.17]{bas93}, $\gr{T}$ is a tree known as the \emph{Bass--Serre tree associated with the pointed graph of groups $(\GG,x_0)$}. There is a natural left $G$-action on $\gr{T}$, making $\gr{T}$ a $G$-tree. The construction of the Bass--Serre tree comes with a natural choice of basepoint, the vertex $\tilde{x}_0 = G_{x_0}$, making $(\gr T, \tilde{x}_0)$ a pointed $G$-tree.

Now let $\mu\colon(\mathbb{G}, x_0)\to (\mathbb{H}, y_0)$ be a (orientation preserving) morphism of pointed graphs of groups and let $(\gr S, \tilde{y}_0)$ be the (pointed) Bass--Serre tree associated with the pointed graph of groups $(\mathbb{H}, y_0)$. Then there is an induced morphism of pointed $\Gp$-trees 
\begin{align*}
    \psi\colon (\gr{T}, \tilde{x}_0)&\to (\gr{S}, \tilde{y}_0) &\text{where $\tilde{x}_0 = G_{x_0}$ and $\tilde y_0 = H_{y_0}$},\\
    g\,G_x&\mapsto\mu_*(g)H_{\mu(x)} &\text{for each $x\in V(\gr{G})$, $g\in \pi^{\mathbb{G}}[x_0, x]$},\\
    g\, \alpha_e(G_e) &\mapsto\mu_*(g)\alpha_e(H_{\mu(e)})&\text{for each $e\in E^+(\gr{G})$, $g\in \pi^{\mathbb{G}}[x_0, o(e)]$}.
\end{align*}
This is outlined in more detail in \cite[Section 4]{bas93}.

\subsection{The pointed functor}

We want to show that the construction outlined above gives rise to a well-defined map between the category of pointed graphs of groups and the category of pointed $\Gp$-trees. This is subtle since within our definition of graphs of groups and morphisms of graphs of groups, there are several pieces of data involved which are not uniquely determined by the object or morphism in $\GrGp^*$. Thus, we need to check that any $\Gp$-tree and any $\Gp$-tree morphism constructed above are independent of any choices involved.

Let $\mu\colon(\mathbb{G}, x_0)\to(\mathbb{G}', x_0')$ be an isomorphism of pointed graphs of groups. Since $\mu$ induces an isomorphism on underlying graphs, an isomorphism on each vertex and edge group and an isomorphism $\pi^{\mathbb{G}}[x_0, x] \to \pi^{\mathbb{G}'}[x_0', \mu(x)]$ for every vertex $x\in V(\gr{G})$, it is clear from the construction of the associated $\Gp$-trees that the induced morphism of $\Gp$-trees is an isomorphism.

Let $\mu_1, \mu_2\colon (\mathbb{G}, x_0)\to (\mathbb{H}, y_0)$ be morphisms of pointed graphs of groups such that $\mu_1\sim \mu_2$. By Proposition~\ref{prop: change of data}, for each vertex $x\in V(\gr{G})$ there is an element $h_x\in H_{\mu(x)}$ (with $h_{\mu(x_0)} = 1$) such that for each element $g\in \pi^{\mathbb{G}}[x_0, x]$, we have $\mu_1(g) \sim_{\mathbb{H}} \mu_2(g)\, h_x$. Hence, since $\mu_1(g)\,H_{\mu(x)} = \mu_2(g)\, h_x\, H_{\mu(x)} = \mu_2(g)\, H_{\mu(x)}$, the morphisms on Bass--Serre trees induced by $\mu_1$ and $\mu_2$ are identical. Moreover, since $h_{\mu(x_0)} = 1$, we have that $(\mu_1)_* = (\mu_2)_*$ and so the induced $\Gp$-tree morphisms are also identical.

Now denote by $F^*\colon \GrGp^*\to \GpT^*$ the map between categories which we have just shown to be well-defined. In order to check that this is a covariant functor, we need to show that $F^*(\id_{\GrGp^*}) = \id_{\GpT^*}$ and that, given any two composable morphisms of graphs of groups $\mu^1\colon (\mathbb{G}_1, x_1)\to (\mathbb{G}_2, x_2)$, $\mu^2\colon (\mathbb{G}_2, x_2)\to (\mathbb{G}_3, x_3)$, we have $F^*(\mu^2)\circ F^*(\mu^1) = F^*(\mu^2\circ \mu^1)$.
The first property is immediate from the construction, whereas the second property follows from the fact that the composition
\[
\pi^{\mathbb{G}_1}[x_1, x] \xrightarrow{\mu^1_*}\pi^{\mathbb{G}_2}[x_2, \mu^1(x)]\xrightarrow{\mu^2_*}\pi^{\mathbb{G}_3}[x_3, \mu^2(\mu^1(x))]
\]
coincides with $\pi^{\mathbb{G}_1}[x_1, x] \xrightarrow{(\mu^2\circ\mu^1)_*}\pi^{\mathbb{G}_3}[x_3, \mu^2(\mu^1(x))]$.
Thus, $F^*$ is a functor. The following theorem follows from the above and results of Bass \cite[Corollaries 4.5 \& 4.6]{bas93}.

\begin{thm}
\label{thm: pointed_functor}
    The functor $F^*\colon \GrGp^*\to \GpT^*$ is bijective and hence is an isomorphism of categories. Moreover, a pointed morphism of graphs of groups $\mu$ is an immersion precisely when $F^*(\mu)$ is a monomorphism of pointed $\Gp$-trees.
\end{thm}

If $(\gr{T}, \tilde{x}_0)$ is a pointed $G$-tree, then we call $(\mathcal{G}, x_0) = (F^*)^{-1}(\gr{T}, \tilde{x}_0)$ the \emph{pointed quotient graph of groups}.

\subsection{The unpointed functor}

Now let $\mathbb{G}$ be a connected (unpointed) graph of groups. Let $x_0, x_1\in V(\gr{G})$ be any two vertices and consider the $\pi_1(\mathbb{G}, x_i)$-tree $(\gr{T}_i, \tilde{x}_0) = F^*(\mathbb{G}, x_i)$ for $i=0, 1$. Let $h\in \pi^{\mathbb{G}}[x_0, x_1]$ be any element. The isomorphism $\varphi\colon\pi_1(\mathbb{G}, x_0) \to \pi_1(\mathbb{G}, x_1)$ given by $\varphi(g) = h^{-1}gh$ induces a $\varphi$-equivariant isomorphism $\psi\colon \gr{T}_0\to \gr{T}_1$. Thus, as before, we may associate to $\mathbb{G}$ a $G$-tree $\gr{T}$, where $G\cong \pi_1(\mathbb{G}, x_0)$, which does not depend on the choice of basepoint. Moreover, as before, we may associate to any morphism of graphs of groups $\mathbb{G}\to \mathbb{H}$ a morphism of $\Gp$-trees and use Proposition~\ref{prop: change of data} to show that this morphism is independent of any choices made in its construction, sending $\approx$-equivalent morphisms of graphs of groups to $\approx$-equivalent morphisms of $\Gp$-trees. We thus have a commutative diagram
\[
\begin{tikzcd}
\GrGp^* \arrow[r, "F^*"] \arrow[d] & \GpT^* \arrow[d] \\
\ConGrGp \arrow[r, "F"]               & \GpT            
\end{tikzcd}
\]
where $\ConGrGp$ is the subcategory of $\GrGp$ containing only connected objects, where the vertical maps are the forgetful functors which forget the basepoint and where $F$ is the map described above. We just need to check that $F$ is a functor. The fact that $F(\id_{\GrGp}) = \id_{\GpT}$ is immediate from the construction. The fact that $F(\mu^2)\circ F(\mu^1) = F(\mu^2\circ\mu^1)$ for any pair of morphisms of graphs of groups $\mu^1, \mu^2\colon \mathbb{G}\to \mathbb{H}$ follows from Corollary \ref{cor: composition} and a diagram chase. Thus, $F$ is a functor.

The following unpointed version of Theorem \ref{thm: pointed_functor} also follows from the above and results of Bass \cite[Corollaries 4.5 \& 4.6]{bas93}.

\begin{thm}
\label{thm: unpointed_functor}
    The functor $F\colon \ConGrGp\to \GpT$ is bijective. Moreover, a morphism of graphs of groups $\mu$ is an immersion precisely when $F(\mu)$ is a monomorphism of $\Gp$-trees.
\end{thm}

As before, if $\gr{T}$ is a $\Gp$-tree, then we call $F^{-1}(\gr{T})$ the \emph{quotient graph of groups}.

Before moving on, we remind the reader that in both categories $\GrGp$ and $\GpT$, there is no well-defined group morphism associated to a morphism, rather an equivalence class of group morphisms. However, any pair of morphisms $(\varphi_1, \varphi_2)$ lying in such an equivalence class is conjugate, meaning that there is some inner automorphism $\gamma_h$ of the target group such that $\varphi_1 = \gamma_h\circ \varphi_2$. This was part of the definition for morphisms in $\GpT$ and the corresponding statement for $\GrGp$ follows from Proposition~\ref{prop: change of data}.

\section{Collins' Theorem}\label{sec: collins}

Let $\AA$ be a graph of groups. Two $\AA$-circuits $p$ and $q$ in $\AA$ are said to be \emph{conjugate} if there exists an $\AA$-path $r$ from the initial vertex of $p$ to that of $q$ such that $q =_\AA r\inv pr$. In this section we characterise precisely when two $\AA$-circuits are conjugate. This result is a generalisation of a classical result, commonly known as Collins' Lemma or Collins' Theorem \cite[Theorem IV.2.5]{ls01}. The original statement of Collins' Theorem is for HNN extensions, whereas we address all fundamental groups of graphs of groups. The proof below is simply a translation of the proof in \cite{ls01} to the more general setting.

Here we use the following terminology: if $p = (p_0, e_1, p_1, \ldots, e_m, p_m)$ is an $\AA$-path, an \emph{$\AA$-prefix of $p$} (of length $j$) is an $\AA$-path of the form $(p_0, e_{1}, p_{2}, \ldots, e_j, p_j)$ where $0\leqslant j\leqslant m$; and an \emph{$\AA$-suffix of $p$} (of length $j$) is an $\AA$-path of the form $(1, e_{m-j+1}, p_{m-j+1}, \ldots, e_m, p_m)$ where $0\leqslant j\leqslant m$. The $\AA$-suffix of length $0$ is understood to be the $\AA$-path $(1)$. In this way, $p$ is equal to the product of its $\AA$-prefix of length $j$ and its $\AA$-suffix of length $m - j$.

\begin{thm}[Collins' Theorem]
\label{Collins'}
    Let $\AA$ be a graph of groups, let
    \begin{align*}
        p &= (1, e_1, p_1, \ldots, e_m, p_m)\\
        q &= (1, f_1, q_1, \ldots, f_n, q_n)
    \end{align*}
    be conjugate cyclically reduced $\AA$-circuits and let $r$ be a reduced $\AA$-path from the initial vertex of $p$ to that of $q$ such that $q =_{\AA} r^{-1}pr$. Then $n = m$ and one of the following holds:
    \begin{enumerate}[(1)]
        \item\label{eq: Collins 1} $n = m = 0$.
        \item\label{eq: Collins 2} There is an $\AA$-suffix $p'$ of $p^\ell$ for some $\ell\geqslant1$ and an element $a\in \omega_{f_m}(A_{f_m})$ such that
        \[
        r \sim_{\AA} p'^{-1}\,a
        \]
        \item\label{eq: Collins 3} There is an $\AA$-prefix $p'$ of $p^\ell$ for some $\ell\geqslant 1$ and an element $a\in \alpha_{f_1}(A_{f_1})$ such that
        \[
        r \sim_{\AA} p'\, a.
        \]
    \end{enumerate}
\end{thm}

\begin{proof}
We may assume, without loss of generality, that $m\geqslant n$. If $n = 0$, an easy induction on the length of $r$ shows that $m = 0$. Now assume that $m\geqslant n>0$. We let
    \begin{align*}
        r &= (r_0, g_1, r_1, \ldots, g_k, r_k)
    \end{align*}
and we proceed by induction on $k$. 

If $k = 0$, then $q = (r_0^{-1})\, p\, (r_0)$ and Proposition~\ref{prop: sim vs equiv} shows that $n = m$ and there exist $a\in \alpha_{f_1}(A_{f_1})$ such that $a = r$ and both \eqref{eq: Collins 2} and \eqref{eq: Collins 3} hold.

Let now $k \ge 1$ and suppose that the result holds when $r$ has length strictly less than $k$. Since $k\geqslant 1$ and $m\geqslant n$, the $\AA$-path $r^{-1}\, p\, r$ must admit some reduction. In other words, either $e_m = g_1\inv$ and $p_mr_0\in \omega_{e_m}(A_{e_m})$, or $e_1 = g_1$ and $r_0^{-1}\in \alpha_{e_1}(A_{e_1})$. If both are true, then $e_1 = e_m\inv$ and $p_m = (p_mr_0)\,(r^{-1}_0) \in \alpha_{e_1}(A_{e_1})$, contradicting the assumption that $p$ is cyclically reduced. We complete the proof if the first situation occurs, as the other case is entirely analogous. Consider the following $\AA$-paths
\begin{align*}
    p' &= (1, e_m, p_m, e_1, \ldots, e_{m-1}, p_{m-1})\\
    r' &= (\alpha_{e_m}(\omega_{e_m}^{-1}(p_mr_0))\,r_1, g_2, \ldots, r_k).
\end{align*}
Then $p'$ is cyclically reduced, it has the same length as $p$, $r'$ is reduced and we have
\begin{align*}
p &=_\AA (p_m^{-1}, e_m^{-1}, 1)\, p'\, (1, e_m, p_m)\\
r &\sim_\AA (p_m^{-1}, e_m^{-1}, 1)\, r',
\end{align*}
and hence, $q =_\AA r'^{-1}\, p'\, r'$. Since $r'$ has length $k-1$, we may use the inductive hypothesis. In particular, $n = m$ and one of the following holds:
\begin{enumerate}[(i)]
    \item There is an $\AA$-suffix $p''$ of $p'^\ell$ for some $\ell\geqslant 1$ and an element $a\in \omega_{f_m}(A_{f_m})$ such that
        \[
        r' \sim_{\AA} p''^{-1}\, a
        \]
        \item There is an $\AA$-prefix $p''$ of $p'^\ell$ for some $\ell\geqslant 1$ and an element $a\in \alpha_{f_1}(A_{f_1})$ such that
        \[
        r' \sim_{\AA} p''\, a.
        \]
\end{enumerate}
In the first case, since $(p_m^{-1}, e_m^{-1}, 1)p''^{-1}$ is the inverse of an $\AA$-suffix of $p^l$ and since $r\sim_{\AA} (p_m^{-1}, e_m^{-1}, 1)r'\sim (p_m^{-1}, e_m^{-1}, 1)p''^{-1}a$, it follows that $r$ satisfies the conclusion. In the second case, if $r'$ has length one or more, then $r\sim_{\AA}(p_m^{-1}, e_m^{-1}, 1)\,p''\,a\sim_\AA (p_m^{-1}, e_m^{-1}, 1, e_m, p_m)\,p'''\,a$ for some reduced $\AA$-path $p'''$. Since we assumed that $r$ was a reduced $\AA$-path, this is a contradiction and so $r'$ has length zero. Hence, $r' \sim_{\AA} a$ and so $r \sim_{\AA} (p_m^{-1}, e_m^{-1}, 1)\,a$ satisfies the conclusion.
\end{proof}

\bigskip
\subsection*{Acknowledgements}
The fourth author expresses gratitude for the generous hospitality received from Harish-Chandra Research Institute, India and Universidad del País Vasco, Spain. We thank Francesco Fournier-Facio for pointing us to the work of Paramantzoglou \cite{pa12}.

\section*{Declarations}

\textbf{Funding}
The first and fourth authors acknowledge support from the Spanish Agencia Estatal de Investigación through grant PID2021-126851NB-100 (AEI/FEDER, UE). The second author acknowledges support from grant 202450E223 (Impulso de líneas científicas estratégicas de ICMAT). The third author has been supported by the Spanish Government grant PID2020-117281GB-I00, partly by the European Regional Development Fund (ERDF), and the Basque Government, grant IT1483-22.\\

\noindent
\textbf{Competing interests}
The authors have no competing interests to declare.\\

\noindent
\textbf{Data availability}
Data sharing is not applicable to this article as no datasets were generated or analysed during the current study.


\end{document}